\documentclass[11pt]{article} 
\usepackage{graphicx} 

\usepackage{amsmath,amssymb,amsthm,mathrsfs,dsfont,bm}
\usepackage[a4paper, total={7in, 9.5in}]{geometry}

\usepackage{titlesec,hyperref}
\usepackage{cite}
\usepackage{color}
\usepackage{fancyhdr}
\usepackage{tikz}
\usepackage{float}
\usepackage[]{cleveref}
\usetikzlibrary{positioning} 
\usepackage{xcolor}
\usepackage{lipsum}
\usepackage{mathtools}

\usepackage[titletoc]{appendix}

\newtheorem{theo}{Theorem}[section]
\newtheorem{rema}[theo]{Remark}
\numberwithin{equation}{section}
\newtheorem{lemm}{Lemma}[section]
\newtheorem{prop}{Proposition}[section]

\newcommand\N{{\mathbb N}}
\newcommand\R{{\mathbb R}}
\newcommand\Z{{\mathbb Z}}
\newcommand\T{{\mathbb T}}
\newcommand\I{{\mathbb I}}

\newcommand\PP{{\mathbb P}}
\newcommand\D{{\mathscr D}}
\newcommand\FFF{{\mathscr F}}
\newcommand\F{{\mathcal F}}
\newcommand\E{{\mathcal E}}
\newcommand\LL{{\mathcal L}}
\newcommand\GG{{\mathcal G}}
\newcommand\HH{{\mathcal H}}

\newcommand\J{{\mathcal J}}

\newcommand\UU{{\mathcal U}}

\newcommand\x{\bm{x}}

\newcommand\uu{\bm{u}}
\newcommand\vv{\bm{v}}

\newcommand\RR{\mathcal{R}}

\newcommand\V{\bm{V}}
\newcommand\FF{\bm{F}}

\let\pa=\partial
\let\na=\nabla
\let\al=\alpha

\let\f=\frac

\let\om=\omega
\let\ep=\epsilon
\let\tri=\triangleq
\def\dive{\mathop{\rm div}\nolimits}
\def\curl{\mathop{\rm curl}\nolimits}
\def\exp{\mathop{\rm exp}\nolimits}

\allowdisplaybreaks[4]

\begin{document}
	\title{Nonlinear stability of 2-D Couette flow for the compressible Navier-Stokes equations at high Reynolds number}
	\author{Minling Li, Chao Wang, Zhifei Zhang}
	\author{Minling Li\thanks{limling@pku.edu.cn}
		\quad Chao Wang\thanks{wangchao@math.pku.edu.cn} 
		\quad Zhifei Zhang\thanks{zfzhang@math.pku.edu.cn}\\[10pt]
		\small {School of Mathematical Sciences, Peking University}\\
		\small {Beijing 100871, P. R. China}\\
		[5pt]}
	\date{}
	\maketitle

	\begin{abstract}
		In this paper, we  investigate the nonlinear stability of the Couette flow for the two-dimensional compressible Navier--Stokes equations at high Reynolds numbers($Re$) regime. It was proved that if the initial data $(\rho_{in},\uu_{in})$ satisfies $\|(\rho_{in},\uu_{in})-(1, y, 0)\|_{H^4(\T\times\R)}\leq \ep Re^{-1}$ for some small $\ep$ independent of $Re$, then the corresponding solution exists globally and remains close to the Couette flow for all time.  Formal asymptotics indicate that this stability threshold is sharp within the class of Sobolev perturbations. The proof relies on the Fourier-multiplier method and exploits three essential ingredients: (i) the introduction of ``good unknowns" that decouple the perturbation system; (ii) the construction of a carefully designed Fourier multiplier that simultaneously captures the enhanced dissipation and inviscid-damping effects while taming the lift-up mechanism; and (iii) the design of distinct energy functionals for the incompressible and compressible modes.
	\end{abstract}

	\setcounter{tocdepth}{1}
	
	
	\maketitle

	\section{Introduction}
	In this paper, we consider the two-dimensional isentropic compressible Navier-Stokes equations in a domain $\T\times \R$:
	\begin{align}\label{eq-CNS}
		\left\{\begin{array}{l}
			\partial_t\rho + \dive_{\x}(\rho \uu) = 0,\\
			\partial_t(\rho \uu) + \dive_{\x}(\rho \uu \otimes \uu) + \f{1}{M}\nabla_{\x} \mathcal{P}(\rho) = \mu \Delta_{\x} \uu + (\mu + \mu')\nabla_{\x} \dive_{\x} \uu,
		\end{array}\right.
	\end{align}
	where $\rho(t,\x)$, $\uu(t,\x)=(u^1,u^2)^T(t,\x)$ and $\mathcal{P}(\rho)$ represent the density, the velocity and the pressure, respectively. 	
	The pressure $\mathcal{P}(\rho)$ here is assumed to be a smooth function in a neighborhood of $1$ with $\mathcal{P}'(1)>0$. 
	Here, $M>0$ is the Mach number, $\mu$ and $\mu'$ are the shear viscosity and the bulk viscosity coefficients, respectively, satisfying the following physical condition
	\begin{align*}
		\mu>0,\quad \mu+\mu'\geq0.
	\end{align*}
	Without loss of generality, we assume that $\mathcal{P}'(1)=1$ and $M=1$ in this paper.  The initial data of system \eqref{eq-CNS} is given by
	\begin{align*}
		(\rho,\uu)|_{t=0}=(\rho_{in},\uu_{in})(\x).
	\end{align*}	
	In this paper, we study the nonlinear stability of the Couette flow given by
	\begin{align*}
		(\bar\rho,\bar\uu)(t,\x)=(1,(y,0)^{T}).
	\end{align*}
	Thus, we introduce  $(n, \vv)$ to be the perturbation of $(\rho, \uu)$ around $(\bar\rho, \bar\uu)$, i.e., 
	\begin{align}\label{def: eq-pertur}
		n\tri \rho-\bar\rho,\quad\vv\tri \uu-\bar\uu,
	\end{align}
	which satisfies the following perturbation system
	\begin{align}\label{eq-pertur}
		\left\{\begin{array}{l}
			\partial_t n +y\partial_x n+\dive_{\x}\vv = F_1,\\[1ex]
			\partial_t \vv + y\partial_x\vv+(v^2,0)^T+\nabla_{\x} n-\mu \Delta_{\x} \vv - (\mu + \mu')\nabla_{\x} \dive_{\x} \vv = \FF_2,\\[1ex]
			(n,\vv)|_{t=0}=(n_{in},\vv_{in})(\x)=(\rho_{in}-\bar\rho,\uu_{in}-\bar\uu),
		\end{array}\right.
	\end{align}
	where the nonlinear terms $F_1$ and $\FF_2$ can be determined explicitly as:
	\begin{align}
		F_1\tri& -\vv\cdot\nabla_{\x} n-n\dive_{\x} \vv,\nonumber\\
		\FF_2\tri&-\vv\cdot\nabla_{\x}\vv-\left(\f{\mathcal{P}'(1+n)}{1+n}-1\right)\nabla_{\x} n-\f{n}{1+n}\left(\mu \Delta_{\x} \vv + (\mu + \mu')\nabla_{\x} \dive_{\x} \vv\right).\label{def-bf-F2}
	\end{align}

	\subsection{Background and known results}	
	
	The stability of plane Couette flow has been investigated since the seminal works of Rayleigh \cite{Rayleigh} and Kelvin \cite{kelvin}. Romanov \cite{ Romanov} proved that Couette flow is spectrally stable for all Reynolds numbers, a result that appears to contradict both numerical simulations and physical experiments \cite{Chapman, D-R, GG, OK, SH, Trefethen, Yaglom}, which indicate that shear flows become unstable and undergo transition to turbulence at sufficiently high Reynolds numbers. This discrepancy is known as the Sommerfeld paradox.  To elucidate the mechanisms underlying transition to turbulence, Trefethen et al. \cite{Trefethen} first formulated the transition threshold problem: namely, to quantify the magnitude of perturbations required to trigger instability and to determine their scaling with the Reynolds number. More recently, Bedrossian, Germain, and Masmoudi \cite{Bedrossian-Germain-Masmoudi-2017, Bedrossian-Germain-Masmoudi-2019} provided a rigorous mathematical framework for this problem, stated as follows.
	
	{\it Given a norm $\|\cdot\|_{X}$, find a $\beta=\beta(X)$ so that
		\begin{align*}
			&\|\vv_{in}\|_{X} \leq Re^{-\beta} \to \textrm{stability},\\
			&\|\vv_{in}\|_{X} \geq Re^{-\beta}\to \textrm{instability}.
		\end{align*}
	}

	The exponent $\beta$ is referred to as the transition threshold.

	A substantial body of work in applied mathematics and physics has been devoted to estimating the exponent $\beta$ (see, e.g., \cite{BT, Chapman, DBL, LK, LHS, RSBH, Waleffe, Yaglom}). Over the past decade, rigorous mathematical results have appeared in rapid succession. For the two-dimensional incompressible case without physical boundaries (i.e., $\T\times\R$), Bedrossian, Masmoudi, and Vicol\cite{Bedrossian-Masmoudi-Vicol-2016} proved that $\beta=0$ when the perturbation space $X$ is of Gevrey class $2+$. If $X$ is taken to be a Sobolev space, subsequent works \cite{MZ, WZ-2023} have established the upper $\beta\leq \f13$.  In the finite channel $\T\times [-1, 1]$, Chen, Li, Wei, and the third author \cite{CLWZ} derived $\beta\le \f12$ for the two-dimensional problem under the no-slip boundary condition. Further two-dimensional stability results for Couette flow can be found in \cite{BWV, MZ-2020, LMZ, Bedrossian-He- Iyer-Wang}.   The picture changes significantly in three dimensions because of a new linear mechanism—the lift-up effect—that induces transient growth of disturbances.  Consequently, the transition threshold problem becomes substantially more difficult.  For the domain $\T\times\R\times\T$, Bedrossian, Germain, and Masmoudi \cite{Bedrossian-Germain-Masmoudi-2020, Bedrossian-Germain-Masmoudi-2022, Bedrossian-Germain-Masmoudi-2017} obtained $\beta\le 1$ when $X$ is a Gevrey space and $\beta\le \f32$ when $X$ is a Sobolev space. Wei and the third author \cite{WZ-cpam} later improved the Sobolev result to $\beta\le 1$. Most recently, Chen, Wei, and the third author \cite{CWZ-MAMS} treated the channel geometry $\T\times [-1, 1]\times\T$ under the no-slip boundary condition and again derived $\beta\le 1$ for Sobolev perturbations.

	Although the stability of compressible flows has a long history, the stability of compressible Couette flow remains far less understood than its incompressible counterpart.  The compressible system is inherently more complicated, and the associated lift-up effect is markedly stronger.  To date, most investigations have been numerical or physical (see, e.g., \cite{Glatzel, Glatzel1, HSH, HZ, MDA, GE, CRS, CRS1, FI}).  On the mathematical side, Kagei \cite{Kagei} established asymptotic stability of Couette flow for small Reynolds numbers, while Li and Zhang \cite{LZ} later treated the case with slip boundary conditions.  Regarding the transition threshold problem, Antonelli, Dolce, and Marcati \cite{Antonelli-Dolce-Marcati-2021} proved linear stability and enhanced dissipation for two-dimensional Couette flow on $\T\times\R$ at high Reynolds number; they recovered the $t^{1/2}$ growth of the $L^2$ norm previously observed in the physical literature \cite{GLB}.  Zeng, Zi, and the third author \cite{Zeng-Zhang-Zi-2022} obtained analogous results on $\T\times\R\times\T$. Most recently, Huang, Li, and Xu \cite{Huang-Li-Xi-2024} addressed the nonlinear problem and showed that $\beta\leq 11/3$ when the perturbation space $X$ is taken to be a Sobolev space.

	The goal of  this paper is to  investigate the nonlinear stability of  the Couette flow for the two-dimensional compressible Navier–Stokes equations and seek the optimal threshold exponent $\beta$.
	
	\subsection{Main result}
	
	Due to the different physical properties of the velocity field $\vv$, we decompose it into two distinct components: the compressible part and the incompressible part. Specifically, we define
	\begin{align}
		d\tri \dive_{\x} \vv,\quad \omega\tri\curl_{\x} \vv=\nabla^{\perp}_{\x}\cdot \vv=-\pa_yv^1+\pa_xv^2.
	\end{align}
	Thus, we have
	\begin{align}\label{decom-v}
		\vv=&\nabla_{\x}\Delta^{-1}_{\x}\dive_{\x} \vv+\nabla^{\perp}_{\x}\Delta^{-1}_{\x}\curl_{\x} \vv=\nabla_{\x}\Delta^{-1}_{\x}d+\nabla^{\perp}_{\x}\Delta^{-1}_{\x}\omega.
	\end{align}
	In view of the system \eqref{eq-pertur}, we find that $(n,d,\omega)$ satisfies the following system
	\begin{align}\label{eq_n-d-om}
		\left\{\begin{array}{l}
			\partial_t n +y\partial_x n+d = F_{1},\\[1ex]
			\partial_t d + y\partial_xd+2\pa_x\pa_y\Delta^{-1}_{\x}d+2\pa_x^2\Delta^{-1}_{\x}\omega+\Delta_{\x} n-\nu \Delta_{\x} d = \dive_{\x}{\FF}_{2},\\[1ex]
			\partial_t \omega + y\partial_x\omega-d-\mu \Delta_{\x} \omega = \curl_{\x}{\FF}_2,\\[1ex]
			(n,d,\omega)|_{t=0}=(n_{in},d_{in},\omega_{in})(\x)=(n_{in},\dive_{\x}\vv_{in},\curl_{\x} \vv_{in})(\x),
		\end{array}\right.
	\end{align}
	where $\nu\tri 2\mu+\mu'$.

	For any $s\in\N^+$, we say that $q(t,x,y)\in \HH^s(\T\times\R)$ whenever
	\begin{align*}
		\|q\|_{\HH^s}^2\tri \sum_{|\al|\leq s}\int_{\T\times\R}|(\pa_x,\pa_y+t\pa_x)^\al q|^2\,dxdy<+\infty.
	\end{align*}
	For any function $q(t,x,y)$, we also define
	\begin{align*}
		\PP_{0} q \tri \f{1}{2\pi} \int_{\T}q(t,x,y)\,dx,\quad \PP_{\not=} q  \tri q-\PP_{0} q.
	\end{align*}
	Here, to simplify the notations, we sometimes use the $(q_0, q_{\not=})$ instead of $(\PP_{0} q,\PP_{\not=} q)$.
	
	Now we are in a position to state our main result.
	
	\begin{theo}\label{theo1}
		There exists a small $\ep$ independent of $\mu, \nu$, such that if $0<\nu, \mu\leq 1$, $\nu$ has the same order of $\mu$, and the initial data $(n_{in},d_{in},w_{in})$ satisfies
		\begin{align}\label{initial data-ass}
			\|(n_{in},\vv_{in})\|_{H^4(\T\times\R)}
			\leq \ep \mu,
		\end{align}
		then the system \eqref{eq_n-d-om} admits a unique global solution $(n, d,\om)$, such that for any $t\geq0$, 
		\begin{align*}
			&\|(\PP_0n,\PP_0\Lambda^{-1}_{\x}d,\PP_0\om)\|_{L_t^{\infty}(\HH^3(\T\times\R))}+	\mu^{\f16}	\|(\PP_{\not=}n,\PP_{\not=}\Lambda^{-1}_{\x}d,\PP_{\not=}\om)\|_{L_t^{\infty}(\HH^3(\T\times\R))}   \leq C	 \|(n_{in},\vv_{in})\|_{H^4(\T\times\R)} ,
		\end{align*}
		where for any $s\in\R$, $\Lambda^{s}_{\x}\tri (-\Delta_{\x})^{\f{s}{2}}$ and $C$ is a constant independent of $t$ and $\mu, \nu$.
	\end{theo}

	\begin{rema}
		Based on Theorem \ref{theo1}, we can derive the decay rate for the zero mode of the solution as
		\begin{align}\label{decay-zero}
			&\|(\PP_0 \pa_y n,\PP_0 d,\PP_0\pa_y \om)\|_{L^2(\R)} 
			\leq C (1+\mu t)^{-\f12}.
		\end{align}
		For non-zero mode of the solution, we can establish the following enhanced dissipation and inviscid damping estimates
		\begin{align}
			&\|(\PP_{\not=}n,\PP_{\not=}\Lambda^{-1}_{\x}d)\|_{L^2(\T\times\R)}
			\leq C(1+t)^{\f12}e^{-\ep_1 \mu^{\f13}t},
			\label{decay-nonzero-1}
			\\
			&(1+t)\|\PP_{\not=}\pa_{x}\Delta_{\x}^{-1} \om\|_{L^2(\T\times\R)}+\|\PP_{\not=}\pa_{y}\Delta_{\x}^{-1} \om\|_{L^2(\T\times\R)}
			\leq
			C(1+t)^{-\f12}e^{-\ep_1 \mu^{\f13}t}.
			\label{decay-nonzero-2}
		\end{align}
		The decay estimate \eqref{decay-zero} can be derived by employing the time-weighted energy estimates.
		The decay estimates \eqref{decay-nonzero-1}-\eqref{decay-nonzero-2} are immediate consequences of applying the weights $m_1$ and $\varphi$ in the proof of Theorem \ref{theo1}. 
		Their proofs are therefore omitted.
	\end{rema}

	\begin{rema}\label{rmk-optimal}
		Let us provide some discussions on the stability threshold of  the initial data via a formal asymptotic analysis. \smallskip
		
		First, the linear stability analysis of Antonelli–Dolce–Marcati \cite{Antonelli-Dolce-Marcati-2021} shows that the linearized system associated with \eqref{decom-v} loses a factor $\mu^{1/6}$;  more precisely,
		\begin{align*}
			\|(\PP_{\not=}  n, \PP_{\not=} \vv)\|_{L^2} \leq \mu^{-1/6} \|(n_{in}, \vv_{in})\|_{L^2}.
		\end{align*}
		Returning to the nonlinear system,
		\begin{align*}
			\pa_t \PP_{\not=}  \vv=- \PP_{\not=} v^2 \pa_y \PP_{\not=}  \vv +\cdots,
		\end{align*}
		we obtain, schematically,
		\begin{align*}
			\PP_{\not=} \vv\sim t\cdot \PP_{\not=} v^2 \pa_y \PP_{\not=} \vv+\cdots.
		\end{align*}
		Enhanced dissipation dictates the time scale $t\sim \mu^{-\f13}$, while $\pa_y=(\pa_y-t\pa_x)+t\pa_x$ introduces an additional loss of  $\mu^{-\f13}$. Collecting these losses, closing the energy estimates for the nonlinear problem requires at least $\|(n_{in}, \vv_{in})\|_{X}\sim \mu^{\alpha}$ with $\alpha\geq 5/6$.  
		
		For compressible flows, however, $\alpha = 5/6$ is still insufficient because of the stronger lift-up effect. This marks another key difference from the incompressible setting, where—even in three dimensions—the lift-up mechanism appears only in the equation for $v^1$.  Exploiting the divergence-free condition, \cite{WZ-cpam} therefore starts from the subsystem $(\PP_{\not=}v^2, \PP_{\not=}v^3)$  and recovers $\PP_{\not=}v^1$  a posteriori; in effect, they estimate $\|(\PP_{\not=}v^2, \PP_{\not=}v^3)\|_{H^2}$ rather than $\|\PP_{\not=}\vv\|_{H^2}$.  Incompressible flows can thus circumvent the lift-up term for non-zero modes.
		Compressible flows lack this structure.  To tame the lift-up growth, the works \cite{Antonelli-Dolce-Marcati-2021,Bedrossian-Germain-Masmoudi-2017,Huang-Li-Xi-2024,Liss-2020,Zeng-Zhang-Zi-2022} introduce the Fourier multiplier $\varphi$ defined in \eqref{def-phi} (see Section \ref{sect-F-M}), which  satisfies  $1\lesssim  \varphi\lesssim  \mu^{-2/3}$. Using this weight, we control $\|\varphi^{-1/4} \hat{ \vv}\|_{H^2}$. Formally,
		\begin{align}\label{rmk-comm}
			(\F^{-1}[\varphi^{-\f14}] \PP_{\not=}  \vv)\sim t\cdot \PP_{\not=} v^2 \pa_y (\F^{-1}[\varphi^{-\f14}]\PP_{\not=} \vv)
			+t[\F^{-1}[\varphi^{-\f14}], \PP_{\not=} v^2 \pa_y ]\PP_{\not=}  \vv +\cdots,
		\end{align}
		where $\hat{ \vv}$ stands for the Fourier transform of $\vv$. The commutator sacrifices the weight $\varphi^{1/4}$, resulting in an additional loss of $\mu^{1/6}$. A similar loss already appears in \cite{Bedrossian-Germain-Masmoudi-2017}, where $\|\PP_{\not=}\vv\|_{H^N}$ is estimated with the aid of the multiplier $\varphi$ to counteract the lift-up effect.  Guided by these considerations,  the stability threshold of the initial data \eqref{initial data-ass} seems optimal. 
	\end{rema}

	\medskip

	Throughout this paper, all constants $C$ may be different in different lines, but universal.  A constant with subscript(s) illustrates the dependence of the constant, for example, $C_{a}$ is a constant which depends on $a$. For $a\lesssim b$, we mean that there exists a universal constant $C$, such that $a\leq Cb$. While $a\gtrsim b$, we mean that there exists a universal constant $C$, such that $a\geq Cb$. And $a\sim b$ means that $a\lesssim b\lesssim a$.  
	


		%
	%
	
	\section{Reformulation of the system and Main strategy}
	\subsection{New formulation}
	
	One of the main differences between compressible and incompressible flows is that the lift-up effect already appears in two-dimensional compressible flows, whereas in the incompressible setting it is genuinely three-dimensional. This mechanism induces linear growth up to the time scale $t\leq {\mu}^{-1}$. To mitigate the lift-up effect, Antonelli, Dolce, and Marcati \cite{Antonelli-Dolce-Marcati-2021} introduced a good unknown  
	\begin{align}\label{def-w}
		w\tri \omega+n-\mu d,
	\end{align}		
	which satisfies
	\begin{align}\label{eq_n-d-w}
		\left\{\begin{array}{l}
			\partial_t n +y\partial_x n+d = F_{1},\\[1ex]
			\partial_t d + y\partial_xd+2\pa_x\pa_y\Delta^{-1}_{\x}d+2\pa_x^2\Delta^{-1}_{\x}(w-n+\mu d)+\Delta_{\x} n-\nu \Delta_{\x} d = F_{2},\\[1ex]
			\partial_t w + y\partial_x w-\mu \Delta_{\x} w+\mu(\mu+\mu') \Delta_{\x} d-2\mu \pa_x\pa_y\Delta^{-1}_{\x}d-2\mu\pa_x^2\Delta^{-1}_{\x}(w-n+\mu d)
			= F_{3},\\[1ex]
			(n,d,w)|_{t=0}=(n_{in},d_{in},w_{in})(\x)\tri
			(n_{in},d_{in},\omega_{in}+n_{in}-\mu d_{in})(\x),
		\end{array}\right.
	\end{align}
	where 	
	\begin{align*}
		F_{1}\tri& -\vv\cdot\nabla_{\x} n-n d,
		\\
		F_{2}\tri&-\vv\cdot\nabla_{\x} d-F_{21},
		\\
		F_{3}\tri&-\vv\cdot\nabla_{\x} w-d(w+\mu d)-\nabla^{\perp}_{\x}\cdot\Big(\f{n}{1+n}\big(\nu \nabla_{\x} d+\mu \nabla^{\perp}_{\x}(w-n+\mu d)\big)\Big)+\mu F_{21},
	\end{align*}
	with
	\begin{align*}
		F_{21}\tri &~(\pa_xv^1)^2+(\pa_yv^2)^2+2\pa_yv^1\pa_xv^2+\dive_{\x}\Big(\big(\f{\mathcal{P}'(1+n)}{1+n}-1\big)\nabla_{\x} n\Big) \\
		&+\dive_{\x} \Big(\f{n}{1+n}\big(\nu \nabla_{\x} d+\mu \nabla^{\perp}_{\x}(w-n+\mu d)\big)\Big).
	\end{align*}

	\begin{rema} We now highlight the advantages of the new formulation.
		\begin{itemize}
			
			\item  The equation for $w$ contains no linear term in $d$, thereby eliminating a potential source of unbounded growth.  This structural property is essential to our analysis.	
			
			\item  For streak solutions, the system governing $(\PP_0 n, \PP_0 d, \PP_0 w)$ is completely devoid of lift-up terms.  Consequently, we obtain uniform-in-time bounds  for $(\PP_0 n, \PP_0 d, \PP_0 w)$ which is controlled by the initial data.  Using the identity $\PP_{0} \om= \partial_y\PP_{0} v^1$,  we further derive comparable estimates for  $(\PP_0 n, \PP_0 v^2, \PP_0 \pa_yv^1)$.	
			Although 	$\PP_0 v^1$ itself may grow in $L^2$, this growth does not propagate to the energy norms we control; hence the a priori estimates close.

			\item  For the non-zero mode, the lift-up effect appearing in the equation for $\PP_{\not=} w$  is strictly weaker than that in the equation for  $\PP_{\not=} \om$.	 We prove that $\PP_{\not=}w$ do not lose $\mu^{\f16}$ suffered by  $\PP_{\not=} \om$. In this sense, $w$ behaves exactly like the vorticity in two-dimensional incompressible flow.

		\end{itemize}

	\end{rema}

	\subsection{Coordinates transformation} 
	
	To eliminate the transport terms, such as $y\pa_x n$, $y\pa_x d$ and $y\pa_x w$, from system \eqref{eq_n-d-w}, we introduce the moving-frame coordinates
	\begin{align}\label{def-t1-x1-y1}
		t\tri t,\quad x_1\tri x-ty,\quad y_1\tri y.
	\end{align}
	In this frame we set
	\begin{align*}
		(N,D,W,\V,P)(t,x_1,y_1)\tri (n,d,w,\vv,\mathcal{P})(t,x_1+ty_1,y_1).
	\end{align*}
	More generally, for any function $q(t,x,y)$(which may be $n,d,w,\vv$), we define $Q(t,x_1,y_1)\tri q(t,x_1+ty_1,y_1)$ under which the differential operators transform as	
	\begin{align*}
		&\widetilde{\nabla}q\tri \nabla_{\x}q= (\pa_{x_1}Q,(\pa_{y_1}-t\pa_{x_1})Q),
		\quad\widetilde{\dive}q\tri \dive_{\x}q=\pa_{x_1} Q^1+(\pa_{y_1}-t\pa_{x_1})Q^2,
		\\
		&\widetilde{\Delta}q\tri \Delta_{\x}q=\pa_{x_1}^2Q+(\pa_{y_1}-t\pa_{x_1})^2Q.
	\end{align*}
	And we denote 
	\begin{align*}
		\nabla Q\tri (\pa_{x_1}Q,\pa_{y_1}Q),
		\quad \dive Q\tri \pa_{x_1} Q^1+\pa_{y_1}Q^2,
		\quad \Delta Q\tri \pa_{x_1}^2Q+\pa_{y_1}^2Q.
	\end{align*}
	The system \eqref{eq_n-d-w} is then reduced to
	\begin{align}\label{eq_n-d-w-1}
		\left\{\begin{array}{l}
			\pa_{t} N+D=\widetilde{F}_1,\\[1ex]
			\pa_{t} D+2\pa_{x_1}(\pa_{y_1}-t\pa_{x_1})\widetilde{\Delta}^{-1}D+2\pa_{x_1}^2\widetilde{\Delta}^{-1}(W-N+\mu D)+\widetilde{\Delta} N-\nu \widetilde{\Delta} D=\widetilde{F}_2,\\[1ex]
			\pa_{t} W-\mu \widetilde{\Delta} W+\mu(\mu+\mu') \widetilde{\Delta} D-2\mu \pa_{x_1}(\pa_{y_1}-t\pa_{x_1})\widetilde{\Delta}^{-1}D-2\mu\pa_{x_1}^2\widetilde{\Delta}^{-1}(W-N+\mu D)=\widetilde{F}_3,\\[1ex]
			(N,D,W)|_{t=0}=(n_{in},d_{in},w_{in})(x_1,y_1),
		\end{array}\right.
	\end{align}
	where the source terms may be determined explicitly,
	\begin{align}
		\widetilde{F}_{1}\tri& -\V\cdot\widetilde{\nabla} N-N D,
		\label{def_wF1}\\
		\widetilde{F}_{2}\tri&-\V\cdot\widetilde{\nabla} D
		-\widetilde{F}_{21},
		\label{def_wF2}\\
		\widetilde{F}_{3}\tri&-\V\cdot\widetilde{\nabla} W-D(W+\mu D)-\widetilde{\nabla}^{\perp}\cdot\left(g(N)\big(\nu \widetilde{\nabla} D+\mu \widetilde{\nabla}^{\perp}(W-N+\mu D)\big)\right)
		+\mu\widetilde{F}_{21},
		\label{def_wF3}\\
		\widetilde{F}_{21}\tri&~ (\pa_{x_1}V^1)^2+((\pa_{y_1}-t\pa_{x_1})V^2)^2+2(\pa_{y_1}-t\pa_{x_1})V^1\pa_{x_1}V^2+\widetilde{\dive}\left(f(N)\widetilde{\nabla} N\right)
		\label{def_wF21}\\
		&+\widetilde{\dive} \left(g(N)\big(\nu \widetilde{\nabla} D+\mu \widetilde{\nabla}^{\perp}(W-N+\mu D)\big)\right),
		\nonumber
	\end{align}
	and the nonlinear functions of $N$ are defined by
	\begin{align}\label{def-f-g}
		f(N)\tri \f{P'(1+N)}{1+N}-1,\quad g(N)\tri \f{N}{1+N}.
	\end{align}

	We introduce the definition of Sobolev space under the new coordinates. For any $s\in\N^+$, we say that $Q(t,x_1,y_1)\in H^s(\T\times\R)$ whenever
	\begin{align*}
		\|Q\|_{H^s}^2\tri \sum_{|\al|\leq s}\int_{\T\times\R}|(\pa_{x_1},\pa_{y_1})^\al Q|^2\,dx_1dy_1<+\infty.
	\end{align*}
	We note that for any $q(t,x,y)$ and $Q(t,x_1,y_1)\tri q(t,x_1+ty_1,y_1)$, we have
	\begin{align*}
		\|q\|_{\HH^s}^2=\|Q\|_{H^s}^2,\quad \forall~ s\in\N^+.
	\end{align*}
	
	\medskip

	Then it is easy to see that Theorem \ref{theo1} is a direct corollary of the following theorem.
	\begin{theo}\label{theo2}
		Under the assumptions in Theorem \ref{theo1}, system \eqref{eq_n-d-w-1} admits a unique global solution $(N,D,W)$ such that, for any $t\geq0$, 
		\begin{align*}
			\|\PP_0(N,\widetilde{\Lambda}^{-1}D)\|_{L_t^{\infty}H^3} +\|W\|_{L_t^{\infty}H^3}  \leq C\big(\|(n_{in},\Lambda^{-1}d_{in})\|_{H^4}+\|w_{in}\|_{H^3}\big),
		\end{align*}
		and
		\begin{align*}
			\mu^{\f16}	\|\PP_{\not=}(N,\widetilde{\Lambda}^{-1}D)\|_{L_t^{\infty}H^3} +\mu^{\f12}	\|(\widetilde{\Lambda}N,D)\|_{L_t^{\infty}H^3} \leq C\big(\|(n_{in},\Lambda^{-1}d_{in})\|_{H^4}+\|w_{in}\|_{H^3}\big),
		\end{align*}		
		where for any $s\in\R$, $\widetilde{\Lambda}^{s}\tri (-\widetilde{\Delta})^{\f{s}{2}}$  and
		$\Lambda^{s}\tri (-\Delta)^{\f{s}{2}}$.
	\end{theo}

	\subsection{Main strategy}
	In this subsection, we outline the main ideas of this paper.\smallskip
	
	$\bullet$ \underline{Introducing a new formulation.} 
	Guided by the linear analysis of Antonelli–Dolce–Marcati \cite{Antonelli-Dolce-Marcati-2021}, we recall that the linearized problem for system \eqref{eq_n-d-w} was treated as  a $3\times3$ system leading to estimates for 
	$$(\varphi^{-\f34}(\widetilde{\Delta})^{\f{1}{2}}\hat{N}, \varphi^{-\f34} \hat{D}, \varphi^{-\f34}\hat{ W} ),$$ 
	where $\hat{q}$ stands for the Fourier transform of $q$, and  the weight $\varphi^{-\f34}$ is introduced  to control the lift-up effect. In this setting, we need to deal with the following nonlinear term
	\begin{align}\label{rmk-comm---1}
		\F^{-1}[\varphi^{-\f34}]  ( \PP_{\not=} V \cdot \widetilde{\na} D )\sim   \PP_{\not=}V\cdot \widetilde{\na} (\F^{-1}[\varphi^{-\f34}]\PP_{\not=} D)+[\F^{-1}[\varphi^{-\f34}], V \cdot \widetilde{\na} ]\PP_{\not=} D +\cdots,
	\end{align}
	Here we notice that the commutator will bring a loss of $\mu^{\frac12}$, preventing the desired transition threshold.

	To overcome this, we return to the unknown $(n,\vv)$ while retaining the favorable structure of system \eqref{eq_n-d-w}.  Define $\UU=\widetilde{\Delta}^{-\f12}D$;  we then regard $(N, \UU)$ as a $2\times2$ system and estimate $W$ separately.  Schematically, we first study
	\begin{align*} 
		\left\{\begin{array}{l}
			\pa_{t} N+\widetilde{\Delta}^{\f12}\UU=\cdots,\\
			\pa_{t} \UU+\pa_{x_1}(\pa_{y_1}-t\pa_{x_1})\widetilde{\Delta}^{-1}\UU+2\pa_{x_1}^2\widetilde{\Delta}^{-\f32}(W-N+\mu D)+\widetilde{\Delta}^{\f12} N-\nu \widetilde{\Delta} \UU=\cdots,
		\end{array}\right.
	\end{align*}
	Enhanced dissipation and inviscid damping are recovered by introducing Fourier multipliers $m_1, m_2$, while the lift-up effect is tamed by the multiplier $\varphi^{\f14}$; see Subsection~\ref{sect-F-M}.  In this system, we use the weight $\varphi^{-\frac14}$, which will bring better estimates, i.e., result in a loss of $\mu^{\frac16}$ in nonlinear  estimates (see \eqref{rmk-comm}). Moreover, the equation for $W$ contains no lift-up term, and its energy estimate is carried out directly.

	$\bullet$ \underline{Constructing  suitable energy functionals.}  
	Estimating $(N,\UU)$ and $W$ separately creates a new difficulty.  On the one hand,
	\begin{align*}
		\pa_t (N, \UU) \sim \widetilde{\Delta}^{-\f12} W,
	\end{align*}
	so $\| (N, \UU) \|_{H^s}$ is controlled  by $\|\widetilde{\Delta}^{-\f12} W\|_{H^s}$. On the other hand, the equation for $\widetilde{\Delta}^{-\f12} W$ is 
	\begin{align*}
		\pa_t(p^{-\f12} \hat{W})+\f{\pa_{t} p}{2p}(p^{-\f12} \hat{W})=\cdots,\quad  p=k^2+(\xi-kt)^2,
	\end{align*}
	where the second term arises from the commutator  $[\pa_{t}, p^{-\f12}]$. This term is not a lift-up contribution and cannot be absorbed by $\varphi$. This obstacle motivated \cite{Antonelli-Dolce-Marcati-2021} to keep the $3\times3$ system. We instead continue to work with $W$ rather than $\widetilde{\Delta}^{-\f12} W$; accordingly,  we need $\|W\|_{H^s}$ to control the energy. The price is that $ (N, \UU)$ now requires one higher derivative.  Indeed, from
	\begin{align*}
		\pa_{t} W-\mu \widetilde{\Delta} W+\mu(\mu+\mu') \widetilde{\Delta} D-2\mu \pa_{x_1}(\pa_{y_1}-t\pa_{x_1})\widetilde{\Delta}^{-1}D=\cdots,
	\end{align*}	
	we get that $\|W\|_{H^s}$ is control by  $\mu^{\frac 56}\| (N, \UU) \|_{H^{s+1}}$. The prefactor $\mu$ compensates the derivative loss, making the estimate straightforward.  Meantime, to control the terms, like $V\cdot n $, we also need the estimates of $\| (N, \UU) \|_{H^{s+1}}$. For this reasons,  we introduce the additional high-order energy functional $\E_{j,3-j}^{com}(t)$  for the compressible part $(N, \UU)$; details are given in Section~\ref{sect:energy fun}.


	\section{Introduction of Fourier multipliers}
	In this section, we introduce some key Fourier multipliers, and collect definitions and auxiliary lemmas that will be used repeatedly in the energy estimates.	
	
	\subsection{The Fourier multipliers}\label{sect-F-M}

	We recall that the symbols $p$ and $2\f{k(\xi-kt)}{p}=-\f{\pa_{t}p}{p}$ correspond to the operators $-\widetilde{\Delta}$ and $2\pa_{x_1}(\pa_{y_1}-t\pa_{x_1})\widetilde{\Delta}^{-1}$, respectively. Since $2\pa_{x_1}(\pa_{y_1}-t\pa_{x_1})\widetilde{\Delta}^{-1}D$ is the bad term in equation $\eqref{eq_n-d-w-1}_2$, it is natural to study the toy model
	\begin{align}\label{eq-toy}
		\partial_{t} \hat{D}^k =\f{\pa_{t}p}{p}\hat{D}^k-\nu p\hat{D}^k.
	\end{align}
	For $t\leq \f{\xi}{k}$, we have $\pa_{t}p=-2k(\xi-kt)\leq 0$; hence the first term on the right-hand side acts as a damping term for $\hat{D}^k$.
	For $t> \f{\xi}{k}$ the derivative becomes positive, $\pa_{t}p> 0$,
	and the same term drives growth.
	Near the critical time $t=\f{\xi}{k}$ the dissipation satisfies $\nu p\sim\nu k^2$;
	thus the dissipation alone cannot uniformly control the growth in $\nu$.
	Nevertheless, sufficiently far from the critical layer, dissipation prevails:  for any fixed $\beta>2$,  the inequality	
	$$\f{\pa_{t}p}{p}\leq \f{1}{4}\nu p,$$
	holds whenever $|t-\f{\xi}{k}|\geq \beta \nu^{-\f13}$.	
	
	To exploit this dichotomy and allow the dissipative term $\nu p \hat{D}^k$ to balance the growth caused by $\f{\pa_tp}{p}\hat{D}^k$, we follow \cite{Antonelli-Dolce-Marcati-2021,Bedrossian-Germain-Masmoudi-2017,Huang-Li-Xi-2024,Liss-2020,Zeng-Zhang-Zi-2022} and introduce the Fourier multiplier $\varphi(t,k,\eta)$:
	\begin{align}\label{def-phi}
		\left\{\begin{array}{l}
			\partial_{t} \varphi = 
			\left\{\begin{array}{l}
				0,\quad t\not\in [\f{\xi}{k},\f{\xi}{k}+\beta\nu^{-\f13}],\\[1ex]
				\f{\pa_{t} p}{p}\varphi	,\quad t\in [\f{\xi}{k},\f{\xi}{k}+\beta\nu^{-\f13}],\\[1ex]
			\end{array}\right.\\
			\varphi(0,k,\xi)=1,
		\end{array}\right.
	\end{align}
	where $\beta>2$ is a constant determined later.
	Specially, we have $\varphi(t,0,\xi)=1$.

	In the energy estimates we also employ the following two Fourier multipliers to obtain enhanced dissipation and inviscid-damping bounds; they have been used extensively in \cite{Antonelli-Dolce-Marcati-2021,Bedrossian-Germain-Masmoudi-2017,Bedrossian-Germain-Masmoudi-2019,Bedrossian-Masmoudi-Vicol-2016,Huang-Li-Xi-2024,Zeng-Zhang-Zi-2022,Zillinger-2017}:	
    \begin{align}\label{def-m1}
		\left\{\begin{array}{l}
			\displaystyle \pa_{t}m_1(t,k,\xi)=\f{2\nu^{\f13}}{1+\nu^{\f23}(\f{\xi}{k}-t)^2}m_1(t,k,\xi),\\[1ex]
			\displaystyle m_1(0,k,\xi)=\exp\Big\{ 2\arctan \Big(-\nu^{\f13}\f{\xi}{k}\Big)\Big\},
		\end{array}\right.
	\end{align}
	and
	\begin{align}\label{def-m2}
		\left\{\begin{array}{l}
			\displaystyle 
            \pa_{t}m_2(t,k,\xi)=A\f{k^2}{p}m_2(t,k,\xi)=A\f{1}{1+(\f{\xi}{k}-t)^2}m_2(t,k,\xi),\\[1ex]
			\displaystyle m_2(0,k,\xi)=\exp\left\{A\arctan \left(-\f{\xi}{k}\right)\right\},
		\end{array}\right.
	\end{align}
	where $A$ is a large, but fixed constant, which will be determined later.
	It is easy to check that, for any $t\geq0$ and $\xi\in\R$, the quantities $m_1$ and $m_2$ can be expressed explicitly as
	\begin{align*}
		&m_1(t,k,\xi)=
		\begin{dcases}
			\exp\left\{2\arctan \left(\nu^{\f13}\left(t-\f{\xi}{k}\right)\right)\right\},\quad \text{for}~~k\in\Z\backslash\{0\},\\
			1,\quad\text{for}~~k=0;
		\end{dcases}
		\\
		&m_2(t,k,\xi)=
		\begin{dcases}
			\exp\left\{A\arctan \left(t-\f{\xi}{k}\right)\right\},\quad \text{for}~~k\in\Z\backslash\{0\},\\
			1,\quad\text{for}~~k=0.
		\end{dcases}
	\end{align*}
	
	These three Fourier multipliers introduced above enjoy the following crucial properties:
	\begin{lemm}\label{sec2:lem-1}
		Let $\varphi$, $m_1$ and $m_2$ be three Fourier multipliers defined in \eqref{def-phi}, \eqref{def-m1} and \eqref{def-m2}.
		Then the multipliers $\varphi$, $m_1$ and $m_2$ are bounded from above and below uniformly with respect to $0<\nu<1$ and $(t,k,\xi)\in\R^+\times \Z\times\R$. Specifically, they are bounded by
		\begin{align*}
			1\leq \varphi\lesssim \nu^{-\f23},\quad
			1\leq m_1\leq e^{2\pi},\quad 
			1\leq m_2\leq e^{A\pi}.
		\end{align*}
		Moreover, for any $t\geq0$, $k\in\Z \backslash \{0\}$ and $\xi\in\R$, we have
		\begin{align}\label{est-varphi-1}
			&1\leq \varphi(t,k,\xi)\leq \beta^2 \nu^{-\f23},\quad \f{\varphi}{p}\leq \f{1}{k^2}.
		\end{align}
		For any $\max\{2\beta^{-1}(\beta^2-1)^{-1},4\beta^{-1}\}<\delta_{\beta}\leq 1$ with $\beta>2$, we also have
		\begin{align}\label{est-varphi-2}
			\left\{\begin{array}{l}
				\delta_{\beta}\left(\f{\pa_{t} m_1}{m_1}+\nu p\right)+\f{\pa_{t}\varphi}{\varphi}-\f{\pa_{t} p}{p}\geq \delta_{\beta} \nu^{\f13}, \\[1ex]
				\delta_{\beta}\left(\f{\pa_{t} m_1}{m_1}+\nu^{\f13}\right)+\f{\pa_{t}\varphi}{\varphi}-\f{\pa_{t} p}{p}\geq \f{\delta_{\beta}}{2} \nu^{\f13}.
			\end{array}\right.
		\end{align}
	\end{lemm}

	\subsection{Some useful lemmas}	
	In order to close the energy estimates for the solution, we need certain commutator estimates.
	First, for any function $q$, we introduce the following operators:
	\begin{align}
		&\RR_1 q\tri \F^{-1}[m_1^{-1}m_2^{-1}\varphi^{-\f14}\hat{q}], \quad 	\RR_2 q\tri \F^{-1}[m_1^{-1}m_2^{-1}\hat{q}],\label{def_R1-R2}\\
		&\RR_3 q\tri \F^{-1}\left[\f{\pa_{t}p}{p}\hat{q}\right]= -\pa_{x_1}(\pa_{y_1}-t\pa_{x_1})\widetilde{\Delta}^{-1}q,\label{def_R3-R4}\\
		&\RR_4 q_0\tri  \pa_{y_1}|\pa_{y_1}|^{-1}q_0,\quad \RR_5 q\tri  \widetilde{\Delta}^{-\f12}\widetilde{\dive} q. \label{def_R5}
	\end{align}
	We then recall the following two lemmas concerning commutator estimates.
	In the following first lemma, for $i=4$, we replace $\V$ and $q$ with $\V_0$ and $q_0$ in \eqref{est-Ri}. For simplify, we also denote as $\V$ and $q$.
	\begin{lemm}\label{sec2:lem-com-1}
		For any $\RR_i (i=3,4,5)$, $q\in L^2(\T\times\R)$ and $\widetilde{\nabla} \V\in L^{\infty}(\T\times\R)$, we have
		\begin{align}\label{est-Ri}
			\|[\RR_i,\V]\cdot \widetilde{\nabla}q\|_{L^2}\lesssim \|\widetilde{\nabla} \V\|_{L^\infty}\|q\|_{L^2}.
		\end{align}
	\end{lemm}
	\begin{proof}
		First, we consider the case $\RR_3$. By using coordinate transformation, we compute that
		\begin{align}
			\|[\RR_3,\V]\cdot \widetilde{\nabla}q\|_{L^2}^2
			=&\int_{\T\times\R} |[\pa_{x_1}(\pa_{y_1}-t\pa_{x_1})\widetilde{\Delta}^{-1},\V]\cdot (\pa_{x_1},\pa_{y_1}-t\pa_{x_1})q|^2(t,x_1,y_1)\,dx_1dy_1
			\label{est-Ri-trans}\\
			=&\int_{\T\times\R} |[\pa_{x}\pa_{y}\Delta_{\x}^{-1},\V] \cdot (\pa_{x},\pa_{y})q|^2(t,x-ty,y)\,dxdy
			\nonumber\\
			\lesssim&\sup_{(x,y)\in\T\times\R} |(\pa_{x},\pa_{y})\cdot \V(t,x-ty,y)|^2
			\int_{\T\times\R} |q(t,x-ty,y) |^2\,dxdy
			\nonumber\\
			=&\sup_{(x_1,y_1)\in\T\times\R} |(\pa_{x_1},\pa_{y_1}-t\pa_{x_1})\cdot \V(t,x_1,y_1)|^2
			\int_{\T\times\R} |q(t,x_1,y_1) |^2\,dx_1dy_1
			\nonumber\\
			=&\|\widetilde{\nabla} \V\|_{L^\infty}^2\|q\|_{L^2}^2.
			\nonumber
		\end{align}
		By the same argument, we can prove  the cases for $\RR_i$ with $i=2,4,5$.
	\end{proof}
	\begin{lemm}\label{sec2:lem-com-1-1}
	    For any $\RR_i (i=1,2)$, $q\in L^2(\T\times\R)$ and $\pa_{y_1} \V_0\in L^{\infty}(\R)$, we have
		\begin{align}\label{est-R1}
			\|[\RR_i,\V_0]\cdot \widetilde{\nabla}q\|_{L^2}\lesssim \|\pa_{y_1}\V_0\|_{L^\infty}\|q\|_{L^2}.
		\end{align}
	\end{lemm}
	
	\begin{proof}
		We denote $m=m_1^{-1}, m_2^{-1}$ or $\varphi^{-\f14}$, and then we have
		\begin{align}
			&m \F[\V_0\cdot\nabla q]-\hat{\V}_0*(m\F[\nabla q])
			\label{est-m}\\
			=&\int_{\R}(m(t,k,\xi)-m(t,k,\xi'))\hat{\V}_0(t,\xi-\xi')\cdot (k,\xi'-kt)\hat{q}(t,k,\xi')\,d\xi'.
			\nonumber
		\end{align}
		It is easy to check that
		\begin{align*}
			&|(m(t,k,\xi)-m(t,k,\xi'))\hat{V}_0^1(t,\xi-\xi')k\hat{q}(t,k,\xi')|
			\\
			\lesssim&~ |k(m(t,k,\xi)-m(t,k,\xi'))(\xi-\xi')^{-1}|
			|(\xi-\xi')\hat{V}_0^1(t,\xi-\xi')||\hat{q}(t,k,\xi')|
			\\
			\lesssim&~ |(\xi-\xi')\hat{V}_0^1(t,\xi-\xi')||\hat{q}(t,k,\xi')|.
		\end{align*}
		For $\xi'-kt$ sufficiently large and far away from $\xi-kt$, we have $|\xi'-kt|\sim |\xi-\xi'|$.
		Hence,
		\begin{align*}
			&|(m(t,k,\xi)-m(t,k,\xi'))\hat{V}_0^2(t,\xi-\xi')(\xi'-kt)\hat{q}(t,k,\xi')|
			\\
			\lesssim &~\|m\|_{L^\infty_{k,\xi}}|(\xi-\xi')\hat{V}_0(t,\xi-\xi')||\hat{q}(t,k,\xi')|.
			\nonumber
		\end{align*}
		For $\xi'-kt$ sufficiently small and far away from $\xi-kt$ or $|\xi'-kt|\sim |\xi-kt|$, we obtain that
		\begin{align*}
			&|(m(t,k,\xi)-m(t,k,\xi'))\hat{V}_0^2(t,\xi-\xi')(\xi'-kt)\hat{q}(t,k,\xi')|
			\\
			=&~|(m(t,k,\xi)-m(t,k,\xi'))(\xi-\xi')^{-1}(\xi'-kt)| |(\xi-\xi')\hat{V}_0^2(t,\xi-\xi')| |\hat{q}(t,k,\xi')|
			\nonumber\\
			\lesssim&~ \big(\|m\|_{L^\infty_{k,\xi}}+\|(\xi-kt)\nabla_{\xi} m\|_{L^\infty_{k,\xi}}\big)|(\xi-\xi')\hat{V}_0^2(t,\xi-\xi')| |\hat{q}(t,k,\xi')|.
			\nonumber
		\end{align*}
		Substituting these three estimates into \eqref{est-m}, we find that
		\begin{align*}
			\|m \F[\V_0\cdot\nabla q]-\hat{\V}_0*(m\F[\nabla q])\|_{L^2}
			\lesssim \|\F[\nabla\V_0]*\hat{q}\|_{L^2}
			\lesssim\|\pa_{y_1} \V_0\|_{L^\infty}\|q\|_{L^2}.
		\end{align*}
		Thus, we conclude that estimate \eqref{est-R1} holds.
	\end{proof}
	
	\medskip		
	
	\begin{lemm}\label{sec2:lem-com-2}
		For any $s\in\N^+$, $g_1\in L^\infty(\T\times\R)$, $\nabla^{s-1}g_1\in L^2(\T\times\R)$, $g_2\in L^\infty(\T\times\R)$, $\widetilde{\nabla}^{s-1}g_2\in L^2(\T\times\R)$, and $\nabla \V\in L^{\infty}(\T\times\R)$, $\nabla^s \V\in L^{2}(\T\times\R)$, $\widetilde{\nabla}^s \V\in L^{2}(\T\times\R)$, we have
		\begin{align}
			&\|[\nabla^{s},\V] g_1\|_{L^2}\lesssim \|\nabla \V\|_{L^\infty}\|\nabla^{s-1}  g_1\|_{L^2}+\|\nabla^s \V\|_{L^2}\| g_1\|_{L^\infty},\label{est-com-1}\\
			&\|[\widetilde{\nabla}^{s},\V] g_2\|_{L^2}\lesssim \|\widetilde{\nabla} \V\|_{L^\infty}\|\widetilde{\nabla}^{s-1}  g_2\|_{L^2}+\|\widetilde{\nabla}^s \V\|_{L^2}\|g_2\|_{L^\infty}.\label{est-com-2}
		\end{align}
	\end{lemm}
	\begin{proof}
		The first commutator estimate \eqref{est-com-1} is standard; its proof is omitted here and the reader may consult \cite{Majda-Bertozzi-2002} for details.
		Following the same argument used to derive \eqref{est-Ri-trans}, the second commutator estimate \eqref{est-com-2} follows directly from \eqref{est-com-1}.
	\end{proof}

	\section{Construction of the energy functional}\label{sect:energy fun}
	In this section we construct the energy functional and establish the {\it a priori} estimates for smooth local solutions of system \eqref{eq_n-d-w-1}.  These estimates will allow us to extend the local solution  $(N,D,W)$ to a global one and play a fundamental role in the proof of Theorem \ref{theo2}.  We begin by taking the Fourier transform in both spatial variables $(x_1,y_1)$; thus, for any function $q(t,x_1,y_1)$, we set
	\begin{align*}
		\hat{q}^k(t,\xi)
		\tri \f{1}{(2\pi)^2} \int_{\T\times\R} q(t,x_1,y_1)e^{-ikx_1-i\xi y_1}\,dx_1 dy_1,\quad \forall ~k\in\Z.
	\end{align*} 
	Taking the Fourier transform of the system \eqref{eq_n-d-w-1},  we obtain
	\begin{align}
		\label{eq_n-d-w-2}
		\left\{\begin{array}{l}
			\pa_{t} \hat{N}^k+\hat{D}^k=\hat{\mathscr{F}}_1^k,\quad t\in\R^+,~~\xi\in\R,~~k\in \Z,\\[1ex]
			\displaystyle
			\pa_{t} \hat{D}^k-\f{\pa_{t}p}{p}\hat{D}^k+2\f{k^2}{p}(\hat{W}^k-\hat{N}^k+\mu\hat{D}^k)-p\hat{N}^k +\nu p\hat{D}^k=\hat{\mathscr{F}}_2^k,\\[1ex]
			\displaystyle
			\pa_{t} \hat{W}^k+\mu p\hat{W}^k-\mu(\mu+\mu') p\hat{D}^k+\mu\f{\pa_{t}p}{p}\hat{D}^k-2\mu \f{k^2}{p}(\hat{W}^k-\hat{N}^k+\mu\hat{D}^k)=\hat{\mathscr{F}}_3^k,\\[1ex]
			\displaystyle
			(\hat{N}^k,\hat{D}^k,\hat{W}^k)|_{t=0}=(\hat{n}_{in}^k,\hat{d}_{in}^k,\hat{w}_{in}^k)(\xi),
		\end{array}\right.
	\end{align}
	where $p$ is defined by
	\begin{align}\label{est-p}
		p\tri k^2+(\xi-kt)^2,\quad \pa_{t}p=-2k(\xi-kt).
	\end{align}
	In the following analysis, we concentrate on system \eqref{eq_n-d-w-2} and derive the {\it a priori} estimates.  When constructing the energy estimates, we will apply the operators $\nabla$ and $\widetilde{\nabla}$ to the system \eqref{eq_n-d-w-2} and refer to $\nabla$ as ``good derivative" and to $\widetilde{\nabla}$ as ``bad derivative".
	
	Therefore, we introduce the notations for the lower-order derivative of the compressible part $(N, \widetilde{\Delta}^{-\f12}D)$: for any $0\leq s\leq 3$, $0\leq j\leq s$,
	\begin{align}
		&\hat{N}_{j,s-j}^k\tri m_1^{-1}m_2^{-1}\varphi^{-\f14}<k,\xi>^{s-j}p^{\f{j}{2}}\hat{N}^k,\label{def_N-j}\\ 
		&\hat{  \mathcal{U} }_{j,s-j}^k\tri m_1^{-1}m_2^{-1}\varphi^{-\f14}<k,\xi>^{s-j}p^{\f{j-1}{2}}\hat{D}^k,\label{def_D-j}
	\end{align}
	where we use the notation $\UU$ represents the compressible part of the velocity, namely $\UU=\widetilde{\Delta}^{-\f12}D=\widetilde{\Delta}^{-\f12}\widetilde{\dive} \V$ in \eqref{def_D-j}.		
	
	For the incompressible part $W$, we define: for any $0\leq s\leq 3$, $0\leq j\leq s$,
	\begin{align}	
		\hat{W}_{j,s-j}^k\tri m_1^{-1}m_2^{-1}<k,\xi>^{s-j}p^{\f{j}{2}}\hat{W}^k.\label{def_W-j+1}
	\end{align}
	To close the energy, we need the higher-order derivative of the compressible part $(N, \widetilde{\Delta}^{-\f12}D)$, namely for any $0\leq s\leq 3$, $0\leq j\leq s$,		
	\begin{align}		
		&\hat{N}_{j+1,s-j}^k\tri  <k,\xi>^{s-j}p^{\f{j+1}{2}}\hat{N}^k,\label{def_N-j+1}\\
		&\hat{D}_{j,s-j}^k\tri  <k,\xi>^{s-j}p^{\f{j}{2}}\hat{D}^k. \label{def_D-j+1}
	\end{align}

	\begin{rema}
		
		We see that for any function $q$, the notation $q_{j,l}$ represents the $j$-th order ``bad derivative" and the $l$-th order ``good derivative" of $q$.
	\end{rema}

	\begin{rema}
		
		For the zero mode of the solution, it is not difficult to check that
		\begin{align}\label{est-P0-N-D-W}
			(\hat{N}_{j,s-j}^0,\hat{\UU}_{j,s-j}^0,\hat{W}_{j,s-j}^0,\hat{N}_{j+1,s-j}^0,\hat{D}_{j,s-j}^0)=(\hat{N}_{0,s}^0, \hat{\UU}_{0,s}^0,\hat{W}_{0,s}^0,\hat{N}_{1,s}^0,\hat{D}_{0,s}^0),\quad \forall~0\leq j\leq s.
		\end{align}
		
	\end{rema}

	Now we are in a position to introduce the energy functional, which is decomposed into compressible and incompressible parts.  For the compressible part, for any $j=0,\cdots,3$, we define
	\begin{align*}
		\E_{j,3-j}^{com,l}(t)\tri&\sum_{s=j}^{3}\sum_{k\in\Z}\int_{\R}\left(|\hat{N}_{j,s-j}^k|^2+|\hat{\UU}_{j,s-j}^k|^2\right) \,d\xi,
		\\
		\D_{j,3-j}^{com,l}(t)\tri&~
		\mu \sum_{s=j}^{3}\sum_{k\in\Z}\int_{\R}p|\hat{\UU}_{j,s-j}^k|^2\,d\xi
		+(3-j)\mu\sum_{s=j+1}^{3}\int_{\R}|\hat{N}_{j,s-j}^0|^2\,d\xi
		\\
		&+\sum_{s=j}^{3}\sum_{k\in\Z\backslash\{0\}}\int_{\R}\left(\mu^{\f13}+A\f{k^2}{p}\right)\left(|\hat{N}_{j,s-j}^k|^2+|\hat{\UU}_{j,s-j}^k|^2\right) \,d\xi.
	\end{align*}	
	For the incompressible part, for any $j=0,\cdots,3$, we define				
	\begin{align*}			
		\E_{j,3-j}^{in}(t)\tri& \sum_{s=j}^{3}\sum_{k\in\Z}\int_{\R}|\hat{W}_{j,s-j}^k|^2 \,d\xi,
		\\
		\D_{j,3-j}^{in}(t)\tri&~
		\mu \sum_{s=j}^{3}\sum_{k\in\Z}\int_{\R}p|\hat{W}_{j,s-j}^k|^2\,d\xi
		+\sum_{s=j}^{3}\sum_{k\in\Z\backslash\{0\}}\int_{\R}\left(\mu^{\f13}+A\f{k^2}{p}\right)|\hat{W}_{j,s-j}^k|^2 \,d\xi.
	\end{align*}	
	Moreover, we also need the high-order energy functional for compressible part: for $j=0,\cdots,3$,
	\begin{align*}	
		\E_{j,3-j}^{com}(t)\tri&\sum_{s=j}^{3} \sum_{k\in\Z}\int_{\R}\left(|\hat{N}_{j+1,s-j}^k|^2+|\hat{D}_{j,s-j}^k|^2\right) \,d\xi,
		\\
		\D_{j,3-j}^{com}(t)\tri&~
		\mu \sum_{s=j}^{3}\sum_{k\in\Z}\int_{\R}p|\hat{D}_{j,s-j}^k|^2\,d\xi
		+(3-j)\mu\sum_{s=j+1}^{3}\int_{\R}|\hat{N}_{j+1,s-j}^0|^2\,d\xi
		\\
		&+\mu^{\f13}\sum_{s=j}^{3}\sum_{k\in\Z\backslash\{0\}}\int_{\R}|\hat{N}_{j+1,s-j}^k|^2 \,d\xi.
	\end{align*}
	Here $A$ is a large yet fixed constant determined later.
    We note that the enhanced dissipation of $\hat{D}_{j,s-j}^k$ can be obtained via the dissipation of $\hat{\UU}_{j,s-j}^k$ as follows:
    \begin{align*}
        \mu^{\f13}\sum_{s=j}^{3}\sum_{k\in\Z\backslash\{0\}}\int_{\R}|\hat{D}_{j,s-j}^k|^2\,d\xi
        \lesssim \sum_{s=j}^{3}\sum_{k\in\Z\backslash\{0\}}\int_{\R}|p^{\f12}\hat{\UU}_{j,s-j}^k|^2\,d\xi.
    \end{align*}

	Then it is easy to find that Theorems \ref{theo1}-\ref{theo2} are a direct corollary of the following energy estimates.
	\begin{prop}\label{prop-1}
		Assume that $(N,\widetilde{\Lambda}^{-1}D)\in L^\infty([0,T];H^4(\T\times\R))$ and $W\in L^\infty([0,T];H^3(\T\times\R))$ are the solution for system \eqref{eq_n-d-w-1} such that, for any $\ep_0>0$ and $t\in[0,T]$,
		\begin{align}\label{priori assumption}
			\sup_{\tau\in[0,t]}\E(\tau)+\int_0^t \D(\tau)\,d\tau\leq \ep_0^2\mu^{2},
		\end{align}
		where $\E(t)$ and $\D(t)$, for some positive constants $c_j (j=0,\cdots,3)$,  are defined by
		\begin{align}
			\E(t)\tri&\sum_{j=0}^3c_j\mu^{\f{2j}{3}}\left(\E_{j,3-j}^{com,l}(t)+\tilde{\delta}_1\E_{j,3-j}^{in}(t)+\tilde{\delta}_2\mu\E_{j,3-j}^{com}(t)\right),\label{def_E}\\
			\D(t)\tri&\sum_{j=0}^3c_j\mu^{\f{2j}{3}}\left(\D_{j,3-j}^{com,l}(t)+\tilde{\delta}_1\D_{j,3-j}^{in}(t)+\tilde{\delta}_2\mu\D_{j,3-j}^{com}(t)\right)
			.\label{def_D}
		\end{align}
		Then for any $t\in[0,T]$, we have
		\begin{align}\label{priori result}
			\sup_{\tau\in[0,t]}\E(\tau)+\int_0^t \D(\tau)\,d\tau\lesssim \E(0)+\ep_0^3\mu^{2}.
		\end{align}
	\end{prop}

	At the end of this section, using the definitions \eqref{def_N-j}–\eqref{def_D-j+1} and the {\it a priori} assumptions \eqref{priori assumption}, we derive several estimates for $(N, \V)$ that will be used repeatedly in the energy estimates for the proof of Proposition~\ref{prop-1}.
	
	\begin{lemm}\label{lem-est-NDW}
		Under the assumptions in Proposition \ref{prop-1}, there holds that for any $j=0,\cdots,3$,	\smallskip
		
		$\bullet$ $(N, \widetilde{\Lambda}^{-1}D)$ satisfies 
		\begin{align*}
			&~
			\|\PP_{\not=}\widetilde{\Lambda}^{j}(N,\widetilde{\Lambda}^{-1}D)\|_{L^\infty_tH^{3-j}}^2
			\lesssim \mu^{-\f13}\E_{j,3-j}^{com,l}(t),\\
			&~\mu\|\PP_{\not=}\widetilde{\Lambda}^{j}D\|_{L^2_tH^{3-j}}^2
			+\mu^{\f13}\|\PP_{\not=}\widetilde{\Lambda}^{j}(N,\widetilde{\Lambda}^{-1}D)\|_{L^2_tH^{3-j}}^2
			\lesssim \mu^{-\f13}\int_0^t\D_{j,3-j}^{com,l}(\tau)\,d\tau, \\
			&~\|\PP_{0}(N,|\pa_{y_1}|^{-1}D)\|_{L^\infty_tH^3}^2
			\lesssim \E_{0,3}^{com,l}(t), \\
			&~
			\mu\|\PP_{0}\pa_{y_1}N\|_{L^2_tH^2}^2
			+\mu\|\PP_{0}D\|_{L^2_tH^3}^2
			\lesssim \int_0^t\D_{0,3}^{com,l}(\tau)\,d\tau.
		\end{align*}	
		
		$\bullet$ $W$ satisfies
		\begin{align*}
			&~\|\widetilde{\Lambda}^{j}W\|_{L^\infty_tH^{3-j}}^2
			\lesssim
			\E_{j,3-j}^{in}(t), \\
			&~\mu\|\widetilde{\Lambda}^{j+1}W\|_{L^2_tH^{3-j}}^2+\mu^{\f13}\|\PP_{\not=}\widetilde{\Lambda}^jW\|_{L^2_tH^{3-j}}^2
			\lesssim \int_0^t\D_{j,3-j}^{in}(\tau)\,d\tau.
		\end{align*}
		
		$\bullet$ $(\widetilde{\Lambda} N, D)$ satisfies
		\begin{align*}
			&~\|\PP_{\not=}\widetilde{\Lambda}^j(\widetilde{\Lambda}N,D)\|_{L^\infty_tH^{3-j}}^2
			\lesssim  \E_{j,3-j}^{com}(t), \\
			&~\mu\|\PP_{\not=}\widetilde{\Lambda}^{j+1}D\|_{L^2_tH^{3-j}}^2
			+\mu^{\f13}\|\PP_{\not=}\widetilde{\Lambda}^j(\widetilde{\nabla}N,D)\|_{L^2_tH^{3-j}}^2
			\lesssim \int_0^t\big(\mu^{-1}\D_{j,3-j}^{com,l}+\D_{j,3-j}^{com}\big)(\tau)\,d\tau, \\
			&~\|\PP_{0}(\pa_{y_1}N,D)\|_{L^\infty_tH^{3}}^2
			\lesssim \E_{0,3}^{com}(t), \\
			&~
			\mu\|\PP_{0}\pa_{y_1}^2N\|_{L^2_tH^2}^2+\mu\|\PP_{0}\pa_{y_1}D\|_{L^2_tH^{3}}^2
			\lesssim \int_0^t\D_{0,3}^{com}(\tau)\,d\tau. 
		\end{align*}
		
		$\bullet$ $\V$	 satisfies	
		\begin{align*}
			&~\|\PP_{\not=}\widetilde{\Lambda}^{j} \V\|_{L^\infty_tH^{3-j}}^2
			\lesssim
			\mu^{-\f13}\left(\E_{j,3-j}^{com,l}+\E_{j,3-j}^{in}\right)(t), 
			\\
			&~\mu\|\PP_{\not=}\widetilde{\Lambda}^{j+1}\V\|_{L^2_tH^{3-j}}^2
			+\mu^{\f13}\|\PP_{\not=}\widetilde{\Lambda}^{j}\V\|_{L^2_tH^{3-j}}^2
			\lesssim \mu^{-\f13}\int_0^t\left(\D_{j,3-j}^{com,l}+\D_{j,3-j}^{in}\right)(\tau)\,d\tau, 
			\\
			&~\|\PP_{\not=}\widetilde{\Lambda}^{j+1} \V\|_{L^\infty_tH^{3-j}}^2
			\lesssim \left(\E_{j,3-j}^{com}+\E_{j,3-j}^{in}\right)(t), 
			\\
			&~\mu\|\PP_{\not=}\widetilde{\Lambda}^{j+2}\V\|_{L^2_tH^{3-j}}^2
			+\mu^{\f13}\|\PP_{\not=}\widetilde{\Lambda}^{j+1}\V\|_{L^2_tH^{3-j}}^2
			\lesssim \int_0^t\left(\mu^{-1}\D_{j,3-j}^{com,l}+\D_{j,3-j}^{com}+\D_{j,3-j}^{in}\right)(\tau)\,d\tau.
		\end{align*}		
		
		$\bullet$ $(\PP_{0}V^1, \PP_{0}V^2)$ satisfies	
		\begin{align*}
			&~\|\PP_{0} \pa_{y_1}V^1\|_{L^\infty_tH^{2}}^2+\|\PP_{0} V^2\|_{L^\infty_tH^{3}}^2
			\lesssim \left(\E_{0,3}^{com,l}+\E_{0,3}^{in}\right)(t), 
			\\
			&~
			\mu\|\PP_{0}\pa_{y_1}^2V^1\|_{L^2_tH^{2}}^2+\mu\|\PP_{0}\pa_{y_1}V^2\|_{L^2_tH^{3}}^2
			\lesssim \int_0^t\left(\D_{0,3}^{com,l}+\D_{0,3}^{in}\right)(\tau)\,d\tau, 
			\\
			&~\|\PP_{0}\pa_{y_1} V^1\|_{L^\infty_tH^{3}}^2+\|\PP_{0}\pa_{y_1} V^2\|_{L^\infty_tH^{3}}^2
			\lesssim
			\left(\E_{0,3}^{com}+\E_{0,3}^{in}\right)(t), 
			\\
			&~
			\mu\|\PP_{0}\pa_{y_1}^2V^1\|_{L^2_tH^{3}}^2+
			\mu\|\PP_{0}\pa_{y_1}^2V^2\|_{L^2_tH^{3}}^2
			\lesssim \int_0^t\left(\D_{0,3}^{com}+\D_{0,3}^{in}\right)(\tau)\,d\tau.
		\end{align*}
	\end{lemm}
	\begin{proof}
		Recalling the definitions in \eqref{def_N-j}–\eqref{def_W-j+1} and invoking \eqref{est-varphi-1} together with the {\it a priori} assumptions \eqref{priori assumption}, we immediately obtain the desired estimates for $(N,\widetilde{\Lambda}^{-1}D)$ and $W$.

		Thanks to $|\PP_{\not=}\widetilde{\Lambda}^{j-1}q|\leq |\PP_{\not=}\widetilde{\Lambda}^{j}q|$, we obtain
		\begin{align*}
			\|\PP_{\not=}\widetilde{\Lambda}^{j} \V\|_{L^\infty_tH^{3-j}}^2
			\lesssim&\|\PP_{\not=}\widetilde{\Lambda}^{j-1}D\|_{L^\infty_tH^{3-j}}^2
			+\|\PP_{\not=}\widetilde{\Lambda}^{j-1}(W-N+\mu D)\|_{L^\infty_tH^{3-j}}^2
			\\
			\lesssim&\|\PP_{\not=}\widetilde{\Lambda}^{j-1}D\|_{L^\infty_tH^{3-j}}^2
			+\|\PP_{\not=}\widetilde{\Lambda}^{j}(W,N)\|_{L^\infty_tH^{3-j}}^2
			\\
			\lesssim& \mu^{-\f13}\left(\E_{j,3-j}^{com,l}(t)+\E_{j,3-j}^{in}(t)\right),
		\end{align*}
		which implies the estimates of $\V$.

		Thanks to
		\begin{align*}
			\PP_{0}\widetilde{\Lambda}^{-1}D=\PP_{0}|\pa_{y_1}|^{-1}\pa_{y_1} V^2,\quad \PP_{0}W=\PP_{0}\pa_{y_1} V^1,
		\end{align*}
		we find that
		\begin{align*} 
			\|\PP_{0} \pa_{y_1}V^1\|_{L^\infty_tH^{2}}^2
			\lesssim \|\PP_{0} (W-N+\mu D)\|_{L^\infty_tH^{2}}^2
			\lesssim\E_{0,3}^{com,l}(t)+\E_{0,3}^{in}(t).
		\end{align*}
		For $\PP_{0} V^2$, we notice that
		\begin{align*} 
			\|\PP_{0} V^2\|_{L^\infty_tH^{3}}^2\lesssim \|\PP_{0}|\pa_{y_1}|^{-1} D\|_{L^\infty_tH^{3}}^2
			\lesssim \E_{0,3}^{com,l}(t).
		\end{align*}
		The remaining estimates for $(\PP_{0}V^1, \PP_{0}V^2)$ follow from the same argument; we leave the details to the reader.			
	\end{proof}

	\section{Energy estimate of compressible part}
	
	In this section, we establish the {\it a priori} energy estimates for the lower-order derivatives of the compressible part (the density and the divergence) of the solution.  The argument follows a standard bootstrap procedure.  
	First, we assume
	\begin{align}\label{ass:density}
		\f12\leq N+1\leq 2,\quad \textrm{for all}\quad t\in[0,\infty). 
	\end{align}
	Hence, we immediately have
	\begin{align}\label{est-fn-gn}
		|(f(N),g(N))|\lesssim |N|,\quad |(f^{(s)}(N),g^{(s)}(N))|\lesssim 1, \quad \forall s\in\N^+,
	\end{align}
	where $f(N)$ and $g(N)$ are nonlinear functions of $N$ defined by \eqref{def-f-g}.

	The main result of this section is stated as follows.			
	\begin{prop}\label{prop-2}
		Under the assumptions of Theorem \ref{theo2} and \eqref{ass:density}, we have
		\begin{align}
			\E_{0,3}^{com,l}(t)+\int_0^t \D_{0,3}^{com,l}(\tau)\,d\tau
			\lesssim
			\E_{0,3}^{com,l}(0)
			+\f{1}{A}\int_0^t\D_{0,3}^{in}(\tau)\,d\tau
			+\mu^{-1}|\E(t)|^{\f12}
			\int_0^t \D (\tau)\,d\tau,
			\label{est-0+1}
		\end{align}
		where $A$ is a large but fixed constant, which will be determined later.
	\end{prop}
	
	Before proceeding to the proof of Proposition \ref{prop-2}, we first state two auxiliary propositions — these are the key ingredients in establishing Proposition \ref{prop-2}.
	It is well-known that the compressible Navier--Stokes equations satisfy an energy equality in $L^2$.  Motivated by this observation and because the $L^2$  estimate for $\PP_{0}V^1$ is unavailable, we derive an independent $L^2$ estimate for $(\PP_{0}N,\PP_{0}V^2)$ with the aim of eliminating the undesired terms containing  $\PP_{0}V^1$.
	To this end, we introduce the following definitions:
	\begin{align}
		\widetilde{\E}_1(t)\tri&\int_{\R}\left(|\hat{N}_{0,0}^0|^2+|\hat{\UU}_{0,0}^0|^2
		\right)\,d\xi,
		\label{def-E-1}\\
		\widetilde{\E}_2(t)\tri&\sum_{k\in\Z\backslash\{0\}}\int_{\R}\left[|\hat{N}_{0,0}^k|^2+|\hat{\UU}_{0,0}^k|^2+\f12\f{\pa_{t}p}{p^{\f32}}Re\left(\hat{N}_{0,0}^k\bar{\hat{\UU}}_{0,0}^k\right)\right]\,d\xi
		\label{def-E-2}
		\\
		&+\sum_{s=1}^{3}\sum_{k\in\Z}\int_{\R}\left[|\hat{N}_{0,s}^k|^2+|\hat{\UU}_{0,s}^k|^2+\f12\f{\pa_{t}p}{p^{\f32}}Re\left(\hat{N}_{0,s}^k\bar{\hat{\UU}}_{0,s}^k\right)\right]\,d\xi
		\nonumber\\
		&-2\delta_1\mu^{\f13}\sum_{s=0}^{3}\sum_{k\in\Z\backslash\{0\}}\int_{\R}p^{-\f12}Re\left(\hat{N}_{0,s}^k\bar{\hat{\UU}}_{0,s}^k\right)\,d\xi
		-2\delta_1\mu\sum_{s=1}^{3}\int_{\R}\left[|\xi|^{-1}Re\left(\hat{N}_{0,s}^0\bar{\hat{\UU}}_{0,s}^0\right)\right]\,d\xi,
		\nonumber
	\end{align}
	and
	\begin{align}		
		\widetilde{\D}_1(t)\tri&~\nu \int_{\R}|\xi|^2|\hat{\UU}_{0,0}^0|^2\,d\xi,
		\label{def-D-1}\\
		\widetilde{\D}_2(t)\tri&\sum_{s=1}^{3}\sum_{k\in\Z\backslash\{0\}}\int_{\R}\left[\f{\pa_{t}m_1}{m_1}+\f{\pa_{t}m_2}{m_2}+\f14\left(\f{\pa_{t}\varphi}{\varphi}-\f{\pa_{t}p}{p}\right)+\delta_1\mu^{\f13}\left(1+2\f{k^2}{p^2}\right)\right]|\hat{N}_{0,s}^k|^2\,d\xi
		\label{def-D-2}
		\\
		&+\delta_1\mu\sum_{s=1}^{3}\int_{\R}|\hat{N}_{0,s}^0|^2\,d\xi
		+\sum_{k\in\Z\backslash\{0\}}\int_{\R}\left[\f{\pa_{t}m_1}{m_1}+\f{\pa_{t}m_2}{m_2}+\f14\left(\f{\pa_{t}\varphi}{\varphi}-\f{\pa_{t}p}{p}\right)+\nu p\right]|\hat{\UU}_{0,0}^k|^2\,d\xi
		\nonumber\\
		&+\sum_{s=1}^{3}\sum_{k\in\Z}\int_{\R}\left[\f{\pa_{t}m_1}{m_1}+\f{\pa_{t}m_2}{m_2}+\f14\left(\f{\pa_{t}\varphi}{\varphi}-\f{\pa_{t}p}{p}\right)+\nu p\right]|\hat{\UU}_{0,s}^k|^2\,d\xi.
		\nonumber
	\end{align}
	Here $\widetilde{\E}_1(t)$ and $\widetilde{\D}_1(t)$ represent $L^2$-norm for the zero-mode of density and the compressible part of the velocity. 
	And $\widetilde{\E}_2(t)$ and $\widetilde{\D}_2(t)$ involve the remaining terms in $\E_{0,3}^{com,l}(t)$ and $\D_{0,3}^{com,l}(t)$, respectively. More precisely, we find, by some direct computations, that
	\begin{align*}
		\E_{0,3}^{com,l}(t)\sim \widetilde{\E}_1(t)+\widetilde{\E}_2(t),\quad \D_{0,3}^{com,l}(t)\sim \widetilde{\D}_1(t)+\widetilde{\D}_2(t).
	\end{align*}
	Thus, Proposition \ref{prop-2} is a direct consequence of the following Propositions \ref{prop-2-1}-\ref{prop-2-2}.

	Firstly, we derive the $L^2$ estimate for the zero-mode of the density and the compressible part of the velocity.
	\begin{prop}\label{prop-2-1}
		Under the assumptions of Theorem \ref{theo2} and \eqref{ass:density}, we have
		\begin{align}
			\widetilde{\E}_1(t)+\int_0^t \widetilde{\D}_1(\tau)\,d\tau
			\lesssim
			\widetilde{\E}_1(0)+
			\mu^{-1}|\E(t)|^{\f12}
			\int_0^t \D (\tau)\,d\tau.
			\label{est-0}
		\end{align}
	\end{prop}
	In order to complete the proof of Proposition \ref{prop-2}, we need to estimate $\widetilde{\E}_{2}(t)$ and $\widetilde{\D}_2(t)$.
	
	\begin{prop}\label{prop-2-2}
		Under the assumptions of Theorem \ref{theo2} and \eqref{ass:density}, we have
		\begin{align}
			\widetilde{\E}_2(t)+\int_0^t \widetilde{\D}_2(\tau)\,d\tau
			\lesssim
			\widetilde{\E}_2(0)+
			\f{1}{A}\int_0^t\D_{0,3}^{in}(\tau)\,d\tau
			+\mu^{-1}|\E(t)|^{\f12}
			\int_0^t \D (\tau)\,d\tau,
			\label{est-1}
		\end{align}
		where $A$ is a large but fixed constant, which will be determined later.
	\end{prop}

	\subsection{Proof of Proposition \ref{prop-2-1}}
	In this subsection, we intend to establish the $L^2$ energy estimate for $\PP_{0}(N_{0,0},{\UU}_{0,0})$.
	To proceed, we exploit the structural properties of the governing equation. 
	First, notice that $\PP_{0}\UU_{0,0}=|\pa_{y_1}|^{-1}\pa_{y_1}\PP_{0}V^2$.
	Consequently, recalling that $(\rho,v^2)$ satisfies 
	\begin{align}\label{eq-pertur-1}
		\left\{\begin{array}{l}
			\partial_t \rho +y\partial_x \rho+\dive_{\x}(\rho\vv) = 0,\\[1ex]
			\rho(\partial_t v^2 + y\partial_xv^2+\vv\cdot\nabla_{\x} v^2)+\pa_{y} \mathcal{P}(\rho)=\mu \Delta_{\x} v^2 +(\mu + \mu')\pa_{y} \dive_{\x} \vv,
		\end{array}\right.
	\end{align}
	Recalling the setting \eqref{def-t1-x1-y1} and setting  
	\begin{align*}
		(\varrho,P)(t,x_1,y_1)\tri(\rho,\mathcal{P})(t,x_1+ty_1,y_1),
	\end{align*}
	then the system \eqref{eq-pertur-1} can be rewritten as
	\begin{align}\label{eq-pertur-2}
		\left\{\begin{array}{l}
			\partial_{t} \varrho +\widetilde{\dive}(\varrho \V) = 0,\\[1ex]
			\varrho(\partial_{t} V^2+\V\cdot\widetilde{\nabla} V^2)
			+(\pa_{y_1}-t\pa_{x_1}) P(\varrho)=\mu \widetilde{\Delta} V^2 +(\mu + \mu')(\pa_{y_1}-t\pa_{x_1}) \widetilde{\dive} \V.
		\end{array}\right.
	\end{align}
	We apply the operator $\PP_{0}$ to system \eqref{eq-pertur-2} and find that
	\begin{align}\label{eq-pertur-3}
		\left\{\begin{array}{l}
			\partial_{t} \varrho_0 + \pa_{y_1}(
			\varrho_0 V^2_0) = G_1,\\[1ex]
			\varrho_0(\partial_t V^2_0+V^2_0\pa_{y_1}V^2_0)
			+\pa_{y_1}P(\varrho_0)-\nu\pa_{y_1}^2V^2_0
			=G_2,
		\end{array}\right.
	\end{align}
	where $G_1$ and $G_2$ are defined by
	\begin{align*}
		G_1\tri& -\pa_{y_1}\PP_{0}(
		\varrho_{\not=} V^2_{\not=}),\\
		G_2\tri&-\varrho_0\PP_{0}(\V_{\not=}\cdot\widetilde{\nabla} V^2_{\not=})
		-\varrho_0\pa_{y_1}\varrho_0\left(\PP_{0}\left(\f{P'(\varrho)}{\varrho}\right)-\f{P'(\varrho_0)}{\varrho_0}\right)
		\\
		&
		-\varrho_0\PP_{0}\left(\PP_{\not=}\left(\f{P'(\varrho)}{\varrho}\right)(\pa_{y_1}-t\pa_{x_1})\varrho_{\not=} \right)
		+\nu \varrho_0 \left(\PP_{0}\left(\f{1}{\varrho}\right)-\f{1}{\varrho_0}\right)\pa_{y_1}^2V^2_0
		\\
		&
		+\varrho_0 \PP_{0}\left(\PP_{\not=}\left(\f{1}{\varrho}\right)\PP_{\not=}\big(\mu \widetilde{\Delta} V^2 +(\mu + \mu')(\pa_{y_1}-t\pa_{x_1}) \widetilde{\dive} \V\big)\right).
	\end{align*}
	We employ Sobolev inequality and the a \textit{priori} assumption \eqref{priori assumption} to obtain 
	\begin{align*}
		\varrho
		\sim\varrho_0
		\sim1.
	\end{align*}

	\begin{proof}[\textbf{Proof of Proposition \ref{prop-2-1}}]
		By standard $L^2$ energy estimate, we can deduce from \eqref{eq-pertur-3} that
		\begin{align}
			&\f{d}{dt}\int_{\R}\left(
			\GG(\varrho_0)+\f12\varrho_0|V^2_0|^2
			\right)\,dy_1+\nu \int_{\R}|\pa_{y_1}V^2_0|^2\,dy_1
			\label{est-energy-P0L2}\\
			=&\int_{\R}G_1 \int_1^{\varrho_0} \f{P(s)-P(1)}{s^2}\,ds dy_1
			+\int_{\R}G_1\f{P(\varrho_0)-P(1)}{\varrho_0}\,dy_1 \nonumber
			\\
			&+\f12 \int_{\R}G_1| V^2_0|^2\,dy_1
			+\int_{\R}G_2 V^2_0\,dy_1
			\tri \sum_{i=1}^{4}H_i,
			\nonumber
		\end{align} 
		where $\GG(\varrho_0)$ is defined by
		\begin{align}\label{def-G}
			\GG(\varrho_0)\tri \varrho_0\int_1^{\varrho_0} \f{P(s)-P(1)}{s^2}\,ds.
		\end{align}
		Clearly, it follows that
		\begin{align*}
			\PP_{\not=} \varrho=\PP_{\not=} N,\quad \pa_{y_1}\varrho=\pa_{y_1}N,
		\end{align*}
		which, together with integration by parts, Sobolev inequality and Lemma \ref{lem-est-NDW}, gives 
		\begin{align}
			\int_0^t|H_1|(\tau)\,d\tau
			\lesssim &\int_0^t\int_{\R}\left|\PP_{0}(
			\varrho_{\not=} V^2_{\not=})  \f{P(\varrho_0)-P(1)}{(\varrho_0)^2}\pa_{y_1}\varrho_0\right| \,dy_1d\tau
			\label{est-H1}\\
			\lesssim& \|N_{\not=}\|_{L^2_tL^2}\|V^2_{\not=}\|_{L^2_tL^2}
			\|N_0\|_{L^\infty_tL^\infty} \|\pa_{y_1}N_0\|_{L^\infty_tL^\infty}
			\nonumber\\
			\lesssim& ~\mu^{-\f23}\E(t)
			\int_0^t \D (\tau)\,d\tau.
			\nonumber
		\end{align}
		By Sobolev inequality and Lemma \ref{lem-est-NDW}, we have
		\begin{align}
			\int_0^t|H_2|(\tau)\,d\tau
			\lesssim 
			\|N_{\not=}\|_{L^2_tH^2}\|V^2_{\not=}\|_{L^2_tH^2}
			\|N_0\|_{L^\infty_tL^\infty}
			\lesssim \mu^{-\f23}
			|\E(t)|^{\f12}
			\int_0^t \D (\tau)\,d\tau.
			\label{est-H2}
		\end{align}
		Following a similar line of reasoning as in the proof above, it suffices to prove that
		\begin{align}
			\int_0^t|H_3|(\tau)\,d\tau
			\lesssim
			\|N_{\not=}\|_{L^2_tH^2}\|V^2_{\not=}\|_{L^2_tH^2}
			\|V^2_0\|_{L^\infty_tL^\infty}^2
			\lesssim \mu^{-\f23}
			\E(t) \int_{0}^{t}\D (\tau)\,d\tau.
			\label{est-H3}
		\end{align}
		Recalling the definition of $G_2$, $H_4$ can be divided into the following five terms:
		\begin{align}
			H_4
			=& \int_{\R}\varrho_0\PP_{0}(\V_{\not=}\cdot\widetilde{\nabla} V^2_{\not=})V^2_0\,dy_1
			-\int_{\R}\varrho_0 \pa_{y_1}\varrho_0\left(\PP_{0}\left(\f{P'(\varrho)}{\varrho}\right)-\f{P'(\varrho_0)}{\varrho_0}\right) V^2_0\,dy_1
			\label{est-H4-1}\\
			&
			-\int_{\R}\varrho_0\PP_{0}\left(\PP_{\not=}\left(\f{P'(\varrho)}{\varrho}\right)(\pa_{y_1}-t\pa_{x_1})\varrho_{\not=} \right)V^2_0\,dy_1
			+\nu\int_{\R}\varrho_0 \left(\PP_{0}\Big(\f{1}{\varrho}\Big)-\f{1}{\varrho_0}\right)\pa_{y_1}^2V^2_0V^2_0\,dy_1
			\nonumber\\
			&
			+\int_{\R}\varrho_0 \PP_{0}\left(\PP_{\not=}\left(\f{1}{\varrho}\right)\PP_{\not=}(\mu \widetilde{\Delta} V^2 +(\mu + \mu')(\pa_{y_1}-t\pa_{x_1}) \widetilde{\dive} \V)\right) V^2_0\,dy_1
			\nonumber\\
			\tri&\sum_{i=1}^5H_{4i}.
			\nonumber
		\end{align}
		One can easily verify that
		\begin{align}
			\int_0^t |H_{41}|(\tau)\,d\tau
			\lesssim
			\|\V_{\not=}\|_{L^2_tL^2}\|\widetilde{\nabla} V^2_{\not=}\|_{L^2_tL^2} \|V^2_0\|_{L^\infty_tL^\infty}
			\lesssim \mu^{-1}|\E(t)|^{\f12}
			\int_0^t \D (\tau)\,d\tau.
			\label{est-H41}
		\end{align}
		Since $\varrho=\varrho_0+N_{\not=}$,  for any smooth function $h(\rho)$, we have
		\begin{align*}
			h(\varrho)=h(\varrho_0)+h'(\varrho_0)N_{\not=}+O((N_{\not=})^2).
		\end{align*}
		This yields that
		\begin{align}
			\int_0^t |H_{42}|(\tau)\,d\tau
			\lesssim\|\pa_{y_1}N_0\|_{L^2_tL^2}\|N_{\not=}\|_{L^2_tL^2}\|V^2_0\|_{L^\infty_tL^\infty}
			\lesssim
			\mu^{-\f56}|\E(t)|^{\f12}
			\int_0^t \D (\tau)\,d\tau.
			\label{est-H42}
		\end{align}
		A short calculation proves that
		\begin{align}
			\int_0^t |H_{43}|(\tau)\,d\tau
			=&\int_0^t\int_{\R}\varrho_0\PP_{0}\left(\PP_{\not=}\left(\f{P'(\varrho)}{\varrho}-1\right)(\pa_{y_1}-t\pa_{x_1})\varrho_{\not=} \right)V^2_0\,dy_1d\tau
			\label{est-H43}\\
			\lesssim& \|N_{\not=}\|_{L^2_tL^2}\|\widetilde{\nabla} N_{\not=}\|_{L^2_tL^2}\|V^2_0\|_{L^\infty_tL^\infty}
			\nonumber\\
			\lesssim& \mu^{-1}|\E(t)|^{\f12}
			\int_0^t \D (\tau)\,d\tau.
			\nonumber
		\end{align}
		Since
		\begin{align*}
			\PP_{0}\left(\f{1}{\varrho}\right)-\f{1}{\varrho_0}
			=&\f{1}{2\pi}\int_{\T} \left(\f{1}{\varrho}-\f{1}{\varrho_0}\right)\,dx
			=\f{1}{2\pi}\int_{\T}\f{\varrho_0-\varrho}{\varrho\varrho_0}\,dx
			\\
			=&-\PP_{0}\left(\varrho_{\not=} \PP_{\not=}\left(\f{1}{\varrho}\right)\right)\f{1}{\varrho_0}
			=\PP_{0}\left(N_{\not=}\PP_{\not=}\left(\f{1}{\varrho}\right)\right)\f{1}{\varrho_0},
		\end{align*}
		and
		\begin{align}\label{est-Pnot=-varrho}
			\PP_{\not=} \left(\f{1}{\varrho}\right)=\PP_{\not=}\left(\f{1}{N+1}-1\right)=\PP_{\not=} \left(\f{N}{N+1}\right),
		\end{align}
		we then have
		\begin{align}
			\int_0^t |H_{44}|(\tau)\,d\tau
			\lesssim \mu\|N_{\not=}\|_{L^2_tL^2}\|N_{\not=}\|_{L^2_tL^\infty}
			\|\pa_{y_1}^2V^2_0\|_{L^\infty_tL^2}\|V^2_0\|_{L^\infty_tL^\infty}
			\lesssim \mu^{\f13}|\E(t)|^{\f12}
			\int_0^t \D (\tau)\,d\tau.
			\label{est-H44}
		\end{align}
		While it is easy to observe, by applying \eqref{est-Pnot=-varrho} and Lemma \ref{lem-est-NDW}, that
		\begin{align}
			\int_0^t |H_{45}|(\tau)\,d\tau
			\lesssim \mu \|N_{\not=}\|_{L^2_tL^2} \|\widetilde{\nabla}^2\V_{\not=}\|_{L^2_tL^2}
			\|V^2_0\|_{L^\infty_tL^\infty}
			\lesssim \mu^{-\f13}|\E(t)|^{\f12}
			\int_0^t \D (\tau)\,d\tau.
			\label{est-H45}
		\end{align}
		Substituting \eqref{est-H41}-\eqref{est-H43}, \eqref{est-H44} and  \eqref{est-H45} into \eqref{est-H4-1},  we obtain
		\begin{align}
			\int_0^t |H_{4}|(\tau)\,d\tau
			\lesssim
			\mu^{-1}|\E(t)|^{\f12}
			\int_0^t \D (\tau)\,d\tau.
			\label{est-H4-2}
		\end{align}
		Plugging \eqref{est-H1}-\eqref{est-H3} and \eqref{est-H4-2} into \eqref{est-energy-P0L2}, and applying the \textit{a priori} assumption \eqref{priori assumption}, we find that
		\begin{align*}
			&\int_{\R}\left(
			\GG(\varrho_0)+\f12\varrho_0|V^2_0|^2
			\right)\,dy_1+\nu \int_0^t\int_{\R}|\pa_{y_1}V^2_0|^2\,dy_1 d\tau
			\\
			\lesssim& 
			\int_{\R}\left(
			\GG(\PP_{0}\rho_{in})+\f12\PP_{0}\rho_{in}|\PP_{0}v_{in}^2|^2
			\right)\,dy
			+
			\mu^{-1}|\E(t)|^{\f12}
			\int_0^t \D (\tau)\,d\tau.
		\end{align*}
		This, along with \eqref{def-E-1}, \eqref{def-D-1} and \eqref{def-G}, implies that \eqref{est-0} holds directly.
	\end{proof}

	\subsection{Proof of Proposition \ref{prop-2-2}}

	Recalling the definition of $\hat{N}_{j,s-j}^k$ and $\hat{\UU}_{j,s-j}^k$ in \eqref{def_N-j} and \eqref{def_D-j} for $j=0$, it follows from the system \eqref{eq_n-d-w-2} that
	\begin{align}
		\pa_{t}\hat{N}_{0,s}^k=&-\left[\f{\pa_{t}m_1}{m_1}+\f{\pa_{t}m_2}{m_2}+\f14\left(\f{\pa_{t}\varphi}{\varphi}-\f{\pa_{t}p}{p}\right)\right]\hat{N}_{0,s}^k-\f14 \f{\pa_{t}p}{p}\hat{N}_{0,s}^k-p^{\f12}\hat{\UU}_{0,s}^k+\hat{\FFF}_{1,0,s}^k,
		\label{eq_N-0-s}\\
		\pa_{t}\hat{\UU}_{0,s}^k=&-\left[\f{\pa_{t}m_1}{m_1}+\f{\pa_{t}m_2}{m_2}+\f14\left(\f{\pa_{t}\varphi}{\varphi}-\f{\pa_{t}p}{p}\right)+\nu p\right]\hat{\UU}_{0,s}^k+\f14 \f{\pa_{t}p}{p}\hat{\UU}_{0,s}^k+\left(p^{\f12}+2\f{k^2}{p^{\f32}}\right)\hat{N}_{0,s}^k
		\label{eq_D-0-s}\\
		&-2\mu \f{k^2}{p}\hat{\UU}_{0,s}^k
		-2\f{k^2}{p^{\f32}}\varphi^{-\f14}\hat{W}_{0,s}^k
		+\hat{\FFF}_{2,0,s}^k,\nonumber
	\end{align}
	where $\hat{\FFF}_{1,0,s}^k$ and $\hat{\FFF}_{2,0,s}^k$ are defined by
	\begin{align}
		\hat{\FFF}_{1,0,s}^k\tri&m_1^{-1}m_2^{-1}\varphi^{-\f14}<k,\xi>^s\hat{\FFF}_1^k,
		\label{def-F1}\\
		\hat{\FFF}_{2,0,s}^k\tri&m_1^{-1}m_2^{-1}\varphi^{-\f14}<k,\xi>^sp^{-\f12}\hat{\FFF}_2^k.
		\label{def-F2}
	\end{align}
	Making use of equations \eqref{eq_N-0-s} and \eqref{eq_D-0-s}, we obtain
	\begin{align}
		&\f12\f{d}{dt}|(\hat{N}_{0,s}^k,\hat{\UU}_{0,s}^k)|^2
		+\left[\f{\pa_{t}m_1}{m_1}+\f{\pa_{t}m_2}{m_2}+\f14\left(\f{\pa_{t}\varphi}{\varphi}-\f{\pa_{t}p}{p}\right)\right]|\hat{N}_{0,s}^k|^2
		\label{est-energy-1}\\
		&+\left[\f{\pa_{t}m_1}{m_1}+\f{\pa_{t}m_2}{m_2}+\f14\left(\f{\pa_{t}\varphi}{\varphi}-\f{\pa_{t}p}{p}\right)+\nu p\right]|\hat{\UU}_{0,s}^k|^2
		\nonumber\\
		=&-\f14\f{\pa_{t}p}{p}\left(|\hat{N}_{0,s}^k|^2-|\hat{\UU}_{0,s}^k|^2\right)+2\f{k^2}{p^{\f32}}Re\left(\hat{N}_{0,s}^k\bar{\hat{\UU}}_{0,s}^k\right)-2\mu \f{k^2}{p}|\hat{\UU}_{0,s}^k|^2-2\f{k^2}{p^{\f32}}\varphi^{-\f14}Re\left(\hat{W}_{0,s}^k\bar{\hat{\UU}}_{0,s}^k\right)
		\nonumber\\
		&+Re\left(\hat{\FFF}_{1,0,s}^k\bar{\hat{N}}_{0,s}^k\right)+Re\left(\hat{\FFF}_{2,0,s}^k\bar{\hat{\UU}}_{0,s}^k\right).\nonumber
	\end{align}
	In order to cancel the first term presented on the right-hand side of \eqref{est-energy-1}, we find that
	\begin{align}
		&\f{d}{dt}\left[\f{\pa_{t}p}{p^{\f32}}Re\left(\hat{N}_{0,s}^k\bar{\hat{\UU}}_{0,s}^k\right)\right]
		\label{est-energy-2}\\
		=&~\f{\pa_{t}p}{p}\left(|\hat{N}_{0,s}^k|^2-|\hat{\UU}_{0,s}^k|^2\right)+2\f{k^2\pa_{t}p}{p^3}|\hat{N}_{0,s}^k|^2
		-\left[\f{\pa_{t}p}{p^{\f32}}\left(2\f{\pa_{t}m_1}{m_1}+2\f{\pa_{t}m_2}{m_2}+\f12\f{\pa_{t}\varphi}{\varphi}
		\right.\right.
		\nonumber\\
		&\left.\left.-\f12\f{\pa_{t}p}{p}+\nu p\right)+2\mu \f{k^2\pa_{t}p}{p^{\f52}}-2\f{k^2}{p^{\f32}}+\f32\f{(\pa_{t}p)^2}{p^{\f52}}
		\right]Re\left(\hat{N}_{0,s}^k\bar{\hat{\UU}}_{0,s}^k\right)
		\nonumber\\
		&-2\f{k^2\pa_{t}p}{p^3}\varphi^{-\f14}Re\left(\hat{N}_{0,s}^k\bar{\hat{W}}_{0,s}^k\right)+\f{\pa_{t}p}{p^{\f32}}\left[Re\left(\hat{\FFF}_{1,0,s}^k\bar{\hat{\UU}}_{0,s}^k\right)+Re\left(\hat{\FFF}_{2,0,s}^k\bar{\hat{N}}_{0,s}^k\right) \right].
		\nonumber
	\end{align}
	Multiplying equality \eqref{est-energy-2} by $\f14$, adding the resulting equality to \eqref{est-energy-1}, then performing the summation and integration, we infer that
	\begin{align}
		&\f12\f{d}{dt}\left\{
		\sum_{k\in\Z\backslash\{0\}}\int_{\R}\left[|\hat{N}_{0,0}^k|^2+|\hat{\UU}_{0,0}^k|^2+\f12\f{\pa_{t}p}{p^{\f32}}Re\left(\hat{N}_{0,0}^k\bar{\hat{\UU}}_{0,0}^k\right)\right] \,d\xi
		\right.
		\label{est-energy-4}\\
		&\left.+\sum_{s=1}^{3}\sum_{k\in\Z}\int_{\R}\left[|\hat{N}_{0,s}^k|^2+|\hat{\UU}_{0,s}^k|^2+\f12\f{\pa_{t}p}{p^{\f32}}Re\left(\hat{N}_{0,s}^k\bar{\hat{\UU}}_{0,s}^k\right)\right] \,d\xi
		\right\}
		\nonumber\\
		&+\sum_{s=0}^{3}\sum_{k\in\Z}\int_{\R}\left[\f{\pa_{t}m_1}{m_1}+\f{\pa_{t}m_2}{m_2}+\f14\left(\f{\pa_{t}\varphi}{\varphi}-\f{\pa_{t}p}{p}\right)\right]|\hat{N}_{0,s}^k|^2\,d\xi
		\nonumber\\
		&+\sum_{k\in\Z\backslash\{0\}}\int_{\R}\left[\f{\pa_{t}m_1}{m_1}+\f{\pa_{t}m_2}{m_2}+\f14\left(\f{\pa_{t}\varphi}{\varphi}-\f{\pa_{t}p}{p}\right)+\nu p\right]|\hat{\UU}_{0,0}^k|^2\,d\xi
		\nonumber\\
		&+\sum_{s=1}^{3}\sum_{k\in\Z}\int_{\R}\left[\f{\pa_{t}m_1}{m_1}+\f{\pa_{t}m_2}{m_2}+\f14\left(\f{\pa_{t}\varphi}{\varphi}-\f{\pa_{t}p}{p}\right)+\nu p\right]|\hat{\UU}_{0,s}^k|^2\,d\xi
		\nonumber\\
		=&\sum_{s=0}^{3}\sum_{k\in\Z\backslash\{0\}}\int_{\R}
		\left\{\f12\f{k^2\pa_{t}p}{p^3}|\hat{N}_{0,s}^k|^2-2\mu\f{k^2}{p}|\hat{\UU}_{0,s}^k|^2 + \left[\f52\f{k^2}{p^{\f32}}-\f{\mu}{2}\f{k^2\pa_{t}p}{p^{\f52}}
		\right.\right.
		\nonumber\\
		&\left.
		-\f{\pa_{t}p}{p^{\f32}}\left(\f12\f{\pa_{t}m_1}{m_1}+\f12\f{\pa_{t}m_2}{m_2}
		+\f18\f{\pa_{t}\varphi}{\varphi}+\f14\f{\pa_{t}p}{p}+\f14\nu p\right) \right]Re\left(\hat{N}_{0,s}^k\bar{\hat{\UU}}_{0,s}^k\right)
		\nonumber\\
		&\left.
		-\f12\f{k^2\pa_{t}p}{p^3}\varphi^{-\f14}Re\left(\hat{N}_{0,s}^k\bar{\hat{W}}_{0,s}^k\right)-2\f{k^2}{p^{\f32}}\varphi^{-\f14}Re\left(\hat{\UU}_{0,s}^k\bar{\hat{W}}_{0,s}^k\right)\right\}\,d\xi
		\nonumber\\
		&+\sum_{k\in\Z\backslash\{0\}}\int_{\R}\left[
		Re\left(\hat{\FFF}_{1,0,0}^k\bar{\hat{N}}_{0,0}^k\right)+Re\left(\hat{\FFF}_{2,0,0}^k\bar{\hat{\UU}}_{0,0}^k\right)\right]\,d\xi \nonumber\\
		&+\sum_{s=1}^{3}\sum_{k\in\Z}\int_{\R}
		\left[Re\left(\hat{\FFF}_{1,0,s}^k\bar{\hat{N}}_{0,s}^k\right)
		+Re\left(\hat{\FFF}_{2,0,s}^k\bar{\hat{\UU}}_{0,s}^k\right)\right]\,d\xi
		\nonumber\\
		&+\f14\sum_{s=0}^{3}\sum_{k\in\Z\backslash\{0\}}\int_{\R}
		\f{\pa_{t}p}{p^{\f32}}\left[Re\left(\hat{\FFF}_{1,0,s}^k\bar{\hat{\UU}}_{0,s}^k\right)+Re\left(\hat{\FFF}_{2,0,s}^k\bar{\hat{N}}_{0,s}^k\right) \right]\,d\xi.
		\nonumber
	\end{align}
	Next, we aim to recover the dissipation of $\hat{N}_{0,s}^k$ as
	\begin{align}
		&-\f{d}{dt}\left[p^{-\f12}Re\left(\hat{N}_{0,s}^k\bar{\hat{\UU}}_{0,s}^k\right)\right]
		+\left(1+2\f{k^2}{p^2}\right)|\hat{N}_{0,s}^k|^2
		\label{est-energy-5}\\
		=&~|\hat{\UU}_{0,s}^k|^2+\f{1}{p^{\f12}}\left(2\f{\pa_{t}m_1}{m_1}+2\f{\pa_{t}m_2}{m_2}+\f12\f{\pa_{t}\varphi}{\varphi}+2\mu \f{k^2}{p}+\nu p\right)Re\left(\hat{N}_{0,s}^k\bar{\hat{\UU}}_{0,s}^k\right)
		\nonumber\\
		&+2\f{k^2}{p^2}\varphi^{-\f14}Re\left(\hat{N}_{0,s}^k\bar{\hat{W}}_{0,s}^k\right)-\f{1}{p^{\f12}}\left[Re\left(\hat{\FFF}_{1,0,s}^k\bar{\hat{\UU}}_{0,s}^k\right)+Re\left(\hat{\FFF}_{2,0,s}^k\bar{\hat{N}}_{0,s}^k\right)\right].
		\nonumber
	\end{align}
	We note that the dissipation estimates of the density for zero and non-zero frequencies are different.
	Multiplying equality \eqref{est-energy-5} by $\delta_1\mu^{\f13}$, performing an integration with respect to $\xi$, then summing over $k\in\Z\backslash\{0\}$ and $s$ from $0$ to $3$, we obtain
	\begin{align}
		&-\delta_1\mu^{\f13}\f{d}{dt}\sum_{s=0}^{3}\sum_{k\in\Z\backslash\{0\}}\int_{\R}\left[p^{-\f12}Re\left(\hat{N}_{0,s}^k\bar{\hat{\UU}}_{0,s}^k\right)\right]\,d\xi
		+\delta_1\mu^{\f13}\sum_{s=0}^{3}\sum_{k\in\Z\backslash\{0\}}\int_{\R}\left(1+2\f{k^2}{p^2}\right)|\hat{N}_{0,s}^k|^2\,d\xi
		\label{est-energy-5-1}\\
		=&~\delta_1\mu^{\f13}\sum_{s=0}^{3}\sum_{k\in\Z\backslash\{0\}}\int_{\R}\left\{|\hat{\UU}_{0,s}^k|^2+\f{1}{p^{\f12}}\left(2\f{\pa_{t}m_1}{m_1}+2\f{\pa_{t}m_2}{m_2}+\f12\f{\pa_{t}\varphi}{\varphi}+2\mu \f{k^2}{p}+\nu p\right)Re\left(\hat{N}_{0,s}^k\bar{\hat{\UU}}_{0,s}^k\right)
		\right.
		\nonumber\\
		&\left.+2\f{k^2}{p^2}\varphi^{-\f14}Re\left(\hat{N}_{0,s}^k\bar{\hat{W}}_{0,s}^k\right)
		-\f{1}{p^{\f12}}\left[Re\left(\hat{\FFF}_{1,0,s}^k\bar{\hat{\UU}}_{0,s}^k\right)+Re\left(\hat{\FFF}_{2,0,s}^k\bar{\hat{N}}_{0,s}^k\right)\right]\right\}\,d\xi.
		\nonumber
	\end{align}
	On the other hand, we multiply equality \eqref{est-energy-5} by $\delta_1\mu$ for $k=0$, perform an integration with respect to $\xi$, then sum over $s$ from $1$ to $3$ to obtain 
	\begin{align}
		&-\delta_1\mu\f{d}{dt}\sum_{s=1}^{3}\int_{\R}\left[|\xi|^{-1}Re\left(\hat{N}_{0,s}^0\bar{\hat{\UU}}_{0,s}^0\right)\right]\,d\xi
		+\delta_1\mu\sum_{s=1}^{3}\int_{\R}|\hat{N}_{0,s}^0|^2\,d\xi
		\label{est-energy-5-2}\\
		=&~\delta_1\mu\sum_{s=1}^{3}\int_{\R}\left\{|\hat{\UU}_{0,s}^0|^2+\nu |\xi|Re\left(\hat{N}_{0,s}^0\bar{\hat{\UU}}_{0,s}^0\right)
		-|\xi|^{-1}\left[Re\left(\hat{\FFF}_{1,0,s}^0\bar{\hat{\UU}}_{0,s}^0\right)+Re\left(\hat{\FFF}_{2,0,s}^0\bar{\hat{N}}_{0,s}^0\right)\right]\right\}\,d\xi.
		\nonumber
	\end{align}
	Plugging \eqref{est-energy-4}, \eqref{est-energy-5-1} and \eqref{est-energy-5-2} together, we find that
	\begin{align}
		&\f12\f{d}{dt}\widetilde{\E}_2(t)+\widetilde{\D}_2(t)
		=\LL_1(t)+\sum_{i=1}^{5}I_i(t),\label{est-energy-6}
	\end{align}
	where $\widetilde{\E}_2(t)$ and $\widetilde{\D}_2(t)$ are defined in \eqref{def-E-2} and \eqref{def-D-2}, respectively. 
	And the linear term $\LL_1$ is defined by
	\begin{align}
		\LL_1(t)\tri&\sum_{s=0}^{3}\sum_{k\in\Z\backslash\{0\}}\int_{\R}\left\{\f12\f{k^2\pa_{t}p}{p^3}|\hat{N}_{0,s}^k|^2+\left(\delta_1\mu^{\f13}-2\mu\f{k^2}{p}\right)|\hat{\UU}_{0,s}^k|^2
		+\left[\f52\f{k^2}{p^{\f32}}-\f{\mu}{2}\f{k^2\pa_{t}p}{p^{\f52}}
		\right.\right.
		\label{def_LL1}\\
		&
		-\f{\pa_{t}p}{p^{\f32}}\left(\f12\f{\pa_{t}m_1}{m_1}
		+\f12\f{\pa_{t}m_2}{m_2}
		+\f18\f{\pa_{t}\varphi}{\varphi}+\f14\f{\pa_{t}p}{p}+\f14\nu p\right)
		+2\delta_1\mu^{\f43}\f{k^2}{p^{\f32}}
		\nonumber\\
		&\left.
		+2\delta_1\mu^{\f13}\f{1}{p^{\f12}}\left(
		\f{\pa_{t}m_1}{m_1}+\f{\pa_{t}m_2}{m_2}+\f14\f{\pa_{t}\varphi}{\varphi}
		\right)
		+\delta_1\mu^{\f13}\nu p^{\f12} \right]
		Re\left(\hat{N}_{0,s}^k\bar{\hat{\UU}}_{0,s}^k\right)
		\nonumber\\
		&\left.
		+\left(2\delta_1\mu^{\f13}\f{k^2}{p^2}-\f12\f{k^2\pa_{t}p}{p^3}\right)\varphi^{-\f14}Re\left(\hat{N}_{0,s}^k\bar{\hat{W}}_{0,s}^k\right)
		-2\f{k^2}{p^{\f32}}\varphi^{-\f14}Re\left(\hat{\UU}_{0,s}^k\bar{\hat{W}}_{0,s}^k\right)
		\right\}\,d\xi
		\nonumber\\
		&
		+\delta_1\mu\sum_{s=1}^{3}\int_{\R}\left\{|\hat{\UU}_{0,s}^0|^2+\nu |\xi|Re\left(\hat{N}_{0,s}^0\bar{\hat{\UU}}_{0,s}^0\right)\right\}\,d\xi,
		\nonumber
	\end{align}
	and the nonlinear terms $I_i (i=1,\cdots,5)$ are defined by
	\begin{align}
		I_1(t)\tri&\sum_{k\in\Z\backslash\{0\}}\int_{\R}Re\left(\hat{\FFF}_{1,0,0}^k\bar{\hat{N}}_{0,0}^k\right)
		+\sum_{s=1}^{3}\sum_{k\in\Z}\int_{\R} Re\left(\hat{\FFF}_{1,0,s}^k\bar{\hat{N}}_{0,s}^k\right)\,d\xi,
		\nonumber\\
		I_2(t)\tri& 
		\sum_{k\in\Z\backslash\{0\}}\int_{\R} Re\left(\hat{\FFF}_{2,0,0}^k\bar{\hat{\UU}}_{0,0}^k\right)\,d\xi
		+\sum_{s=1}^{3}\sum_{k\in\Z}\int_{\R}Re\left(\hat{\FFF}_{2,0,s}^k\bar{\hat{\UU}}_{0,s}^k\right)\,d\xi,
		\nonumber\\
		I_3(t)\tri&-\delta_1\mu\sum_{s=1}^{3}\int_{\R}|\xi|^{-1}\left[Re\left(\hat{\FFF}_{1,0,s}^0\bar{\hat{\UU}}_{0,s}^0\right)+Re\left(\hat{\FFF}_{2,0,s}^0\bar{\hat{N}}_{0,s}^0\right)\right]\,d\xi,
		\nonumber\\
		I_4(t)\tri&~\f14\sum_{s=0}^{3}\sum_{k\in\Z\backslash\{0\}}\int_{\R}\f{\pa_{t}p}{p^{\f32}}\left[Re\left(\hat{\FFF}_{1,0,s}^k\bar{\hat{\UU}}_{0,s}^k\right)
		+Re\left(\hat{\FFF}_{2,0,s}^k\bar{\hat{N}}_{0,s}^k\right) \right]\,d\xi,
		\nonumber\\
		I_5(t)\tri&-\delta_1\mu^{\f13}\sum_{s=0}^{3}\sum_{k\in\Z\backslash\{0\}}\int_{\R}
		\f{1}{p^{\f12}}\left[Re\left(\hat{\FFF}_{1,0,s}^k\bar{\hat{\UU}}_{0,s}^k\right)+Re\left(\hat{\FFF}_{2,0,s}^k\bar{\hat{N}}_{0,s}^k\right)\right]\,d\xi.
		\nonumber
	\end{align}
	We now deal with all the terms on the right-hand side of \eqref{est-energy-6}.
	First, we present the following lemma concerning the estimate of the linear term $\LL_1$.
	\begin{lemm}\label{lem-est-L1}
		It holds that
		\begin{align}\label{est-LL1}
			\LL_1(t) \leq 
			\f12\widetilde{\D}_2(t)+\f{C}{A}\D_{0,3}^{in}(t).
		\end{align}
	\end{lemm}
	\begin{proof}
		We begin by handling the linear terms with $k\not=0$ on the right-hand side of \eqref{def_LL1}.
		Recalling the definitions of $\varphi$ and $p$ from \eqref{def-phi} and \eqref{est-p}, respectively, we immediately obtain
		\begin{align}\label{est-p-varphi}
			|\pa_{t}p|\leq 2|k|p^{\f12},\quad |\pa_{t}p|\leq p,\quad \left|\f{\pa_{t}\varphi}{\varphi}\right| \leq \left|\f{\pa_{t}p}{p}\right|.
		\end{align}
		This implies that
		\begin{align}\label{est-L1-1}
			\left|\f12\f{k^2\pa_{t}p}{p^3}\right|   |\hat{N}_{0,s}^k|^2
			\lesssim \f{k^2}{p} |\hat{N}_{0,s}^k|^2.
		\end{align}
		It is not difficult to check that
		\begin{align}\label{est-L1-2}
			-2\mu\f{k^2}{p}  |\hat{\UU}_{0,s}^k|^2
			\leq 0.
		\end{align}
		With the help of \eqref{est-p-varphi} and Young's inequality, we get
		\begin{align}
			&\left|-\f{\nu}{4}\f{\pa_{t}p}{p^{\f12}} Re\left(\hat{N}_{0,s}^k\bar{\hat{\UU}}_{0,s}^k\right)\right|
			\lesssim \nu |k| |Re\left(\hat{N}_{0,s}^k\bar{\hat{\UU}}_{0,s}^k\right)| 
			\leq \widetilde{\ep}\nu p|\hat{\UU}_{0,s}^k|^2+ C_{\widetilde{\ep}} \f{k^2}{p}  |\hat{N}_{0,s}^k|^2,
			\label{est-L1-3}\\
			&\left|\f12\delta_1\mu^{\f13}\f{1}{p^{\f12}}\f{\pa_{t}\varphi}{\varphi}Re\left(\hat{N}_{0,s}^k\bar{\hat{\UU}}_{0,s}^k\right)\right|
			\lesssim \mu^{\f13}\left|\f{\pa_{t}p}{p^{\f32}}\right| \left|Re\left(\hat{N}_{0,s}^k\bar{\hat{\UU}}_{0,s}^k\right)\right|
			\leq \widetilde{\ep}\mu^{\f13} |\hat{\UU}_{0,s}^k|^2+ C_{\widetilde{\ep}} \f{k^2}{p}  |\hat{N}_{0,s}^k|^2.
			\label{est-L1-4}
		\end{align}
		Again thanks to \eqref{est-p-varphi}, we have
		\begin{align}
			&\left|\left(-\f12\f{\pa_{t}p}{p^{\f32}}+2\delta_1\mu^{\f13}\f{1}{p^{\f12}}\right)\left(\f{\pa_{t}m_1}{m_1}+\f{\pa_{t}m_2}{m_2}\right)Re\left(\hat{N}_{0,s}^k\bar{\hat{\UU}}_{0,s}^k\right)\right|
			\label{est-L1-5}\\
			\leq& \left(\f12+2\delta_1\mu^{\f13}\right) \left(\f{\pa_{t}m_1}{m_1}+\f{\pa_{t}m_2}{m_2}\right) \left|Re\left(\hat{N}_{0,s}^k\bar{\hat{\UU}}_{0,s}^k\right)\right|
			\nonumber\\
			\leq& ~\f12\left(\f12+2\delta_1\mu^{\f13}\right) \left(\f{\pa_{t}m_1}{m_1}+\f{\pa_{t}m_2}{m_2}\right)\left(|\hat{N}_{0,s}^k|^2+|\hat{\UU}_{0,s}^k|^2\right),
			\nonumber
		\end{align}
		and
		\begin{align}
			&\left|\left[\f52\f{k^2}{p^{\f32}}-\f{\mu}{2}\f{k^2\pa_{t}p}{p^{\f52}}-\f14\f{\pa_{t}p}{p^{\f32}}\left(\f12\f{\pa_{t}\varphi}{\varphi}+\f{\pa_{t}p}{p}\right)+2\delta_1\mu^{\f43}\f{k^2}{p^{\f32}}\right]Re\left(\hat{N}_{0,s}^k\bar{\hat{\UU}}_{0,s}^k\right)  \right|
			\label{est-L1-6}\\
			\lesssim& ~ \f{k^2}{p}\left(|\hat{N}_{0,s}^k|^2+|\hat{\UU}_{0,s}^k|^2\right).
			\nonumber
		\end{align}
		By virtue of \eqref{est-p-varphi} and Young's inequality, we get 
		\begin{align}
			\left|\delta_1\mu^{\f13}\nu p^{\f12} Re\left(\hat{N}_{0,s}^k\bar{\hat{\UU}}_{0,s}^k\right)\right|
			\lesssim\delta_1\mu^{\f43}\left(|\hat{N}_{0,s}^k|^2+p|\hat{\UU}_{0,s}^k|^2\right).
			\label{est-L1-7}
		\end{align} 
		The application of \eqref{est-varphi-1} and \eqref{est-p-varphi} yields
		\begin{align}
			&\left|\left(2\delta_1\mu^{\f13}\f{k^2}{p^2}-\f12\f{k^2\pa_{t}p}{p^3}\right)\varphi^{-\f14}Re\left(\hat{N}_{0,s}^k\bar{\hat{W}}_{0,s}^k\right)
			-2\f{k^2}{p^{\f32}}\varphi^{-\f14}Re\left(\hat{\UU}_{0,s}^k\bar{\hat{W}}_{0,s}^k\right)\right|  
			\label{est-L1-8}\\
			\lesssim&~ \f{k^2}{p}\left(|\hat{N}_{0,s}^k|^2+|\hat{\UU}_{0,s}^k|^2+|\hat{W}_{0,s}^k|^2\right).
			\nonumber
		\end{align}
		For $k=0$, it is easy to verify that
		\begin{align}
			&\left|\delta_1\mu\nu |\xi|Re\left(\hat{N}_{0,s}^0\bar{\hat{\UU}}_{0,s}^0\right)\right|
			\lesssim \delta_1\mu^2\left(|\hat{N}_{0,s}^0|^2+|\xi|^2|\hat{\UU}_{0,s}^0|^2\right).
			\label{est-L1-9}
		\end{align}
		Collecting all estimates \eqref{est-L1-1}-\eqref{est-L1-9} above together, then choosing $A$ defined in \eqref{def-m2} suitably large, $\delta_1$ and $\widetilde{\ep}$ suitably small, we obtain \eqref{est-LL1} immediately.
		We thereby complete the proof of this lemma.
		
	\end{proof}
	
	Furthermore, all the nonlinear terms $I_i(i=1,\cdots,5)$ will be handled according to the following five lemmas.
	\begin{lemm}\label{lem1-1}
		Under the assumptions of Proposition \ref{prop-2-2}, we have
		\begin{align}
			\int_0^tI_{1}(\tau)\,d\tau
			\lesssim  \mu^{-1}|\E(t)|^{\frac12}\int_0^t  \D(\tau) \, d\tau. 
			\label{est-I1-2}
		\end{align}
	\end{lemm}
	\begin{lemm}\label{lem1-2}
		Under the assumptions of Proposition \ref{prop-2-2}, we have
		\begin{align}
			\int_0^t I_2(\tau)\,d\tau \lesssim& \mu^{-1}|\E(t)|^{\frac12}\int_0^t  \D(\tau) \,d\tau.
			\label{est-I2-2} 
		\end{align}
	\end{lemm}
	
	\begin{lemm}\label{lem1-3}
		Under the assumptions of Proposition \ref{prop-2-2}, we have
		\begin{align}
			\int_0^t I_{3}(\tau)\,d\tau
			\lesssim 
			|\E(t)|^{\frac12}\int_0^t  \D(\tau) \, d\tau.
			\label{est-I3-2}
		\end{align}
	\end{lemm}
	
	\begin{lemm}\label{lem1-4}
		Under the assumptions of Proposition \ref{prop-2-2}, we have
		\begin{align}
			\int_0^t I_{4}(\tau) \,d\tau
			\lesssim&\mu^{-1}|\E(t)|^{\frac12}\int_0^t  \D(\tau) \, d\tau.
			\label{est-I5-2}
		\end{align}
	\end{lemm}
	
	\begin{lemm}\label{lem1-5}
		Under the assumptions of Proposition \ref{prop-2-2}, we have
		\begin{align}
			\int_0^tI_{5}(\tau)\,d\tau
			\lesssim &\mu^{-\f23}|\E(t)|^{\frac12}\int_0^t  \D(\tau) \, d\tau.
			\label{est-I6}
		\end{align}
	\end{lemm}
	
	We shall assume Lemmas \ref{lem1-1}–\ref{lem1-5} for the moment and proceed with the proof of Proposition \ref{prop-2-2}; the proofs of these lemmas will be provided at the end of this section.
	
	\begin{proof}[\textbf{Proof of Proposition \ref{prop-2-2}}]
		Integrating identity \eqref{est-energy-6} in time and inserting \eqref{est-LL1} together with \eqref{est-I1-2}–\eqref{est-I6} into the resulting expression, we obtain \eqref{est-1}.
	\end{proof}

	\subsubsection{Proof of Lemma \ref{lem1-1}}
	Remembering the definitions of $\RR_1$ and $\FFF_{1,0,s} (s=0,\cdots,3)$ in \eqref{def_R1-R2} and \eqref{def-F1}, respectively, it is not difficult to deduce from \eqref{def_wF1} that
	\begin{align}
		|I_1|
		\lesssim& 
		\left|\int_{\T\times\R}\PP_{\not=}\RR_1(\V\cdot\widetilde{\nabla} N) \PP_{\not=}N_{0,0} \,dx_1dy_1\right|
		+\sum_{s=1}^{3}
		\left|\int_{\T\times\R}\RR_1\nabla^s(\V\cdot\widetilde{\nabla} N) N_{0,s} \,dx_1dy_1\right|
		\label{est-I1-1}\\
		&+
		\left|\int_{\T\times\R}\left(\PP_{\not=}\RR_1(ND) \PP_{\not=}N_{0,0}+\sum_{s=1}^{3}\RR_1\nabla^s(ND)N_{0,s}\right) \,dx_1dy_1\right|
		\nonumber\\
		\tri& I_{11}+I_{12}+I_{13}.
		\nonumber
	\end{align}
	We recall the definitions of $N_{0,s}$ and $\RR_1$ in \eqref{def_N-j} and \eqref{def_R1-R2}, respectively, to obtain that
	\begin{align}\label{eq-I11}
		\PP_{\not=}\RR_1(\V\cdot\widetilde{\nabla} N)
		=
		\PP_{\not=}\left(\V\cdot\PP_{\not=} \widetilde{\nabla}N_{0,0}
		+[\RR_1,\V]\cdot \PP_{\not=}\widetilde{\nabla}N\right)
		+\RR_1\left(\PP_{\not=}V^2\PP_{0}\pa_{y_1}N\right).
	\end{align}
	Then, employing integration by parts, \eqref{eq-I11}, Lemmas \ref{sec2:lem-com-1-1} and \ref{lem-est-NDW}, we find that		
	\begin{align}
		\int_0^t I_{11}(\tau)\,d\tau
		\lesssim
		&
		\|\widetilde{\nabla} \V\|_{L^\infty_tH^2}
		\| \PP_{\not=} N_{0,0} \|_{L^2_tL^2}^2
		+\|\pa_{y_1} \V_0\|_{L^\infty_tH^2}
		\| \PP_{\not=} N \|_{L^2_tL^2}\| \PP_{\not=} N_{0,0} \|_{L^2_t L^2}
		\label{est-I10}
		\\
		&+\|\PP_{\not=} \V\|_{L^2_tH^2}
		\| \widetilde{\nabla} N \|_{L^\infty_tL^2}
		\| \PP_{\not=} N_{0,0} \|_{L^2_tL^2}
		\nonumber\\
		\lesssim&~ \mu^{-1}|\E(t)|^{\f12}
		\int_0^t \D(\tau)\,d\tau.
		\nonumber
	\end{align}
	In a same manner, we have
	\begin{align}
		I_{12}
		\lesssim &\sum_{s=1}^{3}
		\left|\int_{\T\times\R}|\widetilde{\dive}\V | |N_{0,s}|^2\,dx_1dy_1\right|
		+\sum_{s=1}^{3}
		\left|\int_{\T\times\R} [\RR_1,\V]\cdot \widetilde{\nabla}\nabla^sN \cdot  N_{0,s} 
		\,dx_1dy_1\right|
		\label{est-I11-1}\\
		&+\sum_{s=1}^{3}
		\left|\int_{\T\times\R}\RR_1([\nabla^s,\V]\cdot \widetilde{\nabla}N)\cdot  N_{0,s}  \,dx_1dy_1\right|
		\nonumber\\
		\tri&~ I_{121}+I_{122}+I_{123}.
		\nonumber
	\end{align}
	Notice that
	\begin{align}\label{P0dive}
		\PP_{0}\widetilde{\dive}\V 
		=\PP_{0}\pa_{y_1} V^2.
	\end{align}
	The application of Sobolev’s inequality, \eqref{est-varphi-1}, \eqref{P0dive}, and Lemma \ref{lem-est-NDW} yields 
	\begin{align}
		\int_0^t I_{121}(\tau)\,d\tau
		\lesssim& 
		\sum_{s=1}^{3}\left(\|\PP_{0}\pa_{y_1} V^2\|_{L^\infty_tL^\infty}
		\big(\| \PP_{0} N_{0,s} \|_{L^2_tL^2}^2+\| \PP_{\not=} N_{0,s} \|_{L^2_tL^2}^2\big)
		\right.
		\label{est-I111}\\
		&\left.
		+\|\PP_{\not=}\widetilde{\dive} \V\|_{L^2_tL^\infty} \|N_{0,s}\|_{L^\infty_tL^2} \|\PP_{\not=} N_{0,s} \|_{L^2_tL^2}\right)
		\nonumber\\
		\lesssim&~
		\mu^{-\f56}|\E(t)|^{\f12}
		\int_0^t \D(\tau)\,d\tau.\nonumber		
	\end{align}
	While it is observe that
	\begin{align}\label{est-I122-1}
		I_{122}
		=
		\sum_{s=1}^{3}
		\left|\int_{\T\times\R} \left(\PP_{0}\left([\RR_1,\V]\cdot \widetilde{\nabla}\nabla^sN\right) \cdot  \PP_{0}N_{0,s} +\PP_{\not=}\left([\RR_1,\V]\cdot \widetilde{\nabla}\nabla^sN\right) \cdot  \PP_{\not=}N_{0,s}\right)
		\,dx_1dy_1\right|.
	\end{align}
	We note that the decomposition in \eqref{est-I122-1} will be used extensively throughout the proof.
	Notice that
	\begin{align*}
		m_1=m_2=\varphi=1,\quad\text{for}~~k=0,
	\end{align*}
	which, along with the definition of \eqref{def_R1-R2}, gives 
	\begin{align}\label{est-P0-R1-0}
		\PP_{0}\RR_1 q= \PP_{0}q.
	\end{align}
	Hence, we have
	\begin{align}\label{est-P0-R1}
		\PP_{0}[\RR_1,\PP_{0}\V]\cdot \PP_{0}\widetilde{\nabla}\nabla^sN
		=0.
	\end{align}
	The estimates \eqref{est-varphi-1}, \eqref{est-I122-1}, and \eqref{est-P0-R1}, along with Lemmas \ref{sec2:lem-com-1}-\ref{sec2:lem-com-1-1} and \ref{lem-est-NDW}, imply that
	\begin{align}
		\int_0^t I_{122}(\tau)\,d\tau
		\lesssim&
		\sum_{s=1}^{3}
		\int_0^t
		\left(
		\left|\int_{\T\times\R} [\RR_1,\PP_{\not=}\V]\cdot \PP_{\not=}\widetilde{\nabla}\nabla^sN \cdot  N_{0,s} 
		\,dx_1dy_1\right|
		\right.
		\label{est-I112}	\\
		&\left.+\left|\int_{\T\times\R} \PP_{\not=}\left([\RR_1,\PP_{0}\V]\cdot \PP_{\not=}\widetilde{\nabla}\nabla^sN
		+[\RR_1,\PP_{\not=}\V]\cdot \PP_{0}\widetilde{\nabla}\nabla^sN
		\right) \cdot  \PP_{\not=}N_{0,s} 
		\,dx_1dy_1\right|
		\right)\,d\tau
		\nonumber\\
		\lesssim& \sum_{s=1}^{3}\Big(
		\|\PP_{\not=}\V\|_{L^2_tH^2} \|\PP_{\not=}\widetilde{\nabla} \nabla N\|_{L^2_tH^2}\|N_{0,s}\|_{L^\infty_tL^2} 	
		+\|\PP_{0}\pa_{y_1} \V\|_{L^\infty_tH^2}
		\| \PP_{\not=} \nabla N\|_{L^2_tH^2} \| \PP_{\not=} N_{0,s} \|_{L^2_tL^2}
		\nonumber\\
		&
		+\|\PP_{\not=}\V\|_{L^2_tH^2} \|\PP_{0}\pa_{y_1}^2N\|_{L^\infty_tH^2}\|\PP_{\not=}N_{0,s}\|_{L^2_tL^2} 	\Big)
		\nonumber\\
		\lesssim&~
		\mu^{-1}|\E(t)|^{\f12}\int_0^t \D(\tau)\,d\tau. 		\nonumber	
	\end{align}
	Recalling the definition \eqref{def_R1-R2}, we apply Lemmas \ref{sec2:lem-1}, \ref{sec2:lem-com-2}, \ref{lem-est-NDW}, Sobolev inequality, and \eqref{est-P0-N-D-W} to obtain
	\begin{align}
		\int_0^t I_{123}(\tau)\,d\tau
		\lesssim&
		\sum_{s=1}^{3}\left[
		\|\PP_{0}\pa_{y_1} \V\|_{L^\infty_tH^2}\|\PP_{0}\pa_{y_1} N\|_{L^2_tH^2}\|\PP_{0}N_{0,s}\|_{L^2_tL^2}
		\right.
		\label{est-I113}\\
		&
		+\|\PP_{\not=}\nabla \V\|_{L^2_tH^2}
		\|\PP_{\not=}\widetilde{\nabla} N\|_{L^2_tH^2}
		\|N_{0,s}\|_{L^\infty_tL^2}
		\nonumber\\
		&\left.
		+\|\PP_{0}\pa_{y_1}(N, \V)\|_{L^\infty_tH^2}\|\PP_{\not=}(\widetilde{\nabla} N, \nabla \V)\|_{L^2_tL^2}\|\PP_{\not=}N_{0,s}\|_{L^2_tL^2}\right]
		\nonumber\\
		\lesssim&~
		\mu^{-1}|\E(t)|^{\f12}\int_0^t \D(\tau)\,d\tau.
		\nonumber
	\end{align}
	Substituting \eqref{est-I111} \eqref{est-I112} and \eqref{est-I113} into \eqref{est-I11-1}, we find that
	\begin{align}
		\int_0^t I_{12}(\tau)\,d\tau\lesssim \mu^{-1}|\E(t)|^{\f12}\int_0^t \D(\tau)\,d\tau.
		\label{est-I11-2}
	\end{align}
	With the help of Lemmas \ref{sec2:lem-1}, \ref{lem-est-NDW}, and Sobolev inequality, we readily check that
	\begin{align}
		\int_0^tI_{13}(\tau)\,d\tau
		\lesssim&\sum_{s=1}^{3}
		\|\PP_{0}(N,D)\|_{L^\infty_tH^2}\|\PP_{0}\pa_{y_1}^s(N,D)\|_{L^2_tL^2} 
		\|\PP_{0}N_{0,s}\|_{L^2_tL^2}
		\label{est-I12}\\
		&
		+\sum_{s=0}^{3}\left[
		\|\PP_{\not=}N\|_{L^2_tH^3}\|\PP_{\not=}D\|_{L^2_tH^3}\|N_{0,s}\|_{L^\infty_tL^2}
		+\left(\|\PP_{0}N\|_{L^\infty_tH^3}\|\PP_{\not=}D\|_{L^2_tH^3}
		\right.\right.
		\nonumber\\
		&\left.\left.
		+\|\PP_{\not=}N\|_{L^2_tH^3}\|\PP_{0}D\|_{L^\infty_tH^2}
		+\|\PP_{\not=}N\|_{L^\infty_tH^2}\|\PP_{0}\pa_{y_1}^3D\|_{L^2_tL^2}
		\right) \|\PP_{\not=}N_{0,s}\|_{L^2_tL^2}
		\right]
		\nonumber\\
		\lesssim&~\mu^{-1}|\E(t)|^{\f12}\int_0^t \D(\tau)\,d\tau. \nonumber
	\end{align}
	Combining \eqref{est-I1-1}, \eqref{est-I10}, \eqref{est-I11-2}, and \eqref{est-I12}, one can deduce \eqref{est-I1-2} directly.

	\subsubsection{Proof of Lemma \ref{lem1-2}}
	Recalling the definitions of $\RR_1$ and $\RR_5$ in \eqref{def_R1-R2} and \eqref{def_R5}, respectively, and referring to \eqref{def-bf-F2} and \eqref{def_wF2}, we obtain
	\begin{align*}
		\FFF_{2,0,s}=\RR_1\widetilde{\Delta}^{-\f12}\nabla^s\FFF_2=\RR_1\widetilde{\Delta}^{-\f12}\nabla^s\widetilde{\dive}\FF_{2}=\RR_1\RR_5\nabla^s\FF_{2}.
	\end{align*}
	We begin by dealing with the second term by recalling \eqref{def-bf-F2} as follows
	\begin{align}
		I_2
		\leq & \left|\int_{\T\times\R}\PP_{\not=}\RR_1\RR_5\FF_2 \PP_{\not=}\UU_{0,0} \,dx_1dy_1\right|+\sum_{s=1}^{3}
		\left|\int_{\T\times\R}\RR_1\RR_5\nabla^s\FF_2 \UU_{0,s} \,dx_1dy_1\right|
		\label{est-I2-1}\\
		\lesssim & \left|\int_{\T\times\R}\PP_{\not=}\RR_1\RR_5\left(\V\cdot\widetilde{\nabla}\V\right)\PP_{\not=}\UU_{0,0} \,dx_1dy_1\right|
		+\sum_{s=1}^{3}
		\left|\int_{\T\times\R}\RR_1\RR_5\nabla^s\left(\V\cdot\widetilde{\nabla}\V\right)\UU_{0,s} \,dx_1dy_1\right|
		\nonumber\\
		&+
		\left|\int_{\T\times\R}\left(\PP_{\not=}\RR_1\RR_5\left(f(N)\widetilde{\nabla} N\right)\PP_{\not=}\UU_{0,0}+\sum_{s=1}^{3}\RR_1\RR_5\nabla^s\left(f(N)\widetilde{\nabla} N\right)\UU_{0,s}\right) \,dx_1dy_1\right|
		\nonumber\\
		&+
		\left|\int_{\T\times\R}\left(\PP_{\not=}\RR_1\RR_5\left(g(N)\left(\mu \widetilde{\Delta} \V + (\mu + \mu')\widetilde{\nabla} \widetilde{\dive}\V\right)\right)\PP_{\not=}\UU_{0,0}
		\right.\right.
		\nonumber\\
		&\left.\left.+\sum_{s=1}^{3}\RR_1\RR_5\nabla^s\left(g(N)\left(\mu \widetilde{\Delta} \V + (\mu + \mu')\widetilde{\nabla} \widetilde{\dive}\V\right)\right)\UU_{0,s}\right) \,dx_1dy_1\right|
		\nonumber\\
		\tri& \sum_{i=1}^4 I_{2i},\nonumber
	\end{align}
	where $f(N)$ and $g(N)$ are defined in \eqref{def-f-g}.
	Recalling  the definitions of $\RR_1$, $\RR_5$, and $\UU_{0,0}$ in \eqref{def_R1-R2}, \eqref{def_R5}, and \eqref{def_D-j}, respectively, we have
	\begin{align*}
		\PP_{\not=}\RR_1\RR_5\left(\V\cdot\widetilde{\nabla}\V\right)
		=&
		\PP_{\not=}\left[\V\cdot\PP_{\not=} \widetilde{\nabla}\UU_{0,0}
		+\PP_{\not=}[\RR_1,\V]\cdot\PP_{\not=} \widetilde{\nabla}\RR_5\V
		\right.
		\\
		&\left. +\PP_{\not=}\RR_1\left([\RR_5,\V]\cdot\PP_{\not=} \widetilde{\nabla}\V\right)
		+\RR_1\RR_5(\PP_{\not=}V^2\PP_{0}\pa_{y_1} \V)
		\right]
		.		\end{align*}
	This estimate, along with integrating by parts, and Lemma \ref{lem-est-NDW}, leads us to get that
	\begin{align}
		\int_0^t I_{21}(\tau)\,d\tau
		\lesssim&\|\widetilde{\nabla} \V\|_{L^\infty_tH^2}
		\left(\| \PP_{\not=} \UU_{0,0} \|_{L^2_tL^2}+\| \PP_{\not=} \V \|_{L^2_tL^2}\right)
		\| \PP_{\not=} \UU_{0,0} \|_{L^2_tL^2}
		\label{est-I20}\\
		&+\big(
		\| \PP_{\not=} \V \|_{L^\infty_tL^2} \| \PP_{\not=} \widetilde{\nabla}\V \|_{L^2_tL^2}
		+\|\PP_{\not=} V^2\|_{L^2_tL^\infty}\| \PP_{0}\pa_{y_1} \V \|_{L^\infty_tL^2}\Big)
		\| \PP_{\not=} \UU_{0,0} \|_{L^2_tL^2}
		\nonumber\\
		\lesssim&~ \mu^{-1}|\E(t)|^{\f12}\int_0^t \D(\tau)\,d\tau.
		\nonumber
	\end{align}
	By integration by parts again,  we find that
	\begin{align}
		I_{22}
		\lesssim&\sum_{s=1}^{3}
		\left|\int_{\T\times\R}\widetilde{\dive}\V |\UU_{0,s}|^2\,dx_1dy_1\right|
		+\sum_{s=1}^{3}
		\left|\int_{\T\times\R}[\RR_1,\V]\cdot \widetilde{\nabla}\nabla^s\RR_5\V \cdot  \UU_{0,s} \,dx_1dy_1\right|
		\label{est-I21-1}\\
		&+\sum_{s=1}^{3}
		\left|\int_{\T\times\R}\RR_1([\RR_5,\V]\cdot \widetilde{\nabla}\nabla^s\V)\cdot  \UU_{0,s}  \,dx_1dy_1\right|
		\nonumber\\
		&+\sum_{s=1}^{3}
		\left|\int_{\T\times\R}\RR_1\RR_5([\nabla^s,\V]\cdot \widetilde{\nabla}\V)\cdot  \UU_{0,s} \,dx_1dy_1\right|
		\nonumber\\
		\tri& \sum_{i=1}^4I_{22i}.
		\nonumber
	\end{align}
	Owing to $\RR_5\V=\widetilde{\Delta}^{-\f12}D$,
	it then follows, in a similar manner to estimating \eqref{est-I111} and \eqref{est-I112}, that
	\begin{align}
		\int_0^t (I_{221}+I_{222})(\tau)\,d\tau
		\lesssim
		\mu^{-1}|\E(t)|^{\f12}\int_0^t \D(\tau)\,d\tau.
		\label{est-I211-I212}
	\end{align}
	Notice that
	\begin{align}\label{est-P0-R5}
		\PP_{0} \RR_5\V=\PP_{0}\RR_4V^2,
	\end{align}
	which, together with \eqref{est-P0-R1-0}, implies that
	\begin{align}\label{est-P0-R1R5}
		\PP_{0}\RR_1\left([\RR_5,\PP_{0}\V]\cdot \PP_{0}\widetilde{\nabla}\nabla^s\V\right)
		=[\RR_4,\PP_{0}V^2] \PP_{0}\pa_{y_1}^{s+1}V^2.
	\end{align}
	Using a similar decomposition as \eqref{est-I122-1}, and applying  \eqref{est-P0-R1R5}, we obtain that
	\begin{align}\label{est-I223-1}
		I_{223}
		\lesssim&
		\sum_{s=1}^{3}
		\left|\int_{\T\times\R}[\RR_4,\PP_{0}V^2]\cdot \PP_{0}\pa_{y_1}^{s+1}V^2\cdot \PP_{0}\UU_{0,s}\,dx_1dy_1\right|
		\\
		&+\sum_{s=1}^{3}
		\left|\int_{\T\times\R}\RR_1[\RR_5,\PP_{\not=}\V]\cdot \PP_{\not=}\widetilde{\nabla}\nabla^s\V
		\UU_{0,s} \,dx_1dy_1\right|
		\nonumber\\
		&+\sum_{s=1}^{3}
		\left|\int_{\T\times\R}\left(\RR_1[\RR_5,\PP_{0}\V]\cdot \PP_{\not=}\widetilde{\nabla}\nabla^s\V+\RR_1[\RR_5,\PP_{\not=}\V]\cdot \PP_{0}\widetilde{\nabla}\nabla^s\V\right)
		\PP_{\not=}\UU_{0,s} \,dx_1dy_1\right|.
		\nonumber
	\end{align}
	By virtue of Lemmas \ref{sec2:lem-1}, \ref{sec2:lem-com-1}, and \ref{lem-est-NDW}, we deduce from \eqref{est-I223-1} that
	\begin{align}
		\int_0^t I_{223}(\tau)\,d\tau
		\lesssim& \sum_{s=1}^{3}\left(\|\PP_{0}\pa_{y_1} V^2\|_{L^\infty_tL^\infty}\| \PP_{0} \pa_{y_1}^sV^2 \|_{L^2_tL^2} \| \PP_{0} \UU_{0,s} \|_{L^2_tL^2}
		\right.
		\label{est-I213}\\
		&+\|\PP_{\not=}\widetilde{\nabla} \V\|_{L^2_tL^\infty}
		\| \PP_{\not=} \nabla^s \V\|_{L^2_tL^2} \| \UU_{0,s} \|_{L^\infty_tL^2}
		\nonumber\\
		&\left.
		+\|\PP_{0}\pa_{y_1} \V\|_{L^\infty_tH^2}
		\| \PP_{\not=}( \nabla\V, \widetilde{\nabla} V^2)\|_{L^2_tH^2} \| \PP_{\not=} \UU_{0,s} \|_{L^2_tL^2}
		\right)
		\nonumber\\
		\lesssim&~ \mu^{-1}|\E(t)|^{\f12}\int_0^t \D(\tau)\,d\tau.\nonumber				
	\end{align}
	Again thanks to Lemmas \ref{sec2:lem-1}, \ref{sec2:lem-com-2}, and \ref{lem-est-NDW}, along with \eqref{est-P0-R5}, it then follows from a similar decomposition as \eqref{est-I122-1} that
	\begin{align}
		\int_0^t I_{224}(\tau)\,d\tau
		\lesssim&
		\sum_{s=1}^3\left(\|\PP_{0}\pa_{y_1} \V\|_{L^\infty_tH^2}\|\PP_{0}\pa_{y_1} V^2\|_{L^2_tH^2}\|\PP_{0}\UU_{0,s}\|_{L^2_tL^2}
		\right.
		\label{est-I214}\\
		&+\|\PP_{\not=}\nabla \V\|_{L^2_tH^2}\|\PP_{\not=}\widetilde{\nabla} \V\|_{L^2_tH^2}
		 \|\UU_{0,s}\|_{L^\infty_tL^2}
		 \nonumber\\
		&\left.
		+\|\PP_{0}\pa_{y_1}\V\|_{L^\infty_tH^2}\|\PP_{\not=}(\widetilde{\nabla} \V, \nabla \V)\|_{L^2_tH^2}\|\PP_{\not=}\UU_{0,s}\|_{L^2_tL^2}\right)
		\nonumber\\
		\lesssim&~\mu^{-1}|\E(t)|^{\f12}\int_0^t \D(\tau)\,d\tau.\nonumber
	\end{align}
	Substituting \eqref{est-I211-I212} \eqref{est-I213} and \eqref{est-I214} into \eqref{est-I21-1}, we find that
	\begin{align}
		\int_0^t I_{22}(\tau)\,d\tau\lesssim \mu^{-1}|\E(t)|^{\f12}\int_0^t \D(\tau)\,d\tau.\label{est-I21-2}
	\end{align}
	For $I_{23}$, we divide it as the following two terms:
	\begin{align}
		I_{23}
		=&
		\left|\int_{\T\times\R}\left[\PP_{\not=}\left(\RR_1\RR_5\left(f(N)\widetilde{\nabla} N\right)\right)\PP_{\not=}\UU_{0,0}+\sum_{s=1}^{2}\left(\RR_1\RR_5\nabla^s\left(f(N)\widetilde{\nabla} N\right)\right)\UU_{0,s}\right] \,dx_1dy_1\right|
		\label{est-I22-1}\\ 
		&+
		\left|\int_{\T\times\R}\left(\RR_1\RR_5\nabla^3\left(f(N)\widetilde{\nabla} N\right)\right)\UU_{0,3} \,dx_1dy_1\right|
		\nonumber\\
		\tri&~  I_{231}+I_{232}.
		\nonumber
	\end{align}
	While it is easy to observe, by applying \eqref{est-fn-gn}, Lemmas \ref{sec2:lem-1} and \ref{lem-est-NDW}, that
	\begin{align}
		\int_0^t I_{231}(\tau)\,d\tau
		\lesssim&
		\sum_{s=1}^{2}\|\PP_{0}N\|_{L^\infty_tH^3}
		\| \PP_{0} \pa_{y_1} N \|_{L^2_tH^2}
		\| \PP_{0} \UU_{0,s} \|_{L^2_tL^2}
		\label{est-I221}\\
		&+\sum_{s=0}^{2}\left( \|\PP_{\not=}N\|_{L^2_tH^2}
		\| \PP_{\not=} \widetilde{\nabla}N \|_{L^2_tH^2}
		\|  \UU_{0,s} \|_{L^\infty_tL^2}
		\right.
		\nonumber\\
		&\left.
		+\|\PP_{0}N\|_{L^\infty_tH^3}
		\| \PP_{\not=}( \widetilde{\nabla} N,N) \|_{L^2_tH^2}
		\| \PP_{\not=} \UU_{0,s} \|_{L^2_tL^2} \right)
		\nonumber\\
		\lesssim&~ \mu^{-1}|\E(t)|^{\f12}\int_0^t \D(\tau)\,d\tau.
		\nonumber
	\end{align}
	As for $ I_{232}$, notice that there exists a smooth function $\tilde{f}$, such that
	\begin{align*}
		\nabla^3 (f(N)\widetilde{\nabla}N)=\na^3\widetilde{\nabla}(\tilde{f}(N)).
	\end{align*}
	Thus, using Lemmas \ref{sec2:lem-1}, \ref{lem-est-NDW}, and \eqref{est-fn-gn}, we integrate by parts to deduce
	\begin{align}
		\int_0^t I_{232}(\tau)\,d\tau
		\lesssim&
		\int_0^t\left|\int_{\T\times\R}\left(\RR_1\RR_5\left( \nabla^3 (\tilde{f}(N))\right)
		\right) \widetilde{\nabla} \UU_{0,3} \,dx_1dy_1\right|\,d\tau
		\label{est-I222}\\
		\lesssim& \left(
		\|\PP_{0}  N\|_{L^\infty_tH^2}\|\PP_{0}  \pa_{y_1}N\|_{L^2_tH^2}\| \PP_{0}\widetilde{\nabla} \UU_{0,3} \|_{L^2_tL^2}
		\right.
		\nonumber\\
		&+\|\PP_{\not=}  N\|_{L^\infty_tH^3}\|\PP_{\not=}  N\|_{L^2_tH^3}\| \widetilde{\nabla} \UU_{0,3} \|_{L^2_tL^2}
		\nonumber\\
		&\left.
		+\|\PP_{0}  N\|_{L^\infty_tH^3}
		\| \PP_{\not=}N \|_{L^2_tH^3}\| \PP_{\not=}\widetilde{\nabla} \UU_{0,3} \|_{L^2_tL^2}
		\right)
		\nonumber\\
		\lesssim&~\mu^{-1}|\E(t)|^{\f12}\int_0^t \D(\tau)\,d\tau.
		\nonumber
	\end{align}		
	The combination of \eqref{est-I22-1}-\eqref{est-I222} yields that
	\begin{align}
		\int_0^t I_{23}(\tau)\,d\tau\lesssim \mu^{-1}|\E(t)|^{\f12}\int_0^t \D(\tau)\,d\tau.\label{est-I22-2}
	\end{align}
	Now we return to estimate $I_{24}$. By integration by parts, we obtain
	\begin{align}
		I_{24}
		\lesssim&
		\left|\int_{\T\times\R}\left(\PP_{\not=}\RR_1\RR_5\left(g(N)\left(\mu \widetilde{\Delta} \V + (\mu + \mu')\widetilde{\nabla} \widetilde{\dive}\V\right)\right)\PP_{\not=}\UU_{0,0}
		\right.\right.
		\label{est-I23-1}\\
		&\left.\left.+\sum_{s=1}^{2}\RR_1\RR_5\nabla^s\left(g(N)\left(\mu \widetilde{\Delta} \V + (\mu + \mu')\widetilde{\nabla} \widetilde{\dive}\V\right)\right)\UU_{0,s}\right) \,dx_1dy_1\right|
		\nonumber\\
		&+\left|\int_{\T\times\R}\RR_1\RR_5\nabla^3\left(g(N)\left(\mu \widetilde{\Delta} \V + (\mu + \mu')\widetilde{\nabla} \widetilde{\dive}\V\right)\right)\UU_{0,3} \,dx_1dy_1\right|
		\nonumber\\
		\tri&~  I_{241}+I_{242}.
		\nonumber
	\end{align}
	With the help of \eqref{est-fn-gn}, Lemmas \ref{sec2:lem-1} and \ref{lem-est-NDW},  we get
	\begin{align}
		\int_0^t I_{241}(\tau)\,d\tau 
		\lesssim&~\mu\sum_{s=1}^{2}
		\|\PP_{0}N\|_{L^\infty_tH^2}\|\PP_{0}\pa_{y_1}^{2}\V\|_{L^2_tH^2}\|\PP_0\UU_{0,s}\|_{L^2_tL^2}
		\label{est-I231} \\
		&+\mu\sum_{s=0}^{2}
		\left(\|\PP_{\not=}N\|_{L^2_tH^2}\|\PP_{\not=}\widetilde{\nabla}^2\V\|_{L^2_tH^2}\|\UU_{0,s}\|_{L^\infty_tL^2}
		\right.
		\nonumber\\
		&\left.
		+\left(\|\PP_{0}N\|_{L^\infty_tH^2}\|\PP_{\not=}\widetilde{\nabla}^2\V\|_{L^2_tH^2} 
		+\|\PP_{\not=}N\|_{L^\infty_tH^2}\|\PP_{0}\pa_{y_1}^{2}\V\|_{L^2_tH^2}\right)
		\|\PP_{\not=}\UU_{0,s}\|_{L^2_tL^2}
		\right)
		\nonumber\\
		\lesssim&~\mu^{-\f13}|\E(t)|^{\f12}\int_0^t \D(\tau)\,d\tau.
		\nonumber
	\end{align}
	Noticing that $\nabla^3\left(\mu \widetilde{\Delta} \V + (\mu + \mu')\widetilde{\nabla}\widetilde{\dive}\V\right)$ in $I_{242}$ cannot be controlled by all the energy in the $H^3$ framework, we deal with this term by integration by parts and using the dissipation estimate of $D_{0,3}$, namely $\widetilde{\nabla}D_{0,3}$.
	More precisely, we have
	\begin{align}
		I_{242}
		\lesssim& \left|\sum_{s=0}^2\int_{\T\times\R}\RR_1\RR_5\nabla^{s}\left(\nabla g(N)\nabla^{2-s}\left(\mu \widetilde{\Delta} \V + (\mu + \mu')\widetilde{\nabla}\widetilde{\dive}\V\right)\right)
		\UU_{0,3} \,dx_1dy_1\right| 
		\label{est-I232-1}\\
		&+\left| \int_{\T\times\R}
		\left[\RR_1\RR_5\left( g(N)\nabla^3\left(\mu \widetilde{\nabla} \V + (\mu + \mu')\widetilde{\dive}\V\right)\right)\widetilde{\nabla}\UU_{0,3}
		\right.\right.
		\nonumber\\
		&\left.\left.
		\quad
		+  \RR_1\RR_5\left(\widetilde{\nabla} g(N)\nabla^3\left(\mu \widetilde{\nabla} \V + (\mu + \mu')\widetilde{\dive}\V\right)\right)
		\UU_{0,3} \right]\,dx_1dy_1\right|
		\nonumber\\	
		\tri&~I_{2421}+I_{2422}.
		\nonumber
	\end{align}
	According to Lemmas \ref{sec2:lem-1}, \ref{lem-est-NDW}, the Sobolev inequality, and \eqref{est-fn-gn}, we see that
	\begin{align}
		\int_0^t I_{2421}(\tau)\,d\tau
		\lesssim&~\mu 
		\|\PP_{0}\pa_{y_1}g(N)\|_{L^\infty_tH^2}\|\PP_{0}\pa_{y_1}^2\V\|_{L^2_tH^2}
		\|\PP_0\UU_{0,3}\|_{L^2_tL^2}
		+\mu \|\PP_{\not=}\nabla g(N)\|_{L^2_tH^2}
		\label{est-I2321}\\
		&\times
		\|\PP_{\not=}\widetilde{\nabla}^2\V\|_{L^2_tH^2}
		\|\UU_{0,3}\|_{L^\infty_tL^2}
		+\mu \left(\|\PP_{0}\pa_{y_1}g(N)\|_{L^\infty_tH^2}\|\PP_{\not=}\widetilde{\nabla}^2\V\|_{L^2_tH^2}
		\right.
		\nonumber\\
		&\left.
		+\|\PP_{\not=}\nabla g(N)\|_{L^\infty_tH^2}\|\PP_{0}\pa_{y_1}^2\V\|_{L^2_tH^2}
		\right)\|\PP_{\not=}\UU_{0,3}\|_{L^2_tL^2}
		\nonumber
		\\
		\lesssim&~\mu^{-\f13}|\E(t)|^{\f12}\int_0^t \D(\tau)\,d\tau.
		\nonumber
	\end{align}
	For $I_{2422}$, invoking Lemmas~\ref{sec2:lem-1}, \ref{lem-est-NDW}, the Sobolev inequality, and \eqref{est-fn-gn}, we obtain
	\begin{align}
		\int_0^t I_{2422}(\tau)\,d\tau
		\lesssim&~ \mu \|\PP_{0}\pa_{y_1}^4\V\|_{L^2_tL^2}
		\left(
		\|\PP_{0}N\|_{L^\infty_tH^2}\|\PP_0\pa_{y_1}\UU_{0,3}\|_{L^2_tL^2}
		+\|\PP_{0}\pa_{y_1}N\|_{L^\infty_tH^2}\|\PP_0\UU_{0,3}\|_{L^2_tL^2}
		\right.
		\label{est-I2322}\\
		&
		\left.
		+\|\PP_{\not=}N\|_{L^\infty_tH^2}
		\|\PP_{\not=}\widetilde{\nabla}\UU_{0,3}\|_{L^2_tL^2}
		+\|\PP_{\not=}\widetilde{\nabla}N\|_{L^\infty_tH^2}
		\|\PP_{\not=}\UU_{0,3}\|_{L^2_tL^2}
		\right)
		\nonumber\\
		&+\mu \|\PP_{\not=}\widetilde{\nabla}\nabla^3\V\|_{L^2_tL^2}
		\left(
		\|\PP_{\not=}N\|_{L^\infty_tH^2}\|\widetilde{\nabla}\UU_{0,3}\|_{L^2_tL^2}
		+\|\PP_{\not=}\widetilde{\nabla}N\|_{L^2_tH^2}\|\UU_{0,3}\|_{L^\infty_tL^2}
		\right.
		\nonumber\\
		&
		\left.
		+\|\PP_{0}N\|_{L^\infty_tH^2}\|\PP_{\not=}\widetilde{\nabla}\UU_{0,3}\|_{L^2_tL^2}
		+\|\PP_{0}\pa_{y_1}N\|_{L^\infty_tH^2}\|\PP_{\not=}\UU_{0,3}\|_{L^2_tL^2}
		\right)
		\nonumber\\
		\lesssim&~\mu^{-\f13}|\E(t)|^{\f12}\int_0^t \D(\tau)\,d\tau.
		\nonumber
	\end{align}
	Substituting \eqref{est-I2321} and \eqref{est-I2322} into \eqref{est-I232-1}, we deduce that
	\begin{align}
		\int_0^t I_{242}(\tau)\,d\tau \lesssim \mu^{-\f13}|\E(t)|^{\f12}\int_0^t \D(\tau)\,d\tau. \label{est-I232-2}
	\end{align}
	Substituting \eqref{est-I231} and \eqref{est-I232-2} into \eqref{est-I23-1}, we obtain
	\begin{align}
		\int_0^t I_{24}(\tau)\,d\tau \lesssim \mu^{-\f13}|\E(t)|^{\f12}\int_0^t \D(\tau)\,d\tau. \label{est-I23-2}
	\end{align}
	The combination of \eqref{est-I2-1}, \eqref{est-I20}, \eqref{est-I21-2}, \eqref{est-I22-2} and \eqref{est-I23-2} yields \eqref{est-I2-2} directly.

	\subsubsection{Proof of Lemma \ref{lem1-3}}
	Recalling \eqref{def_wF1}, \eqref{def_wF2}, \eqref{def_R5}, and \eqref{est-P0-R5} and integrating by parts, we decompose $I_3$ into the following three terms:
	\begin{align}
		I_3\leq &~\mu\sum_{s=1}^{3}\left|\int_{\T\times\R}\left(\PP_{0}|\pa_{y_1}|^{-1}\FFF_{1,0,s}\PP_{0}\UU_{0,s}+\PP_{0}\FFF_{2,0,s}\PP_{0}|\pa_{y_1}|^{-1}N_{0,s}\right)\,dx_1dy_1\right|
		\label{est-I3-1}\\
		\lesssim 
		&~ \mu \sum_{s=0}^{2}\left|
		\int_{\T\times\R} \PP_{0}\RR_4\pa_{y_1}^s\left(\V\cdot\widetilde{\nabla} N\right)
		\PP_{0}\RR_4\pa_{y_1}^{s+1}V^2
		\,dx_1dy_1
		\right|
		\nonumber\\
		&+\mu \sum_{s=0}^{2}\left|
		\int_{\T\times\R} \PP_{0}\RR_4\pa_{y_1}^s\left(ND\right)
		\PP_{0}\RR_4\pa_{y_1}^{s+1}V^2\,dx_1dy_1
		\right|
		\nonumber\\
		&+\mu \sum_{s=0}^{2}\left|
		\int_{\T\times\R}\PP_{0}\RR_4\pa_{y_1}^{s}\left(\V\cdot\widetilde{\nabla} V^2
		+f(N)\widetilde{\nabla} N
		\right.\right.
		\nonumber\\
		&\left.\left.
		+g(N)(\mu\widetilde{\Delta}\V
		+(\mu+\mu')\widetilde{\nabla}\widetilde{\dive}\V)
		\right)\PP_{0}\RR_4\pa_{y_1}^{s+1}N\,dx_1dy_1
		\right|
		\nonumber\\
		\tri& \sum_{i=1}^3I_{3i}.
		\nonumber
	\end{align}
	Then, we will handled all the terms $I_{3i}~ (i=1,2,3)$. 
	Some direct computations lead us to get that
	\begin{align*}
		\PP_{0}\left(\V\cdot\widetilde{\nabla} N\right)
		=\PP_{0}V^2\PP_{0} \pa_{y_1}N
		+\PP_{0}\left(\PP_{\not=}\V\cdot\PP_{\not=}\widetilde{\nabla} N\right).
	\end{align*}
	We observe that $\PP_{0}V^1$ does not appear; therefore, according to the Sobolev inequality, Lemmas \ref{sec2:lem-1} and \ref{lem-est-NDW}, we have
	\begin{align}
		\int_0^tI_{31}(\tau)\,d\tau
		\lesssim&~
		\mu \sum_{s=0}^{2}\left(\|\PP_{0}V^2\|_{L^\infty_tH^2}\|\PP_{0}\pa_{y_1} N\|_{L^2_tH^2}
		\|\PP_{0}\pa_{y_1}^{s+1}V^2\|_{L^2_tL^2}
		\right.
		\label{est-I31}\\
		&\left.
		+
		\|\PP_{\not=}\V\|_{L^2_tH^2}\|\PP_{\not=}\widetilde{\nabla} N\|_{L^2_tH^2}	\|\PP_{0}\pa_{y_1}^{s+1}V^2\|_{L^\infty_tL^2}
		\right)
		\nonumber\\
		\lesssim&~|\E(t)|^{\f12}\int_0^t \D(\tau)\,d\tau.
		\nonumber
	\end{align}
	Following a similar derivation as in \eqref{est-I31}, we obtain
	\begin{align}
		\int_0^t(I_{32}+I_{33})(\tau)\,d\tau
		\lesssim 
		|\E(t)|^{\f12}\int_0^t \D(\tau)\,d\tau.
		\label{est-I32}
	\end{align}
	Collecting \eqref{est-I3-1}-\eqref{est-I32} together, we obtain \eqref{est-I3-2} immediately.

	\subsubsection{Proof of Lemma \ref{lem1-4}}

	Unlike the estimate for $I_{3}$,  the ``bad derivative" $\widetilde{\nabla}$ does not coincide with  the ``good derivative" $\nabla$ on non-zero modes,  so $\nabla$ and $\widetilde{\Delta}^{-\f12}$ can not be canceled.
	Moreover, neither $\nabla^3\widetilde{\nabla}N$ nor $\PP_{0}V^1$ can be controlled by the energy we have defined.
	Consequently, we must combine the two transport terms and integrate by parts to handle $I_4$.  More precisely, using \eqref{def_wF1} and \eqref{def_wF2}, we obtain
	\begin{align}
		I_4\leq &\sum_{s=0}^3\left|\int_{\T\times\R}\left(\PP_{\not=}\RR_3\FFF_{1,0,s}\PP_{\not=}\widetilde{\Delta}^{-\f12}\UU_{0,s}+\PP_{\not=}\FFF_{2,0,s}\PP_{\not=}\widetilde{\Delta}^{-\f12}\RR_3N_{0,s}\right)\,dx_1dy_1\right|
		\label{est-I5-1}\\
		\lesssim&
		\sum_{s=0}^{3}\left|\int_{\T\times\R}\left(\PP_{\not=}\RR_1\RR_3\nabla^s(\V\cdot\widetilde{\nabla}N)\PP_{\not=} \widetilde{\Delta}^{-\f12}\UU_{0,s}
		+\PP_{\not=}\RR_1\RR_5\nabla^s(\V\cdot\widetilde{\nabla}\V)\PP_{\not=} \RR_3\widetilde{\Delta}^{-\f12}N_{0,s}
		\right)\,dx_1dy_1\right|
		\nonumber\\
		&+\sum_{s=0}^{3}\left|\int_{\T\times\R}\PP_{\not=}\RR_1\RR_3\nabla^s(ND)\PP_{\not=} \widetilde{\Delta}^{-\f12}\UU_{0,s}
		\,dx_1dy_1\right|
		\nonumber\\
		&
		+\sum_{s=0}^{3}\left|\int_{\T\times\R}\PP_{\not=}\RR_1\RR_5\nabla^s(f(N)\widetilde{\nabla}N)\PP_{\not=} \RR_3\widetilde{\Delta}^{-\f12}N_{0,s}
		\,dx_1dy_1\right|
		\nonumber\\
		&+\sum_{s=0}^{3}\left|\int_{\T\times\R}\PP_{\not=}\RR_1\RR_5\nabla^s\left(g(N)(\mu\widetilde{\Delta}\V+(\mu+\mu')\widetilde{\nabla}\widetilde{\dive}\V)\right)\PP_{\not=} \RR_3\widetilde{\Delta}^{-\f12}N_{0,s}
		\,dx_1dy_1\right|
		\nonumber\\
		\tri& \sum_{i=1}^{4}I_{4i}.
		\nonumber
	\end{align}
	For the first term $I_{41}$, while it is easy to observe that
	\begin{align*}
		\PP_{\not=}\RR_1\RR_3\nabla^s(\V\cdot\widetilde{\nabla}N)
		=&~\RR_1\RR_3\nabla^s(\PP_{0}\V\cdot\widetilde{\nabla}\PP_{\not=}N)
		+\PP_{\not=}\RR_1\RR_3\nabla^s(\PP_{\not=}\V\cdot\widetilde{\nabla}N)
		\\
		=&~\PP_{0}\V\cdot\widetilde{\nabla}\PP_{\not=}\RR_3 N_{0,s}
		+[\RR_1\RR_3,\PP_{0}\V]\cdot\widetilde{\nabla}\PP_{\not=}\nabla^sN
		\\
		&
		+\RR_1\RR_3\big([\nabla^s,\PP_{0}\V]\cdot\widetilde{\nabla}\PP_{\not=}N\big)
		+\PP_{\not=}\RR_1\RR_3\nabla^s(\PP_{\not=}\V\cdot\widetilde{\nabla}N),
		\nonumber\\
		\PP_{\not=}\RR_1\RR_5\nabla^s(\V\cdot\widetilde{\nabla}\V)
		=&~\RR_1\RR_5\nabla^s(\PP_{0}\V\cdot\widetilde{\nabla}\PP_{\not=}\V)
		+\PP_{\not=}\RR_1\RR_5\nabla^s(\PP_{\not=}\V\cdot\widetilde{\nabla}\V)
		\\
		=& ~\PP_{0}\V\cdot\widetilde{\nabla}\PP_{\not=}\UU_{0,s}
		+[\RR_1\RR_5,\PP_{0}\V]\cdot\widetilde{\nabla}\PP_{\not=}\nabla^s\V
		\\
		&~
		+\RR_1\RR_5\big([\nabla^s,\PP_{0}\V]\cdot\widetilde{\nabla}\PP_{\not=}\V\big)
		+\PP_{\not=}\RR_1\RR_5\nabla^s(\PP_{\not=}\V\cdot\widetilde{\nabla}\V).
		\nonumber
	\end{align*}
	These estimates enable us to express $I_{41}$ as the sum of four terms:
	\begin{align}
		I_{41}
		\lesssim&\sum_{s=0}^{3}\left|\int_{\T\times\R}\left(\PP_{0}\V\cdot\widetilde{\nabla}\PP_{\not=}\RR_3 N_{0,s}\PP_{\not=} \widetilde{\Delta}^{-\f12}\UU_{0,s}
		+\PP_{0}\V\cdot\widetilde{\nabla}\PP_{\not=}\UU_{0,s}\PP_{\not=} \RR_3\widetilde{\Delta}^{-\f12}N_{0,s}
		\right)\,dx_1dy_1\right|
		\label{est-I51-1}\\
		&+\sum_{s=0}^{3}\Big|\int_{\T\times\R}\big([\RR_1\RR_3,\PP_{0}\V]\cdot\widetilde{\nabla}\PP_{\not=}\nabla^sN\cdot\PP_{\not=}\widetilde{\Delta}^{-\f12}\UU_{0,s}
		\nonumber\\
		&\qquad~~+[\RR_1\RR_5,\PP_{0}\V]\cdot\widetilde{\nabla}\PP_{\not=}\nabla^s\V\cdot \PP_{\not=} \RR_3\widetilde{\Delta}^{-\f12}N_{0,s}
		\big)\,dx_1dy_1\Big|
		\nonumber\\
		&+\sum_{s=1}^{3}\Big|\int_{\T\times\R}\big(\RR_1\RR_3\big([\nabla^s,\PP_{0}\V]\cdot\widetilde{\nabla}\PP_{\not=}N\big)\cdot\PP_{\not=} \widetilde{\Delta}^{-\f12}\UU_{0,s}
		\nonumber\\
		&\qquad~~+\RR_1\RR_5\big([\nabla^s,\PP_{0}\V]\cdot\widetilde{\nabla}\PP_{\not=}\V\big)\cdot\PP_{\not=} \RR_3\widetilde{\Delta}^{-\f12}N_{0,s}
		\big)\,dx_1dy_1\Big|
		\nonumber\\
		&+\sum_{s=0}^{3}\Big|\int_{\T\times\R}\big(\PP_{\not=}\RR_1\RR_3\nabla^s(\PP_{\not=}\V\cdot\widetilde{\nabla}N)\PP_{\not=} \widetilde{\Delta}^{-\f12}\UU_{0,s}
		\nonumber\\
		&\qquad~~+\PP_{\not=}\RR_1\RR_5\nabla^s(\PP_{\not=}\V\cdot\widetilde{\nabla}\V)\PP_{\not=} \RR_3\widetilde{\Delta}^{-\f12}N_{0,s}
		\big)\,dx_1dy_1\Big|
		\nonumber\\
		\tri& \sum_{i=1}^{4}I_{41i}.
		\nonumber
	\end{align}
	For $I_{411}$, by some direct calculations, we have
	\begin{align}\label{est-I411-1}
		\PP_{0}\V\cdot\widetilde{\nabla}\PP_{\not=}\RR_3 N_{0,s}
		=\widetilde{\Delta}^{\f12}\left(\PP_{0}\V\cdot\widetilde{\nabla}\PP_{\not=}\RR_3\widetilde{\Delta}^{-\f12}N_{0,s}\right)
		-\PP_{0}|\pa_{y_1}|\V\cdot\widetilde{\nabla}\PP_{\not=}\RR_3\widetilde{\Delta}^{-\f12}N_{0,s}.
	\end{align}
	Then, applying \eqref{est-I411-1} and integrating by parts, we find that
	\begin{align}
		I_{411}
		=&
		\sum_{s=0}^{3}
		\left|\int_{\T\times\R}\left(
		\PP_{0}\V\cdot\widetilde{\nabla}\PP_{\not=}\RR_3\widetilde{\Delta}^{-\f12}N_{0,s}\PP_{\not=} \UU_{0,s}
		\right.
		\right.
		\label{est-I511-1}\\
		&\left.
		\left.
		\qquad
		-\PP_{0}|\pa_{y_1}|\V \cdot\widetilde{\nabla}\PP_{\not=}\RR_3\widetilde{\Delta}^{-\f12}N_{0,s}\PP_{\not=} \widetilde{\Delta}^{-\f12}\UU_{0,s}
		+\PP_{0}\V\cdot\widetilde{\nabla}\PP_{\not=}\UU_{0,s}\PP_{\not=} \RR_3\widetilde{\Delta}^{-\f12}N_{0,s}
		\right)\,dx_1dy_1\right|
		\nonumber\\
		=&\sum_{s=0}^{3}
		\left|\int_{\T\times\R}\big(
		\PP_{0}\pa_{y_1}V^2\PP_{\not=}\RR_3\widetilde{\Delta}^{-\f12}N_{0,s}\PP_{\not=} \UU_{0,s}
		+\PP_{0}|\pa_{y_1}|\V\cdot\widetilde{\nabla}\PP_{\not=}\RR_3\widetilde{\Delta}^{-\f12}N_{0,s}\PP_{\not=} \widetilde{\Delta}^{-\f12}\UU_{0,s}
		\big)\,dx_1dy_1\right|
		\nonumber
	\end{align}
	Invoking Lemma \ref{lem-est-NDW}, we can deduce from \eqref{est-I511-1} that
	\begin{align}
		\int_0^t I_{411}(\tau)\,d\tau
		\lesssim \sum_{s=0}^{3}
		\|\PP_{0}\pa_{y_1}\V\|_{L^\infty_tL^\infty}\|\PP_{\not=}N_{0,s}\|_{L^2_tL^2}\|\PP_{\not=}\UU_{0,s}\|_{L^2_tL^2}
		\lesssim \mu^{-\f13}|\E(t)|^{\f12}\int_0^t \D(\tau)\,d\tau.
		\label{est-I511}
	\end{align}
	Thanks to Lemmas \ref{sec2:lem-1}, \ref{sec2:lem-com-1-1} and \ref{lem-est-NDW}, we arrive at
	\begin{align}
		\int_0^t I_{412}(\tau)\,d\tau
		\lesssim&\sum_{s=0}^{3}
		\|\PP_{0}\pa_{y_1}\V\|_{L^\infty_tL^\infty}\left(\|\PP_{\not=}\nabla^s N\|_{L^2_tL^2}\|\PP_{\not=}\UU_{0,s}\|_{L^2_tL^2}
		+\|\PP_{\not=}\nabla^s \V\|_{L^2_tL^2}\|\PP_{\not=}N_{0,s}\|_{L^2_tL^2}\right)
		\label{est-I512}\\
		\lesssim&~\mu^{-\f12}|\E(t)|^{\f12}\int_0^t \D(\tau)\,d\tau.
		\nonumber
	\end{align}
	Invoking Lemmas \ref{sec2:lem-1}, \ref{sec2:lem-com-2}, and \ref{lem-est-NDW}, we conclude that
	\begin{align}
		\int_0^t I_{413}(\tau)\,d\tau			
		\lesssim&\sum_{s=1}^{3}
		\left(\|\PP_{0}\pa_{y_1}\V\|_{L^\infty_tH^2}\|\PP_{\not=}\widetilde{\nabla} N\|_{L^2_tH^2}
		\|\PP_{\not=}\UU_{0,s}\|_{L^2_tL^2}
		\right.
		\label{est-I513} \\
		&\left.
		\qquad+\|\PP_{0}\pa_{y_1}\V\|_{L^\infty_tH^2}\|\PP_{\not=}\widetilde{\nabla} \V\|_{L^2_tH^2}
		\|\PP_{\not=}N_{0,s}\|_{L^2_tL^2}
		\right)
		\nonumber\\
		\lesssim&~
		\mu^{-\f56}|\E(t)|^{\f12}\int_0^t \D(\tau)\,d\tau.	\nonumber
	\end{align}
	As for $I_{414}$, while it is easy to observe that
	\begin{align}\label{est-I414-1}
		\PP_{\not=}\RR_1\RR_3\nabla^s(\PP_{\not=}\V\cdot\widetilde{\nabla}N)
		=\PP_{\not=}\RR_1\RR_3\nabla^s\widetilde{\nabla}(\PP_{\not=}\V N)
		-\PP_{\not=}\RR_1\RR_3\nabla^s(\PP_{\not=}\widetilde{\dive}\V N).
	\end{align}
	Applying \eqref{est-I414-1} and using integration by parts, we find that
	\begin{align}
		I_{414}
		=&\sum_{s=0}^{3}
		\left|\int_{\T\times\R}\left(\PP_{\not=}\RR_1\RR_3\nabla^s\widetilde{\nabla}\widetilde{\Delta}^{-\f12}(\PP_{\not=}\V N)\PP_{\not=} \UU_{0,s}
		-\PP_{\not=}\RR_1\RR_3\nabla^s(\PP_{\not=}\widetilde{\dive}\V N)\PP_{\not=} \widetilde{\Delta}^{-\f12}\UU_{0,s}
		\right.\right.
		\label{est-I514-1}\\
		&\left.\left.
		\qquad+\PP_{\not=}\RR_1\RR_5\nabla^s(\PP_{\not=}\V\cdot\widetilde{\nabla}\V)\PP_{\not=} \RR_3\widetilde{\Delta}^{-\f12}N_{0,s}
		\right)\,dx_1dy_1\right|.
		\nonumber
	\end{align}
	Again thanks to  Lemmas \ref{sec2:lem-1} and \ref{lem-est-NDW}, it follows from \eqref{est-I514-1} that
	\begin{align}
		\int_0^t I_{414}(\tau)\,d\tau
		\lesssim&\sum_{s=0}^{3}
		\left(\|\PP_{\not=}\V\|_{L^2_tH^3}+\|\PP_{\not=}\widetilde{\nabla}\V\|_{L^2_tH^3}\right)
		\|N\|_{L^\infty_tH^3}
		\|\PP_{\not=}\UU_{0,s}\|_{L^2_tL^2}
		\label{est-I514}\\
		&
		\qquad+\sum_{s=0}^{3}\|\PP_{\not=}\V\|_{L^\infty_tH^3}\|\widetilde{\nabla}\V\|_{L^2_tH^3}
		\|\PP_{\not=}N_{0,s}\|_{L^2_tL^2}
		\nonumber\\
		\lesssim&~\mu^{-1}|\E(t)|^{\f12}\int_0^t \D(\tau)\,d\tau.
		\nonumber
	\end{align}
	Substituting \eqref{est-I511}-\eqref{est-I513} and \eqref{est-I514} into \eqref{est-I51-1}, we find that
	\begin{align}
		\int_0^t I_{41}(\tau)\,d\tau\lesssim \mu^{-1}|\E(t)|^{\f12}\int_0^t \D(\tau)\,d\tau.
		\label{est-I51-2}
	\end{align}
	It then follows from Lemmas \ref{sec2:lem-1} and \ref{lem-est-NDW} that
	\begin{align}
		\int_0^t I_{42}(\tau)\,d\tau
		\lesssim&\sum_{s=0}^{3}
		\left(\|\PP_{0}N\|_{L^\infty_tH^3}\|\PP_{\not=}D\|_{L^2_tH^3}
		+\|\PP_{\not=}N\|_{L^\infty_tH^3}\|D\|_{L^2_tH^3}
		\right)\|\PP_{\not=}\UU_{0,s}\|_{L^2_tL^2}
		\label{est-I52}\\
		\lesssim&~ \mu^{-1}|\E(t)|^{\f12}\int_0^t \D(\tau)\,d\tau.
		\nonumber
	\end{align}
	Integrating by parts, together with \eqref{est-fn-gn}, Lemmas \ref{sec2:lem-1} and \ref{lem-est-NDW}, we obtain
	\begin{align}
		\int_0^t I_{43}(\tau)\,d\tau
		\lesssim&\sum_{s=0}^{3}\int_0^t 
		\left|\int_{\T\times\R}
		\PP_{\not=}\RR_1\RR_5\widetilde{\nabla} (\tilde{f}(N))
		\PP_{\not=}\RR_3\widetilde{\Delta}^{-\f12}N_{0,s}\,dx_1dy_1\right|
		\,d\tau
		\label{est-I53-2}\\
		\lesssim&\sum_{s=0}^{3}\int_0^t 
		\left|\int_{\T\times\R}
		\PP_{\not=}\RR_1\RR_5 \tilde{f}(N)
		\PP_{\not=}\RR_3 N_{0,s}\,dx_1dy_1\right|
		\,d\tau
		\nonumber\\
		\lesssim&
		\sum_{s=0}^{3}
		\|N\|_{L^\infty_tH^3}\|\PP_{\not=} N\|_{L^2_tH^3}\|\PP_{\not=}N_{0,s}\|_{L^2_tL^2}
		\nonumber\\
		\lesssim&~\mu^{-\f23}|\E(t)|^{\f12}\int_0^t \D(\tau)\,d\tau.
		\nonumber
	\end{align}
	As for $I_{44}$, one can readily see that
	\begin{align}
		&\nabla^s\left(g(N)(\mu\widetilde{\Delta}\V+(\mu+\mu')\widetilde{\nabla}\widetilde{\dive}\V)\right)
		\label{est-I44-1}
		\\
		=&~g(N)\nabla^s\left(\mu\widetilde{\Delta}\V+(\mu+\mu')\widetilde{\nabla}\widetilde{\dive}\V\right)
		+[\nabla^s,g(N)]\left(\mu\widetilde{\Delta}\V+(\mu+\mu')\widetilde{\nabla}\widetilde{\dive}\V\right).
		\nonumber\\
		=&~\widetilde{\dive}\left(g(N)\nabla^s \left(\mu\widetilde{\nabla}\V+(\mu+\mu')\widetilde{\dive}\V\I\right)
		\right)
		+\widetilde{\nabla}g(N)\cdot \nabla^s \left(\mu\widetilde{\nabla}\V+(\mu+\mu')\widetilde{\dive}\V\I\right)
		\nonumber\\
		&+[\nabla^s,g(N)]\left(\mu\widetilde{\Delta}\V+(\mu+\mu')\widetilde{\nabla}\widetilde{\dive}\V\right).
		\nonumber
	\end{align}
	Applying \eqref{est-I44-1} and integration by parts again, it is easy to check that
	\begin{align}
		I_{44}
		\lesssim&\sum_{s=0}^{3}\left|\int_{\T\times\R}
		\PP_{\not=}\RR_1\RR_5\widetilde{\dive}\widetilde{\Delta}^{-\f12}\left(g(N)\nabla^s \left(\mu\widetilde{\nabla}\V+(\mu+\mu')\widetilde{\dive}\V\I\right)
		\right)
		\PP_{\not=}\RR_3 N_{0,s}\,dx_1dy_1\right|
		\nonumber\\
		&+\sum_{s=0}^{3}\left|\int_{\T\times\R}
		\PP_{\not=}\RR_1\RR_5\left(\widetilde{\nabla}g(N)\cdot\nabla^s \left(\mu\widetilde{\nabla}\V+(\mu+\mu')\widetilde{\dive}\V\I\right)
		\right)\PP_{\not=}\RR_3\widetilde{\Delta}^{-\f12}N_{0,s}\,dx_1dy_1\right|
		\nonumber\\
		&+\sum_{s=1}^{3}\left|\int_{\T\times\R}
		\PP_{\not=}\RR_1\RR_5\left([\nabla^s,g(N)]\left(\mu\widetilde{\Delta}\V+(\mu+\mu')\widetilde{\nabla}\widetilde{\dive}\V\right)
		\right)\PP_{\not=}\RR_3\widetilde{\Delta}^{-\f12}N_{0,s}\,dx_1dy_1\right|.
		\nonumber
	\end{align}
	This, together with \eqref{est-fn-gn}, Lemmas \ref{sec2:lem-1}, \ref{sec2:lem-com-2} and \ref{lem-est-NDW}, implies that
	\begin{align}
		\int_0^t I_{44}(\tau)\,d\tau 
		\lesssim&~\mu\sum_{s=0}^{3}
		\|\PP_{0}N\|_{L^\infty_tH^3}
		\|\PP_{\not=}\nabla^{s}\widetilde{\nabla}\V\|_{L^2_tL^2}\|\PP_{\not=} N_{0,s}\|_{L^2_tL^2}
		\label{est-I54-2}\\
		&+\mu\sum_{s=0}^{3}\left(\|\PP_{\not=}N\|_{L^\infty_tH^2}+ \|\PP_{\not=}\widetilde{\nabla}N\|_{L^\infty_tH^2}\right)
		\|\nabla^{s}\widetilde{\nabla}\V\|_{L^2_tL^2} 
		\|\PP_{\not=} N_{0,s}\|_{L^2_tL^2}
		\nonumber\\
		&	+\mu \sum_{s=1}^{3}
		\left(\|\PP_{0}\nabla N\|_{L^\infty_tH^2}\|\PP_{\not=}\widetilde{\nabla}^2\V\|_{L^2_tH^2}
		+\|\PP_{\not=}\nabla N\|_{L^\infty_tH^2}\|\widetilde{\nabla}^2\V\|_{L^2_tH^2}
		\right)\|\PP_{\not=} N_{0,s}\|_{L^2_tL^2}
		\nonumber\\
		\lesssim&~\mu^{-\f13}|\E(t)|^{\f12}\int_0^t \D(\tau)\,d\tau.
		\nonumber
	\end{align}
	Substituting \eqref{est-I51-2}-\eqref{est-I53-2} and \eqref{est-I54-2} into \eqref{est-I5-1}, we obtain \eqref{est-I5-2}.
	
	\subsubsection{Proof of Lemma \ref{lem1-5}}
	It is easy to check that
	\begin{align*}
		|I_5|\leq \mu^{\f13}\sum_{s=0}^3\left|\int_{\T\times\R}\left(\PP_{\not=}\FFF_{1,0,s}\PP_{\not=}\widetilde{\Delta}^{-\f12}\UU_{0,s}+\PP_{\not=}\FFF_{2,0,s}\PP_{\not=}\widetilde{\Delta}^{-\f12}N_{0,s}\right)\,dx_1dy_1\right|.
	\end{align*}
	Throughout the proof of $I_4$ , we treat $\RR_3$ as a bounded operator. Proceeding exactly as in the proof of \eqref{est-I5-2}, we immediately obtain \eqref{est-I6}; the details are omitted for brevity.		
	
	\section{Energy estimate of  incompressible part}
	We observe that $\D_{0,3}^{in}(t)$—the incompressible part of the solution—appears on the right–hand side of \eqref{est-0+1} and remains to be controlled. Accordingly, the present subsection is devoted to deriving the corresponding \textit{a priori} energy estimates for the higher-order derivatives of the vorticity, as stated in Proposition~\ref{prop-3}.

	\begin{prop}\label{prop-3}
		Under the assumptions of Theorem \ref{theo2} and \eqref{ass:density}, we have
		\begin{align}
			\E_{0,3}^{in}(t)+\int_0^t \D_{0,3}^{in}(\tau)\,d\tau \lesssim&~
			\E_{0,3}^{in}(0)
            +\mu^{\f23}\int_0^t
            \big(\D_{0,3}^{com,l}+\mu\D_{0,3}^{com}\big)(\tau)\,d\tau
			+\mu^{-1}|\E(t)|^{\f12}\int_0^t \D(\tau)\,d\tau.
			\label{est-2+2'}
		\end{align}
	\end{prop}

	\subsection{Proof of Proposition \ref{prop-3}}
	Before proceeding, we state the two key propositions on which the proof of Proposition~\ref{prop-3} relies.
	
	Since no $L^2$ estimate  is available for $\PP_{0}V^1$, we cannot bound $\|\PP_{0}W\|_{L^2}$ in $L^2_t$; only an $L^\infty_t$ bound is attainable.  This renders the transport term $\PP_{0}V^2\PP_{0}\pa_{y_1}W$ difficult to control directly.
	Inspired by the treatment of $(\PP_{0}N,\PP_{0}V^2)$ in the previous section, we derive a separate $L^2$ estimate for  $\PP_{0}W$ and exploit the density equation to cancel problematic terms such as $\PP_{0}V^2\PP_{0}\pa_{y_1}W$. Accordingly, we introduce the following definitions.
	\begin{align}
		&\widetilde{\E}_3(t)\tri
		\int_{\R}|\hat{W}_{0,0}^0|^2\,d\xi,
		\label{def-E-3}\\
		&	\widetilde{\E}_4(t)\tri\sum_{k\in\Z\backslash\{0\}}\int_{\R}|\hat{W}_{0,0}^k|^2\,d\xi
		+\sum_{s=1}^{3}\sum_{k\in\Z}\int_{\R}|\hat{W}_{0,s}^k|^2\,d\xi,\label{def-E-4},
	\end{align}
	and
	\begin{align}
		\widetilde{\D}_3(t)\tri&
		\mu\int_{\R}|\xi|^2|\hat{W}_{0,0}^0|^2\,d\xi, 
		\label{def-D-3}\\
		\widetilde{\D}_4(t)\tri &\sum_{k\in\Z\backslash\{0\}}\int_{\R}\left(\f{\pa_{t}m_1}{m_1}+\f{\pa_{t}m_2}{m_2}+\mu p\right)|\hat{W}_{0,0}^k|^2\,d\xi
		\label{def-D-4}\\
		&+\sum_{s=1}^{3}\sum_{k\in\Z}\int_{\R}\left(\f{\pa_{t}m_1}{m_1}+\f{\pa_{t}m_2}{m_2}+\mu p\right)|\hat{W}_{0,s}^k|^2\,d\xi,
		\nonumber
	\end{align}
	where $\widetilde{\E}_3(t)$ and $\widetilde{\D}_3(t)$ are the $L^2$ estimate for the zero-mode of the vorticity. And $\widetilde{\E}_4(t)$ and $\widetilde{\D}_4(t)$ involve
	the remaining terms in $\E_{0,3}^{in}(t)$ and $\D_{0,3}^{in}(t)$.
	The straightforward calculations yields that
	\begin{align}
		\E_{0,3}^{in}(t)\sim \widetilde{\E}_3(t)+\widetilde{\E}_4(t),\quad \D_{0,3}^{in}(t)\sim \widetilde{\D}_3(t)+\widetilde{\D}_4(t).
	\end{align}
	From these results, Proposition~\ref{prop-3} follows once we combine Propositions~\ref{prop-3-1} and \ref{prop-3-2}.

	\medskip
	
	We begin by deriving the $L^2$ estimate for the zero-mode of the vorticity.
	
	\begin{prop}\label{prop-3-1}
		Under the assumptions of Proposition \ref{prop-3}, we have
		\begin{align}
			\widetilde{\E}_3(t)+\int_0^t\widetilde{\D}_3(\tau)\,d\tau
			\lesssim&~
			\widetilde{\E}_3(0)
			+\mu^2\int_{0}^t\D_{0,3}^{com}(\tau)\,d\tau+\mu^{-1} \left(1+|\E(t)|^{\f12}\right)|\E(t)|^{\frac12}\int_0^t  \D(\tau) \, d\tau.
			\label{est-2'}
		\end{align}
	\end{prop}

	Next, we estimate the remaining terms in $\E_{0,3}^{in}(t)$ and $\D_{0,3}^{in}(t)$, namely $\widetilde{\E}_4(t)$ and $\widetilde{\D}_4(t)$.

	\begin{prop}\label{prop-3-2}
		Under the assumptions of Proposition \ref{prop-3}, we have
		\begin{align}
			\widetilde{\E}_{4}(t)+\int_0^t \widetilde{\D}_{4}(t)(\tau)\,d\tau 
			\lesssim&~
			\widetilde{\E}_{4}(0)
            +\mu^{\f23}\int_0^t
            \big(\D_{0,3}^{com,l}+\mu\D_{0,3}^{com}\big)(\tau)\,d\tau
			+\mu^{-1}|\E(t)|^{\frac12}\int_0^t  \D(\tau) \, d\tau.
			\label{est-2}
		\end{align}
	\end{prop}
	

	\subsection{Proof of Proposition \ref{prop-3-1}}
	We apply the operator $\PP_{0}$ to $\eqref{eq_n-d-w-1}_3$, to obtain that
	\begin{align}\label{eq-P0W}
		\pa_{t} \PP_{0}W-\mu \pa_{y_1}^2 \PP_{0}W+\mu(\mu+\mu') \pa_{y_1}^2 \PP_{0} D
		=-\pa_{y_1}(\PP_{0}W \PP_{0}V^2)+G_3,
	\end{align}
	where $G_3$ is defined by
	\begin{align}
		G_{3}\tri&-\PP_{0}(\PP_{\not=}\V\cdot \PP_{\not=}\widetilde{\nabla}W+\PP_{\not=}D\PP_{\not=}W)-\mu\PP_{0}(D^2)
		\label{def_G3}\\
		&+\PP_{0}\pa_{y_1}\left(g(N)(\nu\pa_{x_1}D-\mu(\pa_{y_1}-t\pa_{x_1})(W-N+\mu D))\right)+\mu G_{31},
		\nonumber\\
		G_{31}\tri&(\pa_{y_1}\PP_{0}V^2)^2+
		\PP_{0}\left[ (\PP_{\not=}\pa_{x_1}V^1)^2+(\PP_{\not=}(\pa_{y_1}-t\pa_{x_1})V^2)^2+2\PP_{\not=}(\pa_{y_1}-t\pa_{x_1})V^1\PP_{\not=}\pa_{x_1}V^2\right]
		\label{def_G31}\\
		&
		+\PP_{0}\pa_{y_1}\left(f(N)(\pa_{y_1}-t\pa_{x_1}) N\right)
		+\PP_{0}\pa_{y_1}\left(g(N)(\nu (\pa_{y_1}-t\pa_{x_1}) D+\mu \pa_{x_1}(W-N+\mu D))\right).
		\nonumber
	\end{align}
	Then applying $\PP_{0}$ to $\eqref{eq_n-d-w-1}_1$, we get
	\begin{align}\label{eq-P0N}
		\pa_{t}\PP_{0}N+\PP_{0}D=\PP_{0}\pa_{y_1}(N V^2).
	\end{align}
	
	\begin{proof}[\textbf{Proof of Proposition \ref{prop-3-1}}]
		Applying the standard $L^2$ energy estimate to equation \eqref{eq-P0W}, we obtain		
		\begin{align}
			&\f12\f{d}{dt}\|\PP_{0} W\|_{L^2}^2
			+\mu \|\pa_{y_1}\PP_{0} W\|_{L^2}^2
			\label{est-P0W-1}\\
			=&-\mu(\mu+\mu') \int_{\R}\pa_{y_1}^2 \PP_{0} D\PP_{0} W\,dy_1
			-\int_{\R} \pa_{y_1}(\PP_{0}W \PP_{0}V^2)\PP_{0} W\,dy_1
			+\int_{\R}G_3\PP_{0} W\,dy_1
			\nonumber\\
			=&\mu(\mu+\mu') \int_{\R}\pa_{y_1} \PP_{0} D\pa_{y_1}\PP_{0} W\,dy_1
			-\f12\int_{\R} \PP_{0}\pa_{y_1}V^2|\PP_{0} W|^2\,dy_1
			+\int_{\R}G_3\PP_{0} W\,dy_1.
			\nonumber
		\end{align}
		In order to eliminate the bad term $\int_{\R} \PP_{0}\pa_{y_1}V^2|\PP_{0} W|^2\,dy_1$ on the right-hand side of \eqref{est-P0W-1}, we multiply \eqref{eq-P0N} by $\f12|\PP_{0}W|^2$, to deduce that
		\begin{align}
			&\f12\int_{\R} \PP_{0}\pa_{y_1}V^2|\PP_{0} W|^2\,dy_1 \label{est-P0W-2} \\
			=&-\f12\int_{\R}\pa_{t}\PP_{0} N |\PP_{0} W|^2\,dy_1
			-\f12\int_{\R}\PP_{0}\pa_{y_1} (N V^2) |\PP_{0} W|^2\,dy_1
			\nonumber\\
			=&-\f12\f{d}{dt}\int_{\R}\PP_{0} N |\PP_{0} W|^2\,dy_1
			+\int_{\R}\PP_{0} N \PP_{0} W\pa_{t}\PP_{0} W\,dy_1
			-\f12\int_{\R}\PP_{0}\pa_{y_1} (N V^2) |\PP_{0} W|^2\,dy_1
			\nonumber\\
			=&-\f12\f{d}{dt}\int_{\R}\PP_{0} N |\PP_{0} W|^2\,dy_1
			+\int_{\R}\PP_{0} N \PP_{0} W
			\left(\mu\pa_{y_1}^2\PP_{0} W-\mu(\mu+\mu')\pa_{y_1}^2\PP_{0}D\right)\,dy_1
			\nonumber\\
			&-\int_{\R}\PP_{0} N \PP_{0} W\pa_{y_1}(\PP_{0}W\PP_{0} V^2)\,dy_1
			+\int_{\R}\PP_{0} N \PP_{0} W G_3\,dy_1
			-\f12\int_{\R}\PP_{0}\pa_{y_1} (N V^2) |\PP_{0} W|^2\,dy_1.
			\nonumber
		\end{align}
		Collecting \eqref{est-P0W-1} and \eqref{est-P0W-2} together, we achieve that
		\begin{align}
			&\f12\f{d}{dt}\int_{\R}(1-\PP_{0}N)|\PP_{0} W|^2\,dy_1
			+\mu \|\pa_{y_1}\PP_{0} W\|_{L^2}^2
			\label{est-P0W-3}\\
			=&\mu(\mu+\mu') \int_{\R}\pa_{y_1} \PP_{0} D\pa_{y_1}\PP_{0} W\,dy_1
			-\int_{\R}\PP_{0} N \PP_{0} W\left(\mu\pa_{y_1}^2 \PP_{0} W-\mu(\mu+\mu')\pa_{y_1}^2\PP_{0}D\right)\,dy_1
			\nonumber\\
			&+\int_{\R}\PP_{0} N \PP_{0} W\pa_{y_1}(\PP_{0}W\PP_{0} V^2)\,dy_1
			+\int_{\R}G_3\PP_{0} W(1-\PP_{0} N)\,dy_1
			+\f12\int_{\R}\PP_{0}\pa_{y_1} (N V^2) |\PP_{0} W|^2\,dy_1
			\nonumber\\
			\tri& \sum_{i=1}^{5}\J_i.
			\nonumber
		\end{align}
		By Young’s inequality, we readily obtain
		\begin{align}
			|\J_1| \leq \widetilde{\ep} \mu \|\pa_{y_1}\PP_{0} W\|_{L^2}^2+C_{\widetilde{\ep}}\mu^3 \|\pa_{y_1}\PP_{0} D\|_{L^2}^2
			\leq \widetilde{\ep} \mu \|\pa_{y_1}\PP_{0} W\|_{L^2}^2+C_{\widetilde{\ep}}\mu^2 \D_{0,3}^{com}.
			\label{est-cJ1}
		\end{align}
		Integration by parts yields that
		\begin{align}
			\int_0^t |\J_2|(\tau)\,d\tau 
			\lesssim&\int_0^t\int_{\R}\left|\pa_{y_1}\left(\PP_{0} N \PP_{0} W\right)\left(\mu\pa_{y_1}\PP_{0} W-\mu(\mu+\mu')\pa_{y_1}\PP_{0}D\right)\right|\,dy_1d\tau
			\label{est-cJ2}\\
			\lesssim&~\mu \|\PP_{0} (N, W)\|_{L^\infty_tL^\infty}
			\|\pa_{y_1}\PP_{0}( N,W)\|_{L^2_tL^2}\left( \|\pa_{y_1}\PP_{0} W\|_{L^2_tL^2}+\mu \|\pa_{y_1}\PP_{0} D\|_{L^2_tL^2}\right)
			\nonumber\\
			\lesssim&~|\E(t)|^{\f12}
			\int_0^t \D (\tau)\,d\tau.
			\nonumber
		\end{align}
		In light of Sobolev inequality and Lemma \ref{lem-est-NDW}, we get
		\begin{align}
			\int_0^t |\J_3|(\tau)\,d\tau 	
			\lesssim& \|\PP_{0} (N,W)\|_{L^4_tL^\infty}^2 \|\PP_{0}(W,V^2)\|_{L^\infty_tL^2}\|\pa_{y_1}\PP_{0}(W,V^2)\|_{L^2_tL^2}
			\label{est-cJ3}\\
			\lesssim&
			\|\PP_{0} (N,W)\|_{L^\infty_tH^2}\|\pa_{y_1}\PP_{0} (N,W)\|_{L^2_tH^2}
			\|\PP_{0}(W,V^2)\|_{L^\infty_tL^2}\|\pa_{y_1}\PP_{0}(W,V^2)\|_{L^2_tL^2}
			\nonumber\\
			\lesssim&~\mu^{-1}|\E(t)|^{\f12}\int_0^t \D(\tau)\,d\tau.
			\nonumber
		\end{align} 
		Recalling the definition of $G_3$ in \eqref{def_G3}, we can infer that
		\begin{align}
			|\J_4|=&\left|\int_{\R}(1-\PP_{0} N) \PP_{0} W \PP_{0}(\PP_{\not=}\V\cdot \PP_{\not=}\widetilde{\nabla}W+\PP_{\not=}D\PP_{\not=}W)\,dy_1\right|
			+\mu\left|\int_{\R}(1-\PP_{0} N) \PP_{0} W \PP_{0}(D^2)\,dy_1\right|
			\label{est-cJ4-1}\\
			&+\left|\int_{\R}(1-\PP_{0} N) \PP_{0} W \PP_{0}\pa_{y_1}\left(g(N)(\nu\pa_{x_1}D-\mu(\pa_{y_1}-t\pa_{x_1})(W-N+\mu D))\right)\,dy_1\right|
			\nonumber\\
			&+\mu\left|\int_{\R}(1-\PP_{0} N)\PP_{0} W G_{31}\,dy_1\right|
			\nonumber\\
			\tri& \sum_{i=1}^{4}\J_{4i}.
			\nonumber
		\end{align}
		We then estimate $\J_{4i}$ $(i=1,\cdots,4)$ term by term.
		Again thanks to Lemma \ref{lem-est-NDW} and Sobolev inequality, we find that
		\begin{align}
			\int_0^t |\J_{41}|(\tau)\,d\tau 	
			\lesssim&~ 
			(1+\|\PP_{0}N\|_{L^\infty_tH^2}) \|\PP_{0}W\|_{L^\infty_tH^2}
			\big(\|\PP_{\not=}\V\|_{L^2_tL^2}\|\PP_{\not=}\widetilde{\nabla} W\|_{L^2_tL^2}
			\label{est-cJ41}\\
			&
			+\|\PP_{\not=} D\|_{L^2_tL^2}\|\PP_{\not=} W\|_{L^2_tL^2}
			\big)
			\nonumber\\
			\lesssim&~ \mu^{-\f56}\left(1+|\E(t)|^{\f12}\right)|\E(t)|^{\f12}\int_0^t \D(\tau)\,d\tau,
			\nonumber
		\end{align}
		and
		\begin{align}
			\int_0^t |\J_{42}|(\tau)\,d\tau 
			\lesssim&~
			\mu (1+\|\PP_{0}N\|_{L^\infty_tL^\infty})
			\|\PP_{0}W\|_{L^\infty_tL^\infty}
			(\|\PP_{0}D\|_{L^2_tL^2}^2+\|\PP_{\not=}D\|_{L^2_tL^2}^2)
			\label{est-cJ42}\\
			\lesssim&~\mu^{-\f13}
			\left(1+|\E(t)|^{\f12}\right)|\E(t)|^{\f12}\int_0^t \D(\tau)\,d\tau.
			\nonumber
		\end{align}
		By integrating by parts, one obtains that
		\begin{align}
			\J_{43}
			\lesssim&~
			\mu
			\left|\int_{\R}\pa_{y_1}[(1-\PP_{0} N) \PP_{0} W] \left(\PP_{0}g(N)\pa_{y_1}\PP_{0}(W-N+\mu D)\right)\,dy_1\right|
			\label{est-cJ43-1}\\
			&+\left|\int_{\R}(1-\PP_{0} N) \PP_{0} W \PP_{0}\pa_{y_1}\left(\PP_{\not=}g(N)\PP_{\not=}(\nu\pa_{x_1}D-\mu(\pa_{y_1}-t\pa_{x_1})(W-N+\mu D))\right)\,dy_1\right|.
			\nonumber
		\end{align}
		Invoking Lemma \ref{lem-est-NDW} and Sobolev inequality, we infer from \eqref{est-cJ43-1} that
		\begin{align}
			\int_0^t |\J_{43}|(\tau)\,d\tau 
			\lesssim&~
			\mu (1+\|\PP_{0}(N,W)\|_{L^\infty_tH^2})\|\pa_{y_1}\PP_{0}(N,W)\|_{L^2_tL^2}
			\label{est-cJ43}\\
			&\times 
			\|\PP_{0}N\|_{L^\infty_tH^2} \|\pa_{y_1}\PP_{0}(W-N+\mu D)\|_{L^2_tL^2}+\mu (1+\|\PP_{0}N\|_{L^\infty_tH^2})
			\|\PP_{0}W\|_{L^\infty_tH^2}
			\nonumber\\
			&\times
			\|\PP_{\not=}N\|_{L^2_tH^1}
			\big(\|\nabla\PP_{\not=}D\|_{L^2_tH^1}
			+\|\widetilde{\nabla}\PP_{\not=}(W-N+\mu D)\|_{L^2_tH^1}\big)
			\nonumber\\
			\lesssim&~ \big(1+|\E(t)|^{\f12}\big)|\E(t)|^{\f12}
			\int_0^t \D(\tau)\,d\tau.
			\nonumber
		\end{align}
		Remembering the definition of $G_{31}$ in \eqref{def_G31}, we can get that
		\begin{align}
			|\J_{44}|
			\lesssim&~ \mu\left|\int_{\R}(1-\PP_{0} N)\PP_{0} W (\pa_{y_1}\PP_{0}V^2)^2\,dy_1\right|
			\label{est-cJ44-1}\\
			&+\mu\Big|\int_{\R}(1-\PP_{0} N)\PP_{0} W \PP_{0}\left[ (\PP_{\not=}\pa_{x_1}V^1)^2+(\PP_{\not=}(\pa_{y_1}-t\pa_{x_1})V^2)^2
			\right.	 \nonumber\\
			&\left. \qquad +2\PP_{\not=}(\pa_{y_1}-t\pa_{x_1})V^1\PP_{\not=}\pa_{x_1}V^2\right]\,dy_1\Big|
			\nonumber\\
			&+\mu\left|\int_{\R}(1-\PP_{0} N)\PP_{0} W \PP_{0}\pa_{y_1}\left(f(N)(\pa_{y_1}-t\pa_{x_1}) N\right)\,dy_1\right|
			\nonumber\\
			&+\mu\left|\int_{\R}(1-\PP_{0} N)\PP_{0} W \PP_{0}\pa_{y_1}\left(g(N)(\nu (\pa_{y_1}-t\pa_{x_1}) D+\mu \pa_{x_1}(W-N+\mu D))\right)\,dy_1\right|
			\nonumber\\
			\tri& \sum_{i=1}^{4}\J_{44i}.
			\nonumber
		\end{align}
		With the help of Soboev's inequality and Lemma \ref{lem-est-NDW}, we obtain
		\begin{align}
			\int_0^t |\J_{441}|(\tau)\,d\tau 
			\lesssim&~
			\mu (1+\|\PP_{0}N\|_{L^\infty_tL^\infty})
			\|\PP_{0}W\|_{L^\infty_tL^\infty}
			\|\PP_{0}D\|_{L^2_tL^2}^2
			\label{est-cJ441}\\
			\lesssim&~
			\big(1+|\E(t)|^{\f12}\big)|\E(t)|^{\f12}\int_0^t \D(\tau)\,d\tau,
			\nonumber
		\end{align}
		and
		\begin{align}
			\int_0^t |\J_{442}|(\tau)\,d\tau 
			\lesssim&~
			\mu (1+\|\PP_{0}N\|_{L^\infty_tL^\infty})
			\|\PP_{0}W\|_{L^\infty_tL^\infty}
			\|\PP_{\not=}\widetilde{\nabla}\V\|_{L^2_tL^2}^2
			\label{est-cJ442}\\
			\lesssim&~\mu^{-\f13}
			\big(1+|\E(t)|^{\f12}\big)|\E(t)|^{\f12}\int_0^t \D(\tau)\,d\tau.
			\nonumber
		\end{align}
		The derivation proceeds similarly to \eqref{est-cJ43-1} and \eqref{est-cJ43}, yielding
		\begin{align}
			\int_0^t |\J_{443}|(\tau)\,d\tau 
			\lesssim&~
			\mu \big(1+\|\PP_{0}(N,W)\|_{L^\infty_tH^2}\big)
			\|\pa_{y_1}\PP_{0}(N,W)\|_{L^2_tL^2}\|\PP_{0}N\|_{L^\infty_tH^2} \|\pa_{y_1}\PP_{0}N\|_{L^2_tL^2}
			\label{est-cJ443}\\
			&+	\mu  \big(1+\|\PP_{0}N\|_{L^\infty_tH^2}\big)\|\PP_{0}W\|_{L^\infty_tH^2}
			\|\PP_{\not=}N\|_{L^2_tH^1} \|\widetilde{\nabla}\PP_{\not=}N\|_{L^2_tH^1}
			\nonumber\\
			\lesssim&~ \big(1+|\E(t)|^{\f12}\big)|\E(t)|^{\f12}\int_0^t \D(\tau)\,d\tau.
			\nonumber
		\end{align}
		Employing the same manner as in  \eqref{est-cJ43-1} and \eqref{est-cJ43} again, we conclude that
		\begin{align}
			\int_0^t |\J_{444}|(\tau)\,d\tau 
			\lesssim&~
			\mu^2 (1+\|\PP_{0}(N,W)\|_{L^\infty_tH^2})\|\pa_{y_1}\PP_{0}(N,W)\|_{L^2_tL^2}\|\PP_{0}N\|_{L^\infty_tH^2} \|\pa_{y_1}\PP_{0}D\|_{L^2_tL^2}
			\label{est-cJ444}\\
			&+\mu^2 (1+\|\PP_{0}N\|_{L^\infty_tH^2})
			\|\PP_{0}W\|_{L^\infty_tH^2}
			\|\PP_{\not=}N\|_{L^2_tH^1}
			\nonumber\\
			&\times 
			\left(\|\widetilde{\nabla}\PP_{\not=}D\|_{L^2_tH^1}
			+\|\nabla\PP_{\not=}(W-N+\mu D)\|_{L^2_tH^1}\right)
			\nonumber\\
			\lesssim&~\mu^{\f23} \big(1+|\E(t)|^{\f12}\big)|\E(t)|^{\f12}\int_0^t \D(\tau)\,d\tau.
			\nonumber
		\end{align}
		Substituting \eqref{est-cJ441}-\eqref{est-cJ444} into \eqref{est-cJ44-1},  we achieve that
		\begin{align}
			\int_0^t |\J_{44}|(\tau)\,d\tau 	
			\lesssim \mu^{-\f13} \big(1+|\E(t)|^{\f12}\big)|\E(t)|^{\f12}\int_0^t \D(\tau)\,d\tau.
			\label{est-cJ44-2}
		\end{align}
		Plugging \eqref{est-cJ41}, \eqref{est-cJ42}, \eqref{est-cJ43} and \eqref{est-cJ44-2} into \eqref{est-cJ4-1}, we have
		\begin{align}
			\int_0^t |\J_{4}|(\tau)\,d\tau 
			\lesssim \mu^{-\f56} \big(1+|\E(t)|^{\f12}\big)|\E(t)|^{\f12}\int_0^t \D(\tau)\,d\tau.
			\label{est-cJ4-2}
		\end{align}
		Following the same approach as for $\J_3$, we bound $\J_5$ as follows:
		\begin{align}\label{est-cJ5}
			\int_0^t |\J_{5}|(\tau)\,d\tau 
			\lesssim&~
			\|\PP_{0}(N,V^2)\|_{L^\infty_tH^2}\|\pa_{y_1}\PP_{0}(N,V^2)\|_{L^2_tH^2}\|\PP_{0} W\|_{L^\infty_tL^2}\|\pa_{y_1}\PP_{0} W\|_{L^2_tL^2}
			\\
			&+\|\PP_{\not=}(N,V^2)\|_{L^2_tL^2}\|\PP_{\not=}(N,V^2)\|_{L^2_tH^1}\|\PP_{0} W\|_{L^\infty_tH^2}^2
			\nonumber\\
			\lesssim&~\mu^{-1}\E(t)
			\int_0^t \D(\tau)\,d\tau.
			\nonumber
		\end{align}
		
		Finally, we can conclude that \eqref{est-2'} holds by collecting \eqref{est-cJ1}-\eqref{est-cJ3}, \eqref{est-cJ4-2}, and \eqref{est-cJ5} together and choosing $\widetilde{\ep}$ suitably small. 
	\end{proof}

	\subsection{Proof of Proposition \ref{prop-3-2}}
	Recalling the definition of $\hat{W}_{j,s-j}^k$ in \eqref{def_W-j+1} for $j=0$, we then deduce from the equation $\eqref{eq_n-d-w-2}_3$ that
	\begin{align}
		\pa_{t}\hat{W}_{0,s}^k=&-\left[\f{\pa_{t}m_1}{m_1}+\f{\pa_{t}m_2}{m_2}+\mu p\right]\hat{W}_{0,s}^k-2\mu \f{k^2}{p^{\f32}}\hat{N}_{1,s}^k+\mu(\mu+\mu')p\hat{D}_{0,s}^k
		\label{eq_W-1-s}\\
		&-\mu\left(\f{\pa_{t}p}{p}-2\mu \f{k^2}{p}\right)\hat{D}_{0,s}^k+2\mu \f{k^2}{p}\hat{W}_{0,s}^k+\hat{\FFF}_{3,1,s}^k,\nonumber
	\end{align}
	where
	\begin{align}
		\hat{\FFF}_{3,1,s}^k\tri&m_1^{-1}m_2^{-1}<k,\xi>^s\hat{\FFF}_3^k.
	\end{align}

	\medskip
	
	\begin{proof}[\textbf{Proof of Proposition \ref{prop-3-2}}]
		It follows from equation \eqref{eq_W-1-s} that
		\begin{align}
			&\f12\f{d}{dt}|\hat{W}_{0,s}^k|^2+\left(\f{\pa_{t}m_1}{m_1}+\f{\pa_{t}m_2}{m_2}+\mu p\right)|\hat{W}_{0,s}^k|^2
			\label{est-energy-8}\\
			=&~\mu(\mu+\mu')p Re\left(\hat{D}_{0,s}^k\bar{\hat{W}}_{0,s}^k\right)
			-2\mu \f{k^2}{p^{\f32}} Re\left(\hat{N}_{1,s}^k\bar{\hat{W}}_{0,s}^k\right)
			\nonumber\\
			&
			-\mu \left(\f{\pa_{t}p}{p}-2\mu \f{k^2}{p}\right) Re\left(\hat{D}_{0,s}^k\bar{\hat{W}}_{0,s}^k\right)+2\mu \f{k^2}{p}|\hat{W}_{0,s}^k|^2
			+Re\left(\hat{\FFF}_{3,1,s}^k\bar{\hat{W}}_{0,s}^k\right).\nonumber
		\end{align}
		Applying \eqref{est-p-varphi} and Young's inequality again, the linear terms on the right-hand side of \eqref{est-energy-8} can be controlled by
		\begin{align}
			&\left|\mu(\mu+\mu')p Re\left(\hat{D}_{0,s}^k\bar{\hat{W}}_{0,s}^k\right)\right|
			\leq \widetilde{\ep} \mu p |\hat{W}_{0,s}^k|^2+C_{\widetilde{\ep}}\mu^3 p |\hat{D}_{0,s}^k|^2,
			\label{est-W-j0-1}\\
			&\left|2\mu \f{k^2}{p^{\f32}} Re\left(\hat{N}_{1,s}^k\bar{\hat{W}}_{0,s}^k\right)\right|
			\leq \widetilde{\ep} \f{k^2}{p}|\hat{W}_{0,s}^k|^2+C_{\widetilde{\ep}} \mu^2\f{k^2}{p^2}|\hat{N}_{1,s}^k|^2
            \leq \widetilde{\ep} \f{k^2}{p}|\hat{W}_{0,s}^k|^2+C_{\widetilde{\ep}} \mu^2|\hat{N}_{0,s}^k|^2,
			\label{est-W-j0-2}\\
			&\left|\mu \left(\f{\pa_{t}p}{p}-2\mu \f{k^2}{p}\right) Re\left(\hat{D}_{0,s}^k\bar{\hat{W}}_{0,s}^k\right)\right|
			\leq \widetilde{\ep} \mu^{\f13} |\hat{W}_{0,s}^k|^2+C_{\widetilde{\ep}}\mu^{\f53} \f{k^2}{p} |\hat{D}_{0,s}^k|^2
            \leq \widetilde{\ep} \mu^{\f13} |\hat{W}_{0,s}^k|^2+C_{\widetilde{\ep}}\mu^{\f53} |p^{\f12} \hat{\UU}_{0,s}^k|^2.
			\label{est-W-j0-3}
		\end{align}
		Employing \eqref{est-varphi-2}, and plugging all the estimates above into \eqref{est-energy-8}, then choosing $\widetilde{\ep}$ suitably small and $A$ defined in \eqref{def-m2} suitably large, finally performing the summation and integration, we arrive at
		\begin{align}
			\f{d}{dt}\widetilde{\E}_4(t)
			+\widetilde{\D}_4(t)
			\lesssim&\sum_{k\in\Z\backslash\{0\}}\int_{\R}\left\{\mu^2\f{k^2}{p}|\hat{N}_{1,0}^k|^2+\mu^3 p |\hat{D}_{0,0}^k|^2+\mu^{\f53} \f{k^2}{p} |\hat{D}_{0,0}^k|^2+Re\left(\hat{\FFF}_{3,1,0}^k\bar{\hat{W}}_{0,0}^k\right)\right\}\,d\xi
			\label{est-energy-9}\\
			&+ \sum_{s=1}^{3}\sum_{k\in\Z}\int_{\R}\left\{\mu^2\f{k^2}{p}|\hat{N}_{1,s}^k|^2+\mu^3 p |\hat{D}_{0,s}^k|^2+\mu^{\f53} \f{k^2}{p} |\hat{D}_{0,s}^k|^2+Re\left(\hat{\FFF}_{3,1,s}^k\bar{\hat{W}}_{0,s}^k\right)\right\}\,d\xi
			\nonumber\\
			\lesssim& ~\mu^{\f23}\D_{0,3}^{com,l}(t)+\mu^{\f53}\D_{0,3}^{com}(t)
			+\left|\int_{\T\times\R}
			\left(\PP_{\not=}\FFF_{3,1,0}\PP_{\not=}W_{0,0}+\sum_{s=1}^{3}\FFF_{3,1,s}W_{0,s}\right)\,dx_1dy_1
			\right|,
			\nonumber
		\end{align}
		where $\widetilde{\E}_4(t)$ and $\widetilde{\D}_4(t)$ are defined by \eqref{def-E-4} and \eqref{def-D-4}, respectively.

		We now begin to estimate the nonlinear term on the right-hand side of \eqref{est-energy-9}.
		\begin{align}
			&\left|\int_{\T\times\R}
			\left(\PP_{\not=}\FFF_{3,1,0}\PP_{\not=}W_{0,0}+\sum_{s=1}^{3}\FFF_{3,1,s}W_{0,s}\right)\,dx_1dy_1
			\right|
			\label{est-F3-1}\\
			\lesssim & 
			\left|\int_{\T\times\R}\left(\PP_{\not=}\RR_2(\V\cdot\widetilde{\nabla} W) \PP_{\not=}W_{0,0}+\sum_{s=1}^{3}\RR_2\nabla^s(\V\cdot\widetilde{\nabla} W) W_{0,s}\right) \,dx_1dy_1\right|
			\nonumber\\
			&+
			\left|\int_{\T\times\R}\left(\PP_{\not=}\RR_2(DW ) \PP_{\not=}W_{0,0}+\sum_{s=1}^{3}\RR_2\nabla^s(DW ) W_{0,s}\right) \,dx_1dy_1\right|
			\nonumber\\
			&
			+\mu 
			\left|\int_{\T\times\R}\left(\PP_{\not=}\RR_2(D^2) \PP_{\not=}W_{0,0}+\sum_{s=1}^{3}\RR_2\nabla^s(D^2) W_{0,s}\right) \,dx_1dy_1\right|
			\nonumber\\
			&+
			\left|\int_{\T\times\R}\left(\PP_{\not=}\RR_2\left(\widetilde{\nabla}^{\perp}\cdot\left(g(N)(\nu \widetilde{\nabla} D+\mu \widetilde{\nabla}^{\perp}(W-N+\nu D))\right)\right) \PP_{\not=}W_{0,0}
			\right.\right.
			\nonumber\\
			&\left.\left.
			+\sum_{s=1}^{3}\RR_2\nabla^s\left(\widetilde{\nabla}^{\perp}\cdot\left(g(N)(\nu \widetilde{\nabla} D+\mu \widetilde{\nabla}^{\perp}(W-N+\nu D))\right)\right) W_{0,s}\right) \,dx_1dy_1\right|
			\nonumber\\
			&
			+\mu
			\left|\int_{\T\times\R}\left(\PP_{\not=}\RR_2\widetilde{F}_{21} \PP_{\not=}W_{0,0}+\sum_{s=1}^{3}\RR_2\nabla^s\widetilde{F}_{21} W_{0,s}\right) \,dx_1dy_1\right|
			\nonumber\\
			\tri& \sum_{i=1}^{5}J_i,\nonumber
		\end{align}
		where $\widetilde{F}_{21}$ is defined in \eqref{def_wF21}.
		By a similar derivation as in \eqref{est-I10} and \eqref{est-I11-2}, we have
		\begin{align}
			\int_0^tJ_1(\tau)\,d\tau \lesssim  \mu^{-1} |\E(t)|^{\f12}\int_0^t \D(\tau)\,d\tau.
			\label{est-J1}
		\end{align}
		Using simialr decomposition as \eqref{est-I122-1},  together with Lemma \ref{sec2:lem-1} and Sobolev inequality, we deduce that
		\begin{align}
			\int_0^tJ_2(\tau)\,d\tau
			\lesssim& ~	
			\sum_{s=1}^{3}
			\|\PP_{0}D\|_{L^2_tH^3}\|\PP_{0}W\|_{L^\infty_tH^3}\|\PP_0W_{0,s}\|_{L^2_tL^2}
			\label{est-J2}\\
			&
			+ \sum_{s=0}^{3}
			\|\PP_{\not=}D\|_{L^2_tH^3}\|\PP_{\not=}W\|_{L^2_tH^3}
			\|W_{0,s}\|_{L^\infty_tL^2}
			\nonumber\\
			&
			+ \sum_{s=0}^{3}
			\left(\|\PP_{0}D\|_{L^2_tH^3}\|\PP_{\not=}W\|_{L^\infty_tH^3}
			+\|\PP_{\not=}D\|_{L^2_tH^3}\|\PP_{0}W\|_{L^\infty_tH^3}
			\right)\|\PP_{\not=}W_{0,s}\|_{L^2_tL^2}
			\nonumber\\
			\lesssim&~ \mu^{-1}|\E(t)|^{\f12}\int_0^t \D(\tau)\,d\tau.
			\nonumber
		\end{align}
		Again thanks to Sobolev inequality and Lemma \ref{sec2:lem-1}, by a similar procedure, we find that
		\begin{align}
			\int_0^tJ_3(\tau)\,d\tau
			\lesssim& ~ \mu\sum_{s=1}^{3}
			\|\PP_{0}D\|_{L^2_tH^2}
			\|\PP_{0}D\|_{L^2_tH^3}\|\PP_0W_{0,s}\|_{L^\infty_tL^2}
			\label{est-J3}\\
			&+\mu\sum_{s=0}^{3} \|\PP_{\not=}D\|_{L^2_tH^2}\|\PP_{\not=}D\|_{L^2_tH^3}\|W_{0,s}\|_{L^\infty_tL^2}
			\nonumber\\
			& +\mu\sum_{s=0}^{3}\left[ \left(\|\PP_{0}D\|_{L^\infty_tH^2}\|\PP_{\not=}D\|_{L^2_tH^3}+
			\|\PP_{\not=}D\|_{L^\infty_tH^2}\|\PP_{0}D\|_{L^2_tH^3}\right)\|\PP_{\not=}W_{0,s}\|_{L^2_tL^2}
			\right]
			\nonumber\\
			\lesssim& ~\mu^{-\f13}|\E(t)|^{\f12}\int_0^t \D(\tau)\,d\tau.
			\nonumber
		\end{align}
		Invoking the Sobolev inequality and Lemma~\ref{sec2:lem-1} once more, we integrate by parts to obtain
		\begin{align}
			\int_0^tJ_4(\tau)\,d\tau
			\lesssim&~ \mu\sum_{s=1}^3
			\|\PP_{0} N\|_{L^\infty_tH^3}
			\|\PP_{0} \pa_{y_1} (D,W,N)\|_{L^2_tH^3}
			\|\PP_{0} \pa_{y_1} W_{0,s}\|_{L^2_tL^2}
			\label{est-J4}\\
			&+\mu\sum_{s=0}^{3}
			\|\PP_{\not=} N\|_{L^\infty_tH^3}
			\|\PP_{\not=} \widetilde{\nabla} (D,W,N)\|_{L^2_tH^3}
			\| \widetilde{\nabla} W_{0,s}\|_{L^2_tL^2}
			\nonumber\\
			&+
			\mu\sum_{s=0}^{3}
			\left(
			\|\PP_{0} N\|_{L^\infty_tH^3}
			\|\PP_{\not=} \widetilde{\nabla} (D,W,N)\|_{L^2_tH^3}
			\right.
			\nonumber\\
			&\left.
			+
			\|\PP_{\not=} N\|_{L^\infty_tH^3}
			\|\PP_{0} \pa_{y_1}(D,W,N)\|_{L^2_tH^3}
			\right)\| \PP_{\not=}\widetilde{\nabla} W_{0,s}\|_{L^2_tL^2}
			\nonumber\\
			\lesssim&~ \mu^{-1}|\E(t)|^{\f12}\int_0^t \D(\tau)\,d\tau.
			\nonumber
		\end{align}
		In order to bound the last term $J_5$, we split it into the following three parts:
		\begin{align}
			J_5\lesssim&~
			\mu
			\left|\int_{\T\times\R}\Big(\PP_{\not=}\RR_2\left[(\pa_{x_1}V^1)^2+((\pa_{y_1}-t\pa_{x_1})V^2)^2+2(\pa_{y_1}-t\pa_{x_1})V^1\pa_{x_1}V^2
			\right] \PP_{\not=}W_{0,0}
			\right.
			\label{est-J5-1}\\
			&\left.
			+\sum_{s=1}^{3}\RR_2\nabla^s\left[(\pa_{x_1}V^1)^2+((\pa_{y_1}-t\pa_{x_1})V^2)^2+2(\pa_{y_1}-t\pa_{x_1})V^1\pa_{x_1}V^2
			\right] W_{0,s} \Big)\,dx_1dy_1\right|
			\nonumber\\
			&+\mu
			\left|\int_{\T\times\R}
			\left(\PP_{\not=}\RR_2\widetilde{\dive}\left(f(N)\widetilde{\nabla} N\right)
			\PP_{\not=}W_{0,0}
			+\sum_{s=1}^{3}\RR_2\nabla^s\widetilde{\dive}\left(f(N)\widetilde{\nabla} N\right)
			W_{0,s} \right)\,dx_1dy_1\right|
			\nonumber\\
			&+\mu
			\left|\int_{\T\times\R}\left(\PP_{\not=}\RR_2
			\widetilde{\dive} \left(g(N)(\nu \widetilde{\nabla} D+\mu \widetilde{\nabla}^{\perp}(W-N+\nu D))\right) \PP_{\not=}W_{0,0}
			\right.\right. 
			\nonumber\\
			&\left.\left.
			+\sum_{s=1}^{3}\RR_2\nabla^s
			\widetilde{\dive} \left(g(N)(\nu \widetilde{\nabla} D+\mu \widetilde{\nabla}^{\perp}(W-N+\nu D))\right) W_{0,s}\right) \,dx_1dy_1\right|
			\nonumber\\
			\tri&~ J_{51}+J_{52}+J_{53}.
			\nonumber
		\end{align}
		Then we estimate $J_{5i} (i=1,2,3)$ term by term.
		Since $\PP_{0}\pa_{x_1}q=0$, we observe that
		\begin{align}
			J_{51}
			\lesssim&~ \mu\sum_{s=1}^{3}
			\left|\int_{\T\times\R}\RR_2\pa_{y_1}^s(\pa_{y_1}\PP_{0}V^2)^2
			\PP_{0} W_{0,s} \,dx_1dy_1\right|
			\label{est-J51-1}\\
			&+\mu\sum_{s=0}^{3}
			\Big|\int_{\T\times\R}\RR_2\nabla^s [(\pa_{x_1}\PP_{\not=}V^1)^2+((\pa_{y_1}-t\pa_{x_1})\PP_{\not=}V^2)^2
			\nonumber\\
			&\qquad +2(\pa_{y_1}-t\pa_{x_1})\PP_{\not=}V^1\pa_{x_1}\PP_{\not=}V^2
			] W_{0,s} \,dx_1dy_1 \Big|
			\nonumber\\
			&+\mu\sum_{s=0}^{3}
			\left|\int_{\T\times\R}\RR_2\nabla^s\left[\pa_{y_1}\PP_{0}V^2(\pa_{y_1}-t\pa_{x_1})\PP_{\not=}V^2
			+2\pa_{y_1}\PP_{0}V^1\pa_{x_1}\PP_{\not=}V^2
			\right] \PP_{\not=}W_{0,s} \,dx_1dy_1\right|.
			\nonumber
		\end{align}
		Then employing Lemma \ref{sec2:lem-1} and Sobolev inequality, we can deduce from \eqref{est-J51-1} that
		\begin{align}
			\int_0^tJ_{51}(\tau)\,d\tau
			\lesssim&~\mu
			\sum_{s=1}^{3}\|\PP_{0}\pa_{y_1} V^2\|_{L^\infty_tH^2}\|\PP_{0}\pa_{y_1}V^2\|_{L^2_tH^3}\|\PP_{0}W_{0,s}\|_{L^2_tL^2}
			\label{est-J51}\\
			&	+\mu \sum_{s=0}^{3}
			\|\PP_{\not=}\widetilde{\nabla} \V\|_{L^2_tH^2}\|\PP_{\not=}\widetilde{\nabla}\V\|_{L^2_tH^3}\|W_{0,s}\|_{L^\infty_tL^2}
			\nonumber\\
			&+\mu \sum_{s=0}^{3}
			\|\PP_{0}\pa_{y_1}\V \|_{L^\infty_tH^2}\|\PP_{\not=}\widetilde{\nabla}\V\|_{L^2_tH^3}\|\PP_{\not=}W_{0,s}\|_{L^2_tL^2}
			\nonumber\\
			&	+\mu\sum_{s=0}^{2}\|\PP_{\not=}\widetilde{\nabla} \V\|_{L^2_tH^2}\|\PP_{0}\pa_{y_1}\V\|_{L^\infty_tH^3}\|\PP_{\not=}W_{0,s}\|_{L^2_tL^2}
			\nonumber\\
			&
			+\mu\|\PP_{\not=}\widetilde{\nabla} \V\|_{L^\infty_tH^2}\|\PP_{0}\pa_{y_1}^{4}\V\|_{L^2_tL^2}\|\PP_{\not=}W_{0,s}\|_{L^2_tL^2}
			\nonumber\\
			\lesssim&~
			\mu^{-\f13}|\E(t)|^{\f12}\int_0^t \D(\tau)\,d\tau.
			\nonumber
		\end{align}
		Regarding the term $J_{52}$, integrating by parts, we have
		\begin{align}\label{est-J52-1}
			J_{52}
			=\mu
			\left|\int_{\T\times\R}\left(\PP_{\not=}\RR_2\left(f(N)\widetilde{\nabla} N\right)\cdot
			\widetilde{\nabla} \PP_{\not=}W_{0,0} +\sum_{s=1}^{3}\RR_2\nabla^s\left(f(N)\widetilde{\nabla} N\right)\cdot
			\widetilde{\nabla} W_{0,s}\right) \,dx_1dy_1\right|.
		\end{align}
		Invoking Lemma \ref{sec2:lem-1} together with the Sobolev inequality, we deduce from \eqref{est-J52-1} that
		\begin{align}
			\int_0^tJ_{52}(\tau)\,d\tau
			\lesssim&~
			\mu\sum_{s=1}^{3}
			\|\PP_{0}N\|_{L^\infty_tH^3} \|\PP_{0}\pa_{y_1}N\|_{L^2_tH^3}
			\|\PP_{0}\widetilde{\nabla}W_{0,s}\|_{L^2_tL^2}
			\label{est-J52}\\
			&+\mu\sum_{s=0}^{3}
			\|\PP_{\not=}N\|_{L^\infty_tH^3}
			\|\PP_{\not=}\widetilde{\nabla}N\|_{L^2_tH^3}
			\|\widetilde{\nabla}W_{0,s}\|_{L^2_tL^2}
			\nonumber\\
			&+\mu\sum_{s=0}^{3}
			\left(\|\PP_{0}N\|_{L^\infty_tH^3}
			\|\PP_{\not=}\widetilde{\nabla}N\|_{L^2_tH^3}
			+\|\PP_{\not=}N\|_{L^2_tH^3}\|\PP_{0}\pa_{y_1}N\|_{L^\infty_tH^3}
			\right)\|\PP_{\not=}\widetilde{\nabla}W_{0,s}\|_{L^2_tL^2}
			\nonumber\\
			\lesssim&~\mu^{-\f12}|\E(t)|^{\f12}\int_0^t \D(\tau)\,d\tau.
			\nonumber
		\end{align}
		Following a similar way as $J_{52}$, one gets that
		\begin{align}
			\int_0^tJ_{53}(\tau)\,d\tau 	
			\lesssim&~\mu^2\sum_{s=1}^{3}
			\|\PP_{0}N\|_{L^\infty_tH^3}\|\PP_{0}\pa_{y_1}(D+W+N)\|_{L^2_tH^3}\|\PP_{0}\pa_{y_1}W_{0,s}\|_{L^2_tL^2}
			\label{est-J53}\\
			&
			+\mu^2\sum_{s=0}^{3}
			\|\PP_{\not=}N\|_{L^\infty_tH^3}\|\PP_{\not=}\widetilde{\nabla}(D+W+N)\|_{L^2_tH^3}\|\widetilde{\nabla}W_{0,s}\|_{L^2_tL^2}
			\nonumber\\
			&
			+
			\mu^2\sum_{s=0}^{3}	\left(\|\PP_{0}N\|_{L^\infty_tH^3}\|\PP_{\not=}\widetilde{\nabla}(D+W+N)\|_{L^2_tH^3}
			\right.
			\nonumber\\
			&\left.
			+\|\PP_{\not=}N\|_{L^\infty_tH^3}\|\PP_{0}\pa_{y_1}(D+W+N)\|_{L^2_tH^3}
			\right)\|\PP_{\not=}\widetilde{\nabla}W_{0,s}\|_{L^2_tL^2}
			\nonumber\\
			\lesssim&~|\E(t)|^{\f12}\int_0^t \D(\tau)\,d\tau.
			\nonumber
		\end{align}
		Substituting \eqref{est-J51}, \eqref{est-J52} and \eqref{est-J53} into \eqref{est-J5-1}, we find that
		\begin{align}
			\int_0^tJ_{5}(\tau)\,d\tau
			\lesssim \mu^{-\f12}|\E(t)|^{\f12}\int_0^t \D(\tau)\,d\tau.	
			\label{est-J5-2}
		\end{align}
		Plugging \eqref{est-J1}-\eqref{est-J4} and \eqref{est-J5-2} into \eqref{est-F3-1}, we arrive at
		\begin{align}
			\int_0^t\left|\int_{\T\times\R}\left(\PP_{\not=}\FFF_{3,1,0}\PP_{\not=}W_{0,0}+
			\sum_{s=1}^{3}\FFF_{3,1,s}W_{0,s}\right)\,dx_1dy_1
			\right|\,d\tau
			\lesssim \mu^{-1}|\E(t)|^{\f12}\int_0^t \D(\tau)\,d\tau
		\end{align}
		This, along with \eqref{est-energy-9}, yields \eqref{est-2} immediately.
	\end{proof}

	\section{Higher order energy estimate of compressible part} 
	
	We observe that the term $\D_{0,3}^{com}(t)$, appearing on the right-hand side of \eqref{est-2+2'}, has yet to be estimated. To close the energy estimates, we therefore need to bound the high-order derivatives of the density and the divergence, namely $\E_{0,3}^{com}(t)$ and $\D_{0,3}^{com}(t)$.
	We formulate the required bounds in the following proposition.
	
	\begin{prop}\label{prop-4}
		Under the assumptions of Theorem \ref{theo2} and \eqref{ass:density}, we have
		\begin{align}
			&\mu \E_{0,3}^{com}(t)+\mu \int_0^t \D_{0,3}^{com}(\tau)\,d\tau
			\label{est-3}\\
			\lesssim&~
			\mu \E_{0,3}^{com}(0)
			+ \int_0^t \D_{0,3}^{com,l}(\tau)\,d\tau
			+\f{1}{A}\mu 
			\int_0^t \D_{0,3}^{in}(\tau)\,d\tau +
			\mu^{-1} |\E(t)|^{\frac12}\int_0^t  \D(\tau) \, d\tau. 
			\nonumber
		\end{align}
	\end{prop}
	Before proving Proposition~\ref{prop-4}, we first recall definitions of $\hat{N}_{j+1,s-j}^k$ and $\hat{D}_{j,s-j}^k$ in \eqref{def_N-j+1} and \eqref{def_D-j+1} with $j=0$,  and then derive from equations $\eqref{eq_n-d-w-2}_1$ and $\eqref{eq_n-d-w-2}_2$ that
	\begin{align}
		\pa_{t}\hat{N}_{1,s}^k=&\f12\f{\pa_{t}p}{p}\hat{N}_{1,s}^k
		-p^{\f12}\hat{D}_{0,s}^k+\hat{\FFF}_{1,1,s}^k,
		\label{eq_N-1-s}\\
		\pa_{t}\hat{D}_{0,s}^k=&-\nu p\hat{D}_{0,s}^k+\f{\pa_{t}p}{p}\hat{D}_{0,s}^k
		+\left(p^{\f12}+2\f{k^2}{p^{\f32}}\right)\hat{N}_{1,s}^k
		-2\mu \f{k^2}{p}\hat{D}_{0,s}^k
		-2\f{k^2}{p}\hat{W}_{0,s}^k
		+\hat{\FFF}_{2,1,s}^k,
		\label{eq_D-1-s}
	\end{align}
	where $\hat{\FFF}_{1,1,s}^k$ and $\hat{\FFF}_{2,1,s}^k$ are defined by
	\begin{align}
		\hat{\FFF}_{1,1,s}^k\tri&<k,\xi>^sp^{\f12}\hat{\FFF}_1^k,
		\\
		\hat{\FFF}_{2,1,s}^k\tri&
        <k,\xi>^s\hat{\FFF}_2^k.
	\end{align}
	By some direct computations, we can deduce from  \eqref{eq_N-1-s} and \eqref{eq_D-1-s} that
	\begin{align}
		&\f12\f{d}{dt}|(\hat{N}_{1,s}^k,\hat{D}_{0,s}^k)|^2
		+\nu p |\hat{D}_{0,s}^k|^2
		\label{est-energy-10}\\
		=&~\f12\f{\pa_{t}p}{p} |\hat{N}_{1,s}^k|^2
		+\f{\pa_{t}p}{p}|\hat{D}_{0,s}^k|^2+2\f{k^2}{p^{\f32}}Re\left(\hat{N}_{1,s}^k\bar{\hat{D}}_{0,s}^k\right)-2\mu \f{k^2}{p}|\hat{D}_{0,s}^k|^2-2\f{k^2}{p}Re\left(\hat{W}_{0,s}^k\bar{\hat{D}}_{0,s}^k\right)
		\nonumber\\
		&+Re\left(\hat{\FFF}_{1,1,s}^k\bar{\hat{N}}_{1,s}^k\right)+Re\left(\hat{\FFF}_{2,1,s}^k\bar{\hat{D}}_{0,s}^k\right).\nonumber
	\end{align}
	To cancel the first term on the right-hand side of \eqref{est-energy-10}, we require		
	\begin{align}
		-\f{d}{dt}\left[\f{\pa_{t}p}{p^{\f32}}Re\left(\hat{N}_{1,s}^k\bar{\hat{D}}_{0,s}^k\right)\right] 	
		=&-\f{\pa_{t}p}{p}|\hat{N}_{1,s}^k|^2+\f{\pa_{t}p}{p}|\hat{D}_{0,s}^k|^2-2\f{k^2\pa_{t}p}{p^3}|\hat{N}_{1,s}^k|^2
		\label{est-energy-11}
        \\
		&+\left(\nu \f{\pa_{t}p}{p^{\f12}}+2\mu \f{k^2\pa_{t}p}{p^{\f52}}-2\f{k^2}{p^{\f32}}
		\right)Re\left(\hat{N}_{1,s}^k\bar{\hat{D}}_{0,s}^k\right)
		\nonumber\\
		&+2\f{k^2\pa_{t}p}{p^{\f52}}Re\left(\hat{N}_{1,s}^k\bar{\hat{W}}_{0,s}^k\right)-\f{\pa_{t}p}{p^{\f32}}\left[Re\left(\hat{\FFF}_{1,1,s}^k\bar{\hat{D}}_{0,s}^k\right)+Re\left(\hat{\FFF}_{2,1,s}^k\bar{\hat{N}}_{1,s}^k\right) \right].\nonumber
	\end{align}
	Multiplying \eqref{est-energy-11} by $1/2$, adding the resulting equality to \eqref{est-energy-10}, finally performing the summation and integration, we find that
	\begin{align}
		&\f12\f{d}{dt}\sum_{s=0}^{3}\sum_{k\in\Z}\int_{\R}\left[|\hat{N}_{1,s}^k|^2+|\hat{D}_{0,s}^k|^2
		-\f{\pa_{t}p}{p^{\f32}}Re\left(\hat{N}_{1,s}^k\bar{\hat{D}}_{0,s}^k\right)
		\right]\,d\xi
        +\sum_{s=0}^{3}\sum_{k\in\Z}\int_{\R}\nu p|\hat{D}_{0,s}^k|^2\,d\xi
		\label{est-energy-11-1}\\
		=&\sum_{s=0}^{3}\sum_{k\in\Z\backslash\{0\}}\int_{\R}\Big\{-\f{k^2\pa_{t}p}{p^{3}}|\hat{N}_{1,s}^k|^2
		+\left(\f32\f{\pa_{t}p}{p}
		-2\mu\f{k^2}{p}\right)|\hat{D}_{0,s}^k|^2
		\nonumber\\
		&\quad+\Big(\f12\nu\f{\pa_{t}p}{p^{\f12}}+\mu \f{k^2\pa_{t}p}{p^{\f52}}
		+\f{k^2}{p^{\f32}}
		\Big)
		Re\left(\hat{N}_{1,s}^k\bar{\hat{D}}_{0,s}^k\right)
		\nonumber\\
		&\quad+2\f{k^2}{p^{\f32}}\f{\pa_{t}p}{p}Re\left(\hat{N}_{1,s}^k\bar{\hat{W}}_{0,s}^k\right)
		-2\f{k^2}{p}Re\left(\hat{D}_{0,s}^k\bar{\hat{W}}_{0,s}^k\right)
		\Big\}\,d\xi
		\nonumber\\
		&\quad+\sum_{s=0}^{3}\sum_{k\in\Z}\int_{\R}\left[Re\left(\hat{\FFF}_{1,1,s}^k\bar{\hat{N}}_{1,s}^k\right)
		+Re\left(\hat{\FFF}_{2,1,s}^k\bar{\hat{D}}_{0,s}^k\right)\right]\,d\xi
		\nonumber\\
		&\quad-\sum_{s=0}^{3}\sum_{k\in\Z\backslash\{0\}}\int_{\R}\f{\pa_{t}p}{p^{\f32}}\left[Re\left(\hat{\FFF}_{1,1,s}^k\bar{\hat{D}}_{0,s}^k\right)+Re\left(\hat{\FFF}_{2,1,s}^k\bar{\hat{N}}_{1,s}^k\right)\right]\,d\xi
		.\nonumber
	\end{align} 
	For the sake of obtaining the dissipation for $\hat{N}_{1,s}^k$, we infer that
	\begin{align}
		&-\f{d}{dt}\left[p^{-\f12}Re\left(\hat{N}_{1,s}^k\bar{\hat{D}}_{0,s}^k\right)\right]
		+\left(1+2\f{k^2}{p^2}\right)|\hat{N}_{1,s}^k|^2
		\label{est-energy-12}\\
		=&~|\hat{D}_{0,s}^k|^2+\f{1}{p^{\f12}}\left(-\f{\pa_{t}p}{p}+2\mu \f{k^2}{p}+\nu p\right)Re\left(\hat{N}_{1,s}^k\bar{\hat{D}}_{0,s}^k\right)
		\nonumber\\
		&+2\f{k^2}{p^{\f32}}Re\left(\hat{N}_{1,s}^k\bar{\hat{W}}_{0,s}^k\right)-\f{1}{p^{\f12}}\left[Re\left(\hat{\FFF}_{1,1,s}^k\bar{\hat{D}}_{0,s}^k\right)+Re\left(\hat{\FFF}_{2,1,s}^k\bar{\hat{N}}_{1,s}^k\right)\right].\nonumber
	\end{align}
	For $k\not=0$, we multiply \eqref{est-energy-12} by $\delta_2\mu^{\f13}$, then integrate over $\xi$ in $\R$ and sum over $k$ in $\Z\backslash\{0\}$, to find that
	\begin{align}
		&-\delta_2\mu^{\f13}\f{d}{dt}\sum_{s=0}^{3}\sum_{k\in\Z\backslash\{0\}}\int_{\R}\left[p^{-\f12}Re\left(\hat{N}_{1,s}^k\bar{\hat{D}}_{0,s}^k\right)\right]\,d\xi
		+\delta_2\mu^{\f13}\sum_{s=0}^{3}\sum_{k\in\Z\backslash\{0\}}\int_{\R}\left(1+2\f{k^2}{p^2}\right)|\hat{N}_{1,s}^k|^2\,d\xi
		\label{est-energy-12-1}\\
		=&~\delta_2\mu^{\f13}\sum_{s=0}^{3}\sum_{k\in\Z\backslash\{0\}}\int_{\R}\left\{|\hat{D}_{0,s}^k|^2+\f{1}{p^{\f12}}\left(-\f{\pa_{t}p}{p}+2\mu \f{k^2}{p}+\nu p\right)Re\left(\hat{N}_{1,s}^k\bar{\hat{D}}_{0,s}^k\right)\right.
		\nonumber\\
		&\left.+2\f{k^2}{p^{\f32}}Re\left(\hat{N}_{1,s}^k\bar{\hat{W}}_{0,s}^k\right)
		-\f{1}{p^{\f12}}\left[Re\left(\hat{\FFF}_{1,1,s}^k\bar{\hat{D}}_{0,s}^k\right)+Re\left(\hat{\FFF}_{2,1,s}^k\bar{\hat{N}}_{1,s}^k\right)\right]\right\}\,d\xi.\nonumber
	\end{align}
	While for $k=0$, we multiply \eqref{est-energy-12} by $\delta_2\mu$, and integrate over $\xi$ in $\R$, to get that
	\begin{align}
		&-\delta_2\mu\f{d}{dt}\sum_{s=1}^{3}\int_{\R}\left[p^{-\f12}Re\left(\hat{N}_{1,s}^0\bar{\hat{D}}_{0,s}^0\right)\right]\,d\xi
		+\delta_2\mu\sum_{s=1}^{3}\int_{\R}|\hat{N}_{1,s}^0|^2
		\label{est-energy-12-2}\\
		=&~\delta_2\mu\sum_{s=1}^{3}\int_{\R}\left\{|\hat{D}_{0,s}^0|^2+\nu |\xi|Re\left(\hat{N}_{1,s}^0\bar{\hat{D}}_{0,s}^0\right)-\f{1}{p^{\f12}}\left[Re\left(\hat{\FFF}_{1,1,s}^0\bar{\hat{D}}_{0,s}^0\right)+Re\left(\hat{\FFF}_{2,1,s}^0\bar{\hat{N}}_{1,s}^0\right)\right]\right\}\,d\xi.
		\nonumber
	\end{align}
	We  define 
	\begin{align*}
		\widetilde{\E}_{5}(t)\tri& \sum_{s=0}^{3}\sum_{k\in\Z}\int_{\R}\left[|\hat{N}_{1,s}^k|^2+|\hat{D}_{0,s}^k|^2
		-\f{\pa_{t}p}{p^{\f32}}Re\left(\hat{N}_{1,s}^k\bar{\hat{D}}_{0,s}^k\right)
		\right]\,d\xi
		\\
		&
		-2\delta_2\mu^{\f13}\sum_{s=0}^{3}\sum_{k\in\Z\backslash\{0\}}\int_{\R}p^{-\f12}Re\left(\hat{N}_{1,s}^k\bar{\hat{D}}_{0,s}^k\right)\,d\xi
		-2\delta_2\mu\sum_{s=1}^{3}\int_{\R}p^{-\f12}Re\left(\hat{N}_{1,s}^0\bar{\hat{D}}_{0,s}^0\right)\,d\xi,
		\\
		\widetilde{\D}_{5}(t)\tri&\sum_{s=0}^{3}\sum_{k\in\Z\backslash\{0\}}\int_{\R}\left(\delta_2\mu^{\f13}\left(1+2\f{k^2}{p^2}\right)
		\right)|\hat{N}_{1,s}^k|^2\,d\xi
		+\delta_2\mu\sum_{s=1}^{3}\int_{\R}|\hat{N}_{1,s}^0|^2\,d\xi
		+\sum_{s=0}^{3}\sum_{k\in\Z}\int_{\R}\nu p|\hat{D}_{0,s}^k|^2\,d\xi.
	\end{align*}
	By a few calculations, we obtain
	\begin{align}\label{con-equi-3}
		\widetilde{\E}_{5}(t)\sim \E_{0,3}^{com}(t),\quad \widetilde{\D}_{5}(t)\sim \D_{0,3}^{com}(t).
	\end{align}
	Summing up \eqref{est-energy-11-1}, \eqref{est-energy-12-1} and \eqref{est-energy-12-2} yields directly that
	\begin{align}
		&\f12\f{d}{dt}\widetilde{\E}_{5}(t)+\widetilde{\D}_{5}(t)
		=\LL_2(t)+\sum_{i=1}^{5}K_i(t).\label{est-energy-13}
	\end{align} 
	Here the linear term $\LL_2$ is defined by
	\begin{align}
		\LL_2(t)\tri&
		\sum_{s=0}^{3}\sum_{k\in\Z\backslash\{0\}}\int_{\R}\left\{
		-\f{k^2\pa_{t}p}{p^{3}}|\hat{N}_{1,s}^k|^2
		+\left(\f32\f{\pa_{t}p}{p}
		+\delta_2\mu^{\f13}-2\mu\f{k^2}{p}\right)|\hat{D}_{0,s}^k|^2
		\right.
		\label{def_L2}\\
		&
		+\left[\delta_2\mu^{\f13}\left(\nu p^{\f12}+2\mu \f{k^2}{p^{\f32}}\right)
		+\f12\nu\f{\pa_{t}p}{p^{\f12}}+\mu \f{k^2\pa_{t}p}{p^{\f52}}
		+\f{k^2}{p^{\f32}}
		-\delta_2\mu^{\f13}\f{\pa_{t}p}{p^{\f32}}
		\right] Re\left(\hat{N}_{1,s}^k\bar{\hat{D}}_{0,s}^k\right)
		\nonumber\\
		&\left.+2\f{k^2}{p^{\f32}}\left(\f{\pa_{t}p}{p}+\delta_2\mu^{\f13}\right)Re\left(\hat{N}_{1,s}^k\bar{\hat{W}}_{0,s}^k\right)
		-2\f{k^2}{p}Re\left(\hat{D}_{0,s}^k\bar{\hat{W}}_{0,s}^k\right)\right\}\,d\xi
		\nonumber\\
		&
		+\delta_2\mu\sum_{s=1}^{3}\int_{\R}\left\{|\hat{D}_{0,s}^0|^2+\nu |\xi|Re\left(\hat{N}_{1,s}^0\bar{\hat{D}}_{0,s}^0\right)\right\}\,d\xi,
		\nonumber
	\end{align}
	and the nonlinear terms $K_i (i=1,\cdots,5)$ are defined by
	\begin{align}
		K_1(t)\tri& \sum_{s=0}^{3}\sum_{k\in\Z}\int_{\R}Re\left(\hat{\FFF}_{1,1,s}^k\bar{\hat{N}}_{1,s}^k\right)\,d\xi,
		\nonumber\\
		K_2(t)\tri&\sum_{s=0}^{3}\sum_{k\in\Z}\int_{\R}Re\left(\hat{\FFF}_{2,1,s}^k\bar{\hat{D}}_{0,s}^k\right)\,d\xi,
		\nonumber\\
		K_3(t)\tri& -\delta_2\mu\sum_{s=1}^{3}\int_{\R}\f{1}{p^{\f12}}\left[Re\left(\hat{\FFF}_{1,1,s}^0\bar{\hat{D}}_{0,s}^0\right)+Re\left(\hat{\FFF}_{2,1,s}^0\bar{\hat{N}}_{1,s}^0\right)\right]\,d\xi,
		\nonumber\\
		K_4(t)\tri&-\sum_{s=0}^{3}\sum_{k\in\Z\backslash\{0\}}\int_{\R}\f{\pa_{t}p}{p^{\f32}}\left[Re\left(\hat{\FFF}_{1,1,s}^k\bar{\hat{D}}_{0,s}^k\right)+Re\left(\hat{\FFF}_{2,1,s}^k\bar{\hat{N}}_{1,s}^k\right)\right]\,d\xi,
		\nonumber\\
		K_5(t)\tri& -\delta_2\mu^{\f13}\sum_{s=0}^{3}\sum_{k\in\Z\backslash\{0\}}\int_{\R}\f{1}{p^{\f12}}\left[Re\left(\hat{\FFF}_{1,1,s}^k\bar{\hat{D}}_{0,s}^k\right)+Re\left(\hat{\FFF}_{2,1,s}^k\bar{\hat{N}}_{1,s}^k\right)\right]\,d\xi.
		\nonumber
	\end{align}
	We now estimate each term on the right-hand side of \eqref{est-energy-13}; the following lemma addresses the linear term $\LL_2$.
	\begin{lemm}\label{lem2-1}
		It holds that
		\begin{align}\label{est-L2}
			\LL_2(t)\leq \f34\widetilde{\D}_5(t)+C\mu^{-1}\D_{0,3}^{com,l}(t)+\f{C}{A}\D_{0,3}^{in}(t).
		\end{align}
	\end{lemm}
	\begin{proof}
		We first estimate all the linear terms for $k\not=0$ on the right-hand side of \eqref{est-energy-13}.
		In view of \eqref{est-varphi-1} and \eqref{est-p-varphi}, we obtain
		\begin{align}\label{est-L2-1}
			\left|\f32\f{\pa_{t}p}{p}|\hat{D}_{0,s}^k|^2\right|
			\lesssim\f{|k|}{p^{\f12}}|\hat{D}_{0,s}^k|^2
			\lesssim |p^{\f12}\hat{\UU}_{0,s}^k|^2.
		\end{align}
		The lift-up term in $D$ is controlled by the dissipation of the lower-order velocity derivatives. 
        It is easy to verify that
        \begin{align*}
            \Big|\big(\delta_2\mu^{\f13}-2\mu\f{k^2}{p}\big)\Big||\hat{D}_{0,s}^k|^2
            \lesssim
            |p^{\f12}\hat{\UU}_{0,s}^k|^2.
        \end{align*}
        Exploiting \eqref{est-p-varphi}, we obtain
		\begin{align}\label{est-L2-2}
			\left|\f{k^2\pa_{t}p}{p^{3}}|\hat{N}_{1,s}^k|^2\right|
			\lesssim \f{|k|^3}{p^{\f52}}|\hat{N}_{1,s}^k|^2
            \lesssim \f{|k|^2}{p}|\hat{N}_{0,s}^k|^2.
		\end{align}
		It is not difficult to check that
		\begin{align}
			&\left|\delta_2\mu^{\f13}\nu p^{\f12}Re\left(\hat{N}_{1,s}^k\bar{\hat{D}}_{0,s}^k\right)\right|
			\lesssim  \delta_2\mu^{\f43}\left( |\hat{N}_{1,s}^k|^2+p |\hat{D}_{0,s}^k|^2 \right),
			\label{est-L2-3}\\
			&\left|\f12\nu \f{\pa_{t}p}{p^{\f12}} Re\left(\hat{N}_{1,s}^k\bar{\hat{D}}_{0,s}^k\right) \right|
			\lesssim\nu |k| \left|Re\left(\hat{N}_{1,s}^k\bar{\hat{D}}_{0,s}^k\right) \right|
			\leq \widetilde{\ep}\nu p |\hat{D}_{0,s}^k|^2+ C_{\widetilde{\ep}}\mu |\hat{N}_{1,s}^k|^2.
			\label{est-L2-4}
		\end{align}
		In light of \eqref{est-varphi-1} and \eqref{est-p-varphi}, we achieve that
		\begin{align}
			&\quad \left|
            \left(2\delta_2\mu^{\f43}\f{k^2}{p^{\f32}}+\mu \f{k^2\pa_{t}p}{p^{\f52}}
			-\delta_2\mu^{\f13}\f{\pa_{t}p}{p^{\f32}}
			\right)
			Re\left(\hat{N}_{1,s}^k\bar{\hat{D}}_{0,s}^k\right)\right|
            \label{est-L2-6}
            \\
            &\lesssim \mu^{\f13}\f{|k|}{p}|\hat{N}_{1,s}^k||\hat{D}_{0,s}^k|
            \leq \widetilde{\ep} |p^{\f12}\hat{\UU}_{0,s}^k|^2
            +C_{\widetilde{\ep}} \mu^{\f13}
            |\hat{N}_{0,s}^k|^2,
			\nonumber
        \end{align}
        where we have used the fact that
        \begin{align*}
            \f{|k|^{\f12}}{p^{\f34}}|\hat{N}_{1,s}^k|
            \lesssim |\hat{N}_{0,s}^k|.
        \end{align*}
        Analogously, we have
        \begin{align}
            &\left|\f{k^2}{p^{\f32}}
            Re\left(\hat{N}_{1,s}^k\bar{\hat{D}}_{0,s}^k\right)\right|\lesssim |\hat{N}_{0,s}^k|^2+ |p^{\f12}\hat{\UU}_{0,s}^k|^2,
            \label{est-L2-6-1}\\
			&\left|2\f{k^2}{p^{\f32}}\left(\f{\pa_{t}p}{p}+\delta_2\mu^{\f13}\right)Re\left(\hat{N}_{1,s}^k\bar{\hat{W}}_{0,s}^k\right)\right|
			\lesssim \f{k^2}{p}  |\hat{W}_{0,s}^k|^2
            +|\hat{N}_{0,s}^k|^2,
			\label{est-L2-7}\\
			&\left|2\f{k^2}{p}Re\left(\hat{D}_{0,s}^k\bar{\hat{W}}_{0,s}^k\right)\right|
			\lesssim  \f{k^2}{p}|\hat{W}_{0,s}^k|^2+|p^{\f12}\hat{\UU}_{0,s}^k|^2.
			\label{est-L2-8}
		\end{align}
		For $k=0$, a direct computation yields immediately that
		\begin{align}
			&\left|\delta_1\mu\nu |\xi|Re\left(\hat{N}_{1,s}^0\bar{\hat{D}}_{0,s}^0\right)\right|
			\lesssim \delta_1\mu^2\left(|\hat{N}_{1,s}^0|^2+|\xi|^2|\hat{D}_{0,s}^0|^2\right).
			\label{est-L2-9}
		\end{align}
		Substituting \eqref{est-L2-1}–\eqref{est-L2-9} into \eqref{def_L2}, invoking \eqref{est-varphi-2}, and then fixing $A$ (defined in \eqref{def-m2}) sufficiently large while taking $\delta_2$ and $\widetilde{\ep}$  sufficiently small, we immediately obtain \eqref{est-L2}.		
	\end{proof}
	
	The nonlinear terms $K_i (i=1,\cdots,5)$ will be dealt with in the following lemma.
	\begin{lemm}\label{lem2-2}
		Under the assumptions of Proposition \ref{prop-4}, we have
		\begin{align}\label{est-Ki}
			\mu \sum_{i=1}^{5}\int_0^t K_i(\tau)\,d\tau
			\lesssim
			\mu^{-1} |\E(t)|^{\frac12}\int_0^t  \D(\tau) \, d\tau. 
		\end{align}
	\end{lemm} 
	
	\begin{proof}
		For the first term $K_1$, proceeding as in the estimate of  $I_1$ in Lemma~\ref{lem1-1}, we obtain
		\begin{align}\label{est-K1}
			\mu \int_0^t K_1(\tau)\,d\tau
			\lesssim \mu^{-1} |\E(t)|^{\frac12}\int_0^t  \D(\tau) \, d\tau. 
		\end{align}
		Recalling \eqref{def_wF2}, we get
		\begin{align}
			K_2
			\lesssim & \sum_{s=0}^{3}
			\left|\int_{\T\times\R}\RR_2\nabla^s\left(\V\cdot\widetilde{\nabla} D\right) D_{0,s} \,dx_1dy_1\right|
			+\sum_{s=0}^{3}
			\left|\int_{\T\times\R}\RR_2\nabla^s\widetilde{F}_{21} D_{0,s} \,dx_1dy_1\right|
			\tri \sum_{i=1}^2K_{2i},\label{est-K2-1}
		\end{align}
		where $\widetilde{F}_{21}$ is defined in \eqref{def_wF21}. Following a similar approach as in the estimate of $I_1$ in Lemma~\ref{lem1-1}, we deduce that	
		\begin{align}\label{est-K21}
			\mu \int_0^t K_{21}(\tau)\,d\tau
			\lesssim \mu^{-1} |\E(t)|^{\frac12}\int_0^t  \D(\tau) \, d\tau. 
		\end{align}
		The estimate proceeds in essentially the same manner as that for $J_5$ in the proof of Proposition~\ref{prop-3-2}, yielding
		\begin{align}
			\mu \int_0^t K_{22}(\tau)\,d\tau \lesssim \mu^{-1} |\E(t)|^{\frac12}\int_0^t  \D(\tau) \, d\tau.
			\label{est-K22} 
		\end{align}
		The combination of estimates \eqref{est-K2-1}–\eqref{est-K22} immediately yields
		\begin{align}
			\mu \int_0^t K_{2}(\tau)\,d\tau
			\lesssim \mu^{-1} |\E(t)|^{\frac12}\int_0^t  \D(\tau) \, d\tau
			\label{est-K2-2} 
		\end{align}
		For $K_3$, arguing as in the proof of \eqref{est-I3-2} in Lemma~\ref{lem1-3}, we deduce that
		\begin{align}
			\mu \int_0^t K_{3}(\tau)\,d\tau
			\lesssim \mu^{-1} |\E(t)|^{\frac12}\int_0^t  \D(\tau) \, d\tau. 				\label{est-K3} 
		\end{align}
		For $K_4$, we divide it into the following three terms:
		\begin{align}
			K_4
			\lesssim&\sum_{s=0}^{3}\left|\int_{\T\times\R}
			\left(\PP_{\not=}\RR_3\nabla^{s}\left(\V\cdot \widetilde{\nabla} N\right)\PP_{\not=}D_{0,s}
			+\PP_{\not=}\nabla^s\left(\V\cdot \widetilde{\nabla}D
			\right)\PP_{\not=}\widetilde{\Delta}^{-\f12}\RR_3N_{1,s}\right)\,dxdy\right|
			\label{est-K4-1}\\
			&+\sum_{s=0}^{3}\left|\int_{\T\times\R}
			\PP_{\not=}\RR_3\nabla^{s}\left(ND\right)\PP_{\not=}D_{0,s}
			\,dxdy\right|
			+\sum_{s=0}^{3}\left|\int_{\T\times\R}
			\left(\PP_{\not=}\nabla^s\widetilde{F}_{21}
			\PP_{\not=}\widetilde{\Delta}^{-\f12}\RR_3 N_{1,s}\right)\,dxdy\right|,
			\nonumber 	 
		\end{align}
		where $\widetilde{F}_{21}$ is defined in \eqref{def_wF21}.
		Proceeding exactly as in the proofs of \eqref{est-I51-2} and \eqref{est-I52} in Lemma~\ref{lem1-4} and of \eqref{est-J5-2} in Proposition~\ref{prop-3-2}, we deduce from \eqref{est-K4-1} that
		\begin{align}
			\mu  \int_0^t K_4(\tau)\,d\tau\lesssim
			\mu^{-1} |\E(t)|^{\frac12}\int_0^t  \D(\tau) \, d\tau	.			\label{est-K4-2} 
		\end{align}
		It then follows by the same argument used to establish \eqref{est-K4-2} that
		\begin{align}
			\mu \int_0^t K_5(\tau)\,d\tau\lesssim
			\mu^{-1} |\E(t)|^{\frac12}\int_0^t  \D(\tau) \, d\tau.				\label{est-K5} 
		\end{align}
		Collecting \eqref{est-K1}, \eqref{est-K2-2}, \eqref{est-K3}, \eqref{est-K4-2} and \eqref{est-K5} together, we obtain \eqref{est-Ki}.
	\end{proof}
	
	With Lemmas \ref{lem2-1}-\ref{lem2-2} at hand, we can complete the proof of Proposition \ref{prop-4}.
	
	\begin{proof}[\textbf{Proof of Proposition \ref{prop-4}}]
		Integrating \eqref{est-energy-13} with respect to t and then substituting \eqref{est-L2} and \eqref{est-Ki} into the resulting inequality, we conclude that \eqref{est-3} holds.		
	\end{proof}

	\section{Mixed type energy estimates}
	Because the operator $\widetilde{\nabla}$ differs from $\nabla$ whenever $k\not=0$, cross-terms that couple “good” and “bad” derivatives of the density and velocity are no longer controllable. Consequently, even though the estimates for  $\E_{j,3-j}^{com,l}$, $\E_{j,3-j}^{in}$ and $\E_{j,3-j}^{com}$ with $j=0$ were already obtained in Propositions \ref{prop-2}, \ref{prop-3}, \ref{prop-4}, we must still derive,  within the $H^s$-framework with  $1\le s\le 3$, analogous {\it a priori} energy bounds for solutions possessing $(s-j)$ good derivatives and $j$ bad derivatives, where $j$ ranges from 1 to $s$.

	\subsection{Cross terms of compressible part}
	In this subsection, we establish the following proposition, which provides the {\it a priori} energy estimate for the cross-terms that couple the “good’’ and “bad’’ lower–order derivatives of the compressible part of the solution.
	
	\begin{prop}\label{prop-5}
		Under the assumptions of Theorem \ref{theo2} and \eqref{ass:density}, we have
		\begin{align}
			&\sum_{j=1}^{3}c_j\mu^{\f{2j}{3}}\E_{j,3-j}^{com,l}(t)+\sum_{j=1}^{3}c_j\mu^{\f{2j}{3}}\int_0^t \D_{j,3-j}^{com,l}(\tau)\,d\tau 
			\label{est-4}\\
			\lesssim&~
			\sum_{j=1}^{3}c_j\mu^{\f{2j}{3}}\E_{j,3-j}^{com,l}(0)+
			\sum_{j=1}^3c_j\mu^{\f{2(j-1)}{3}}\int_0^t\D_{j-1,3-(j-1)}^{com,l}(\tau)\,d\tau 
			\nonumber\\
			&		+\f{1}{A}\sum_{j=1}^3c_j\mu^{\f{2j}{3}}\int_0^t\D_{j,3-j}^{in}(\tau)\,d\tau	+
			\mu^{-1} |\E(t)|^{\frac12}\int_0^t  \D(\tau) \, d\tau	,
			\nonumber
		\end{align}
		where $c_j ~(j=1,2,3)$ are some positive constant that will be determined later.
	\end{prop}
	We derive from \eqref{eq_n-d-w-2} that
	\begin{align}
		\pa_{t}\hat{N}_{j,s-j}^k=&-\left(\f{\pa_{t}m_1}{m_1}+\f{\pa_{t}m_2}{m_2}+\f14\f{\pa_{t}\varphi}{\varphi}\right)\hat{N}_{j,s-j}^k+\f{j}{2} \f{\pa_{t}p}{p}\hat{N}_{j,s-j}^k-p^{\f12}\hat{\UU}_{j,s-j}^k+\hat{\FFF}_{1,j,s-j}^k,
		\label{eq_N-j}\\
		\pa_{t}\hat{\UU}_{j,s-j}^k=&-\left(\f{\pa_{t}m_1}{m_1}+\f{\pa_{t}m_2}{m_2}+\f14\f{\pa_{t}\varphi}{\varphi}+\nu p\right)\hat{\UU}_{j,s-j}^k+\f{j+1}{2} \f{\pa_{t}p}{p}\hat{\UU}_{j,s-j}^k
		\label{eq_D-j}\\
		&+\left(p^{\f12}+2\f{k^2}{p^{\f32}}\right)\hat{N}_{j,s-j}^k-2\mu \f{k^2}{p}\hat{\UU}_{j,s-j}^k
		-2\f{k^2}{p^{\f32}}\varphi^{-\f14}\hat{W}_{j,s-j}^k
		+\hat{\FFF}_{2,j,s-j}^k,\nonumber
	\end{align}
	where  nonlinear terms are be determined explicitly as
	\begin{align}
		\hat{\FFF}_{1,j,s-j}^k\tri&m_1^{-1}m_2^{-1}\varphi^{-\f14}<k,\xi>^{s-j}p^{\f{j}{2}}\hat{\FFF}_1^k,
		\\
		\hat{\FFF}_{2,j,s-j}^k\tri&m_1^{-1}m_2^{-1}\varphi^{-\f14}<k,\xi>^{s-j}p^{\f{j-1}{2}}\hat{\FFF}_2^k.
	\end{align}
	It follows from  \eqref{eq_N-j} and \eqref{eq_D-j} that
	\begin{align}
		&\f12\f{d}{dt}\sum_{s=1}^{3}\sum_{j=1}^{s}c_j\mu^{\f{2j}{3}}\sum_{k\in\Z}\int_{\R}|(\hat{N}_{j,s-j}^k,\hat{\UU}_{j,s-j}^k)|^2\,d\xi
		\label{est-energy-16}\\
		&+\sum_{s=1}^{3}\sum_{j=1}^{s}c_j\mu^{\f{2j}{3}}\sum_{k\in\Z}\int_{\R}\left(\f{\pa_{t}m_1}{m_1}+\f{\pa_{t}m_2}{m_2}+\f14\f{\pa_{t}\varphi}{\varphi}\right)|\hat{N}_{j,s-j}^k|^2\,d\xi
		\nonumber\\
		&+\sum_{s=1}^{3}\sum_{j=1}^{s}c_j\mu^{\f{2j}{3}}\sum_{k\in\Z}\int_{\R}\left(\f{\pa_{t}m_1}{m_1}+\f{\pa_{t}m_2}{m_2}+\f14\f{\pa_{t}\varphi}{\varphi}+\nu p\right)|\hat{\UU}_{j,s-j}^k|^2\,d\xi
		\nonumber\\
		=&\sum_{s=1}^{3}\sum_{j=1}^{s}c_j\mu^{\f{2j}{3}}\sum_{k\in\Z\backslash\{0\}}\int_{\R}\left\{\f{j}{2}\f{\pa_{t}p}{p}|\hat{N}_{j,s-j}^k|^2+\f{j+1}{2}\f{\pa_{t}p}{p}|\hat{\UU}_{j,s-j}^k|^2+2\f{k^2}{p^{\f32}}Re\left(\hat{N}_{j,s-j}^k\bar{\hat{\UU}}_{j,s-j}^k\right)
		\right.
		\nonumber\\
		&\left. \qquad-2\mu \f{k^2}{p}|\hat{\UU}_{j,s-j}^k|^2
		-2\f{k^2}{p^{\f32}}\varphi^{-\f14}Re\left(\hat{W}_{j,s-j}^k\bar{\hat{\UU}}_{j,s-j}^k\right)\right\}\,d\xi
		\nonumber\\
		&\quad+\sum_{s=1}^{3}\sum_{j=1}^{s}c_j\mu^{\f{2j}{3}}\sum_{k\in\Z}\int_{\R}\left\{Re\left(\hat{\FFF}_{1,j,s-j}^k\bar{\hat{N}}_{j,s-j}^k\right)+Re\left(\hat{\FFF}_{2,j,s-j}^k\bar{\hat{\UU}}_{j,s-j}^k\right)\right\}\,d\xi,
		\nonumber
	\end{align}
	where $c_j ~(j=1,2,3)$ are some positive constant that will be determined later.
	We now establish the dissipation estimate for the density:
	\begin{align}
		&-\f{d}{dt}\left[p^{-\f12}Re\left(\hat{N}_{j,s-j}^k\bar{\hat{\UU}}_{j,s-j}^k\right)\right]
		+\left(1+2\f{k^2}{p^2}\right)|\hat{N}_{j,s-j}^k|^2
		\label{est-energy-17}\\
		=&~|\hat{\UU}_{j,s-j}^k|^2+\f{1}{p^{\f12}}\left(2\f{\pa_{t}m_1}{m_1}+2\f{\pa_{t}m_2}{m_2}+\f12\f{\pa_{t}\varphi}{\varphi}-j\f{\pa_{t}p}{p}+2\mu \f{k^2}{p}+\nu p\right)Re\left(\hat{N}_{j,s-j}^k\bar{\hat{\UU}}_{j,s-j}^k\right)
		\nonumber\\
		&+2\f{k^2}{p^2}\varphi^{-\f14}Re\left(\hat{N}_{j,s-j}^k\bar{\hat{W}}_{j,s-j}^k\right)
		-\f{1}{p^{\f12}}\left[Re\left(\hat{\FFF}_{1,j,s-j}^k\bar{\hat{\UU}}_{j,s-j}^k\right)+Re\left(\hat{\FFF}_{2,j,s-j}^k\bar{\hat{N}}_{j,s-j}^k\right)\right].\nonumber
	\end{align}
	We observe that the dissipation estimates for the density differ between the zero and the non-zero frequencies.
	For $k\not=0$, we multiply \eqref{est-energy-17} by $\delta_5\mu^{\f13}$, then integrate over $\xi$ in $\R$, and sum over $k$ in $\Z\backslash\{0\}$ to obtain
	\begin{align}
		&-\delta_5\mu^{\f13}\f{d}{dt}\sum_{s=1}^{3}\sum_{j=1}^{s}c_j\mu^{\f{2j}{3}}\sum_{k\in\Z\backslash\{0\}}\int_{\R}p^{-\f12}Re\left(\hat{N}_{j,s-j}^k\bar{\hat{\UU}}_{j,s-j}^k\right)\,d\xi
		\label{est-energy-17-1}\\
		&+\delta_5\mu^{\f13}
		\sum_{s=1}^{3}\sum_{j=1}^{s}c_j\mu^{\f{2j}{3}}\sum_{k\in\Z\backslash\{0\}}\int_{\R}\left(1+2\f{k^2}{p^2}\right)|\hat{N}_{j,s-j}^k|^2\,d\xi
		\nonumber\\
		=&~\delta_5\mu^{\f13}\sum_{s=1}^{3}\sum_{j=1}^{s}c_j\mu^{\f{2j}{3}}\sum_{k\in\Z\backslash\{0\}}\int_{\R}\left\{|\hat{\UU}_{j,s-j}^k|^2+\f{1}{p^{\f12}}\left(2\f{\pa_{t}m_1}{m_1}+2\f{\pa_{t}m_2}{m_2}+\f12\f{\pa_{t}\varphi}{\varphi}
		\right.\right.
		\nonumber\\
		&\left.\left.
		-j\f{\pa_{t}p}{p}+2\mu \f{k^2}{p}+\nu p\right)Re\left(\hat{N}_{j,s-j}^k\bar{\hat{\UU}}_{j,s-j}^k\right)
		+2\f{k^2}{p^2}\varphi^{-\f14}Re\left(\hat{N}_{j,s-j}^k\bar{\hat{W}}_{j,s-j}^k\right)
		\right.
		\nonumber\\
		&\left.-\f{1}{p^{\f12}}\left[Re\left(\hat{\FFF}_{1,j,s-j}^k\bar{\hat{\UU}}_{j,s-j}^k\right)+Re\left(\hat{\FFF}_{2,j,s-j}^k\bar{\hat{N}}_{j,s-j}^k\right)\right]\right\}\,d\xi.\nonumber
	\end{align}
	While for $k=0$, we multiply \eqref{est-energy-17} by $\delta_5\mu$, and integrate over $\xi$ in $\R$, to obtain 
	\begin{align}
		&-\delta_5\mu\f{d}{dt}\sum_{s=2}^{3}\sum_{j=1}^{s-1}c_j\mu^{\f{2j}{3}}\int_{\R}
		p^{-\f12}Re\left(\hat{N}_{j,s-j}^0\bar{\hat{\UU}}_{j,s-j}^0\right)\,d\xi
		+\delta_5\mu
		\sum_{s=2}^{3}\sum_{j=1}^{s-1}c_j\mu^{\f{2j}{3}}\int_{\R}|\hat{N}_{j,s-j}^0|^2\,d\xi
		\label{est-energy-17-2}\\
		=&~\delta_5\mu\sum_{s=2}^{3}\sum_{j=1}^{s-1}c_j\mu^{\f{2j}{3}}\int_{\R}\Big\{|\hat{\UU}_{j,s-j}^0|^2+\nu |\xi| Re\left(\hat{N}_{j,s-j}^0\bar{\hat{\UU}}_{j,s-j}^0\right)  \nonumber\\
		&   \quad-|\xi|^{-1}\left[Re\left(\hat{\FFF}_{1,j,s-j}^0\bar{\hat{\UU}}_{j,s-j}^0\right)+Re\left(\hat{\FFF}_{2,j,s-j}^0\bar{\hat{N}}_{j,s-j}^0\right)\right]\Big\}\,d\xi.
		\nonumber
	\end{align}
	We define
	\begin{align*}
		\widetilde{\E}_6(t)\tri&
		\sum_{s=1}^{3}\sum_{j=1}^{s}c_j\mu^{\f{2j}{3}}\sum_{k\in\Z}\int_{\R}
		\left[|\hat{N}_{j,s-j}^k|^2+|\hat{\UU}_{j,s-j}^k|^2\right]\,d\xi
		\\		
		&-2\delta_5\mu^{\f13}\sum_{s=1}^{3}\sum_{j=1}^{s}c_j\mu^{\f{2j}{3}}\sum_{k\in\Z\backslash\{0\}}\int_{\R} p^{-\f12}Re\left(\hat{N}_{j,s-j}^k\bar{\hat{\UU}}_{j,s-j}^k\right)\,d\xi
		\\
		&-2\delta_5\mu\sum_{s=2}^{3}\sum_{j=1}^{s-1}c_j\mu^{\f{2j}{3}}\int_{\R}p^{-\f12}Re\left(\hat{N}_{j,s-j}^0\bar{\hat{\UU}}_{j,s-j}^0\right)\,d\xi,
		\nonumber\\
		\widetilde{\D}_6(t)\tri&
		\sum_{s=1}^{3}\sum_{j=1}^{s}c_j\mu^{\f{2j}{3}}\sum_{k\in\Z\backslash\{0\}}\int_{\R}
		\left(\f{\pa_{t}m_1}{m_1}+\f{\pa_{t}m_2}{m_2}+\f14\f{\pa_{t}\varphi}{\varphi}
		\right.
		\\
		&\left.
		+\delta_5\mu^{\f13}\left(1+2\f{k^2}{p^2}\right)
		\right)|\hat{N}_{j,s-j}^k|^2\,d\xi
		+\delta_5\mu\sum_{s=2}^{3}\sum_{j=1}^{s-1}c_j\mu^{\f{2j}{3}}\int_{\R}|\hat{N}_{j,s-j}^0|^2\,d\xi
		\\
		&+\sum_{s=1}^{3}\sum_{j=1}^{s}c_j\mu^{\f{2j}{3}}\sum_{k\in\Z}\int_{\R}
		\left(\f{\pa_{t}m_1}{m_1}+\f{\pa_{t}m_2}{m_2}+\f14\f{\pa_{t}\varphi}{\varphi}+\nu p\right)|\hat{\UU}_{j,s-j}^k|^2\,d\xi.
		\nonumber
	\end{align*}
	By some direct calculations, we obtain 
	\begin{align}\label{con-equi-4}
		\widetilde{\E}_6(t)\sim \sum_{j=1}^{3}c_j\mu^{\f{2j}{3}}\E_{j,3-j}^{com}(t),\quad \widetilde{\D}_6(t)\sim \sum_{j=1}^{3}c_j\mu^{\f{2j}{3}}\D_{j,3-j}^{com}(t).
	\end{align}
	Combining \eqref{est-energy-16}, \eqref{est-energy-17-1} and \eqref{est-energy-17-2}, we immediately obtain
	\begin{align}
		&\f12\f{d}{dt}\widetilde{\E}_6(t)+\widetilde{\D}_6(t)
		=\LL_3(t)+\sum_{i=1}^3L_i.
		\label{est-energy-18}
	\end{align} 
	Here the linear term $\LL_3$ and the nonlinear terms $L_i(i=1,2,3)$ are defined by
	\begin{align}
		\LL_3(t)\tri &\sum_{s=1}^{3}\sum_{j=1}^{s}c_j\mu^{\f{2j}{3}}\sum_{k\in\Z\backslash\{0\}}\int_{\R}
		\left\{\f{j}{2}\f{\pa_{t}p}{p}|\hat{N}_{j,s-j}^k|^2
		+\left(\f{j+1}{2}\f{\pa_{t}p}{p}
		+\delta_5\mu^{\f13}-2\mu\f{k^2}{p}\right)|\hat{\UU}_{j,s-j}^k|^2
		\right.
		\label{def_L3}\\
		&
		+\left[\delta_5\mu^{\f13}\left(\nu p^{\f12}+2\mu\f{k^2}{p^{\f32}}\right)
		+2\f{k^2}{p^{\f32}}
		+\delta_5\mu^{\f13}\f{1}{p^{\f12}}\left(2\f{\pa_{t}m_1}{m_1}+2\f{\pa_{t}m_2}{m_2}+\f12\f{\pa_{t}\varphi}{\varphi}-j\f{\pa_{t}p}{p}\right)
		\right]
		\nonumber\\
		&\left.\times Re\left(\hat{N}_{j,s-j}^k\bar{\hat{\UU}}_{j,s-j}^k\right)
		+2\delta_5\mu^{\f13}\f{k^2}{p^2}\varphi^{-\f14}Re\left(\hat{N}_{j,s-j}^k\bar{\hat{W}}_{j,s-j}^k\right)
		-2\f{k^2}{p^{\f32}}\varphi^{-\f14}Re\left(\hat{W}_{j,s-j}^k\bar{\hat{\UU}}_{j,s-j}^k\right)\right\}\,d\xi
		\nonumber\\
		&+\delta_5\mu\sum_{s=2}^{3}\sum_{j=2}^{s}c_j\mu^{\f{2j}{3}}\int_{\R}\left\{|\hat{\UU}_{j,s-j}^0|^2+\nu |\xi|Re\left(\hat{N}_{j,s-j}^0\bar{\hat{\UU}}_{j,s-j}^0\right)
		\right\}\,d\xi,
		\nonumber\\
		L_1\tri&\sum_{s=1}^{3}\sum_{j=1}^{s}c_j\mu^{\f{2j}{3}}\sum_{k\in\Z}\int_{\R}
		\left\{
		Re\left(\hat{\FFF}_{1,j,s-j}^k\bar{\hat{N}}_{j,s-j}^k\right)+Re\left(\hat{\FFF}_{2,j,s-j}^k\bar{\hat{\UU}}_{j,s-j}^k\right)\right\}\,d\xi,
		\nonumber\\
		L_2\tri&-\delta_5\mu\sum_{s=2}^{3}\sum_{j=1}^{s-1}c_j\mu^{\f{2j}{3}}\int_{\R}|\xi|^{-1}\left[Re\left(\hat{\FFF}_{1,j,s-j}^0\bar{\hat{\UU}}_{j,s-j}^0\right)
		+Re\left(\hat{\FFF}_{2,j,s-j}^0\bar{\hat{N}}_{j,s-j}^0\right)\right]\,d\xi,
		\nonumber\\
		L_3\tri&-\delta_5\mu^{\f13}\sum_{s=1}^{3}\sum_{j=1}^{s}c_j\mu^{\f{2j}{3}}\sum_{k\in\Z\backslash\{0\}}\int_{\R}\f{1}{p^{\f12}}\left[Re\left(\hat{\FFF}_{1,j,s-j}^k\bar{\hat{\UU}}_{j,s-j}^k\right)+Re\left(\hat{\FFF}_{2,j,s-j}^k\bar{\hat{N}}_{j,s-j}^k\right)\right]\,d\xi.
		\nonumber
	\end{align}
	We now begin to deal with all the terms on the right-hand side of \eqref{est-energy-18}.
	Firstly, we give the following lemma to deal with the linear term $\LL_3$.
	\begin{lemm}\label{lem-L3}
		It holds that
		\begin{align}\label{est-L3}
			\LL_3(t) \leq \f12\widetilde{\D}_6(t)
			+C\sum_{j=1}^{3}c_j\mu^{\f{2(j-1)}{3}}\D_{j-1,3-(j-1)}^{com,l}(t) +\f{C}{A}\sum_{j=1}^{3}\D_{j,3-j}^{in}(t).
		\end{align}
	\end{lemm}
	\begin{proof}
		We begin by estimating every linear term in \eqref{def_L3} for $k\not=0$.
		The lift-up term containing the $j$-order bad derivative in the right-hand side of \eqref{def_L3} can be controlled by the enhanced dissipation term involving the $j$-order and $(j-1)$-order bad derivative.  Indeed, applying \eqref{est-p-varphi} together with Young’s inequality yields
		\begin{align}
			&\left|\f{j}{2}\f{\pa_{t}p}{p}|\hat{N}_{j,s-j}^k|^2
			+\f{j+1}{2}\f{\pa_{t}p}{p}|\hat{\UU}_{j,s-j}^k|^2\right|
			\label{est-L3-1}\\
			\lesssim&~ \f{|k|}{p^{\f12}} \left(|\hat{N}_{j,s-j}^k|^2+ |\hat{\UU}_{j,s-j}^k|^2\right)
			\nonumber\\
			\lesssim&~ \widetilde{\ep} \mu^{\f13}\left(|\hat{N}_{j,s-j}^k|^2+ |\hat{\UU}_{j,s-j}^k|^2\right) + C_{\widetilde{\ep}} \mu^{-\f13} \left(|\hat{N}_{j-1,s-(j-1)}^k|^2+ |\hat{\UU}_{j-1,s-(j-1)}^k|^2\right).
			\nonumber
		\end{align}
		The remaining terms in \eqref{def_L3} can be estimated in the same way as in the proof of \eqref{est-LL1}.
		By Young's inequality, we get 
		\begin{align}
			&\left|2\delta_5\mu^{\f13}\f{1}{p^{\f12}}\left(\f{\pa_{t}m_1}{m_1}+\f{\pa_{t}m_2}{m_2}+\f14\f{\pa_{t}\varphi}{\varphi}\right)
			Re\left(\hat{N}_{j,s-j}^k\bar{\hat{\UU}}_{j,s-j}^k\right)\right|
			\label{est-L3-2}\\
			\lesssim&~
			\delta_5\left(\f{\pa_{t}m_1}{m_1}+\f{\pa_{t}m_2}{m_2}+\f14\f{\pa_{t}\varphi}{\varphi}\right)
			\left(|\hat{N}_{j,s-j}^k|^2+ |\hat{\UU}_{j,s-j}^k|^2\right).
			\nonumber
		\end{align}
		After some direct calculations, we infer that
		\begin{align}
			&\left|\delta_5\mu^{\f13}\nu p^{\f12}Re\left(\hat{N}_{j,s-j}^k\bar{\hat{\UU}}_{j,s-j}^k\right)\right|
			\lesssim
			\delta_5\mu^{\f43}
			\left(|\hat{N}_{j,s-j}^k|^2+ p|\hat{\UU}_{j,s-j}^k|^2\right),
			\label{est-L3-3}\\
			&\left|2\delta_5\mu^{\f13}\f{k^2}{p^2}\varphi^{-\f14}Re\left(\hat{N}_{j,s-j}^k\bar{\hat{W}}_{j,s-j}^k\right)
			-2\f{k^2}{p^{\f32}}\varphi^{-\f14}Re\left(\hat{W}_{j,s-j}^k\bar{\hat{\UU}}_{j,s-j}^k\right)\right|
			\label{est-L3-4}\\
			\lesssim& ~
			\f{k^2}{p}\left(|\hat{N}_{j,s-j}^k|^2+ |\hat{\UU}_{j,s-j}^k|^2+|\hat{W}_{j,s-j}^k|^2\right).
			\nonumber
		\end{align}
		By \eqref{est-p-varphi} again, we find that
		\begin{align}
			&\left|\left[2\f{k^2}{p^{\f32}}+\delta_5\mu^{\f13}\f{1}{p^{\f12}}\left(2\mu\f{k^2}{p}-j\f{\pa_{t}p}{p}\right)
			\right]
			Re\left(\hat{N}_{j,s-j}^k\bar{\hat{\UU}}_{j,s-j}^k\right)\right|
			\lesssim
			\f{k^2}{p}\left(|\hat{N}_{j,s-j}^k|^2+ |\hat{\UU}_{j,s-j}^k|^2\right).
			\label{est-L3-5}
		\end{align}
		While for $k=0$, the linear terms on the right-hand side of \eqref{est-energy-18} can be controlled by
		\begin{align}
			\left|\delta_5\mu\nu |\xi|Re\left(\hat{N}_{j,s-j}^0\bar{\hat{\UU}}_{j,s-j}^0\right)\right|
			\lesssim \delta_5\mu^2\left(|\hat{N}_{j,s-j}^0|^2+|\xi|^2|\hat{\UU}_{j,s-j}^0|^2
			\right).
			\label{est-L3-6}
		\end{align}
		Collecting all estimates \eqref{est-L3-1}-\eqref{est-L3-6} together, and then choosing $A$ defined in \eqref{def-m2} suitably large, $\widetilde{\ep}$ and $\delta_5$ suitably small, we can deduce \eqref{est-L3}.
	\end{proof}

	With the help of Lemma \ref{lem-L3}, we can prove Proposition \ref{prop-5}.
	
	\begin{proof}[\textbf{Proof of Proposition \ref{prop-5}}]
		The nonlinear terms $L_i (i=1,2,3)$  can be handled by the same arguments employed in the proof of Proposition~\ref{prop-2}.
		Indeed, because the energy  $\widetilde{\E}_6$  already incorporates at least one derivative of the solution, no separate estimate of the zero-mode in $L^2$ is required. Consequently, the reasoning presented in Proposition~\ref{prop-2} yields			
		\begin{align}
			\sum_{i=1}^{3}\int_0^tL_i(\tau)\,d\tau \lesssim \mu^{-1} |\E(t)|^{\frac12}\int_0^t  \D(\tau) \, d\tau.
			\nonumber
		\end{align}
		Inserting the preceding estimate together with \eqref{est-L3} into \eqref{est-energy-18} and invoking \eqref{est-varphi-2} and \eqref{con-equi-4}, we arrive at \eqref{est-4}.			
	\end{proof}

	\subsection{Cross-terms of  incompressible part}
	Our goal in this subsection is to derive the {\it a priori} energy estimate for the cross-terms that couple the good and bad derivatives of the vorticity.
	
	Recalling $\eqref{eq_n-d-w-2}_3$ once more, we obtain 
	\begin{align}
		\pa_{t}\hat{W}_{j,s-j}^k=&-\left(\f{\pa_{t}m_1}{m_1}+\f{\pa_{t}m_2}{m_2}+\mu p\right)\hat{W}_{j,s-j}^k+\f{j}{2}\f{\pa_{t}p}{p}\hat{W}_{j,s-j}^k-2\mu \f{k^2}{p^{\f32}}\hat{N}_{j+1,s-j}^k
		\label{eq_W-j+1}\\
		&+\mu(\mu+\mu')p\hat{D}_{j,s-j}^k
		-\mu\left(\f{\pa_{t}p}{p}-2\mu \f{k^2}{p}\right)\hat{D}_{j,s-j}^k
		+2\mu \f{k^2}{p}\hat{W}_{j,s-j}^k+\hat{\FFF}_{3,j+1,s-j}^k,\nonumber
	\end{align}
	where nonlinear term is defined by
	\begin{align}
		\hat{\FFF}_{3,j+1,s-j}^k\tri&m_1^{-1}m_2^{-1}<k,\xi>^{s-j}p^{\f{j}{2}}\hat{\FFF}_3^k.
	\end{align}
	
	Here is the main results of this subsection.				
	\begin{prop}\label{prop-6}
		Under the assumptions of Theorem \ref{theo2} and \eqref{ass:density}, we have
		\begin{align}
			&\sum_{j=1}^{3}c_j\mu^{\f{2j}{3}}\E_{j,3-j}^{in}(t)+\sum_{j=1}^{3} c_j\mu^{\f{2j}{3}}\int_0^t\D_{j,3-j}^{in}(\tau)\,d\tau 
			\label{est-5}\\
			\lesssim&~
			\sum_{j=1}^{3}c_j\mu^{\f{2j}{3}}\E_{j,3-j}^{in}(0)+
			\sum_{j=1}^3c_j\mu^{\f{2(j-1)}{3}}\int_0^t\D_{j-1,3-(j-1)}^{in}(\tau)\,d\tau
            \nonumber\\
			&\quad +
			\mu^{\f23}\sum_{j=0}^3c_j\mu^{\f{2j}{3}}\int_0^t\big(\D_{j,3-j}^{com,l}+\mu\D_{j,3-j}^{com}\big)(\tau)\,d\tau
            +\mu^{-1} |\E(t)|^{\frac12}\int_0^t  \D(\tau) \, d\tau. 
			\nonumber
		\end{align}
	\end{prop}

	\begin{proof}
		We deduce from \eqref{eq_W-j+1} that
		\begin{align}
			&\f12\f{d}{dt}|\hat{W}_{j,s-j}^k|^2+\left(\f{\pa_{t}m_1}{m_1}+\f{\pa_{t}m_2}{m_2}+\mu p\right)|\hat{W}_{j,s-j}^k|^2
			\label{est-energy-20}\\
			=&~\f{j}{2}\f{\pa_{t}p}{p}|\hat{W}_{j,s-j}^k|^2
			+\mu(\mu+\mu')p Re\left(\hat{D}_{j,s-j}^k\bar{\hat{W}}_{j,s-j}^k\right)
			-2\mu \f{k^2}{p^{\f32}} Re\left(\hat{N}_{j+1,s-j}^k\bar{\hat{W}}_{j,s-j}^k\right)
			\nonumber\\
			&-\mu \left(\f{\pa_{t}p}{p}-2\mu \f{k^2}{p}\right) Re\left(\hat{D}_{j,s-j}^k\bar{\hat{W}}_{j,s-j}^k\right)
			+2\mu \f{k^2}{p}|\hat{W}_{j,s-j}^k|^2
			+Re\left(\hat{\FFF}_{3,j+1,s-j}^k\bar{\hat{W}}_{j,s-j}^k\right).
			\nonumber
		\end{align}
		The lift-up term with $j$-order bad derivative on the right-hand side of \eqref{est-energy-20} can be controlled by the enhanced dissipation term with $j$-order and $(j-1)$-order bad derivative. In fact, we have
		\begin{align}\label{L-W-j+1}
			\left|\f{j}{2}\f{\pa_{t}p}{p}|\hat{W}_{j,s-j}^k|^2\right| \lesssim\f{|k|}{p^{\f12}}|\hat{W}_{j,s-j}^k|^2\lesssim \widetilde{\ep}\mu^{\f13}|\hat{W}_{j,s-j}^k|^2+ C_{\widetilde{\ep}} \mu^{-\f13} |\hat{W}_{j-1,s-(j-1)}^k|^2.
		\end{align}
		Then by the same procedure as in \eqref{est-W-j0-1}-\eqref{est-W-j0-3}, it is easy to verify that the other linear terms on the right-hand side of \eqref{est-energy-20} can be computed as follows
		\begin{align*}
			&\left|\mu(\mu+\mu')p Re\left(\hat{D}_{j,s-j}^k\bar{\hat{W}}_{j,s-j}^k\right)\right|
			\leq \widetilde{\ep} \mu p |\hat{W}_{j,s-j}^k|^2+C_{\widetilde{\ep}}\mu^3 p |\hat{D}_{j,s-j}^k|^2,
			\\
			&\left|2\mu \f{k^2}{p^{\f32}} Re\left(\hat{N}_{j+1,s-j}^k\bar{\hat{W}}_{j,s-j}^k\right)\right|
			\leq \widetilde{\ep} \f{k^2}{p}|\hat{W}_{j,s-j}^k|^2+C_{\widetilde{\ep}} \mu^2|\hat{N}_{j,s-j}^k|^2,
			\\
			&\left|\mu \left(\f{\pa_{t}p}{p}-2\mu \f{k^2}{p}\right) Re\left(\hat{D}_{j,s-j}^k\bar{\hat{W}}_{j,s-j}^k\right)\right|
			\leq \widetilde{\ep} \mu^{\f13} |\hat{W}_{j,s-j}^k|^2+C_{\widetilde{\ep}}\mu^{\f53}  |p^{\f12}\hat{\UU}_{j,s-j}^k|^2.
		\end{align*}
		Plugging all the estimates above and \eqref{L-W-j+1} into \eqref{est-energy-20}, then choosing $A$ defined in \eqref{def-m2} suitably large and $\widetilde{\ep}$ suitably small, we end up with
		\begin{align}
			\f{d}{dt}\widetilde{\E}_7(t)
			+\widetilde{\D}_7(t)
			\lesssim& \sum_{s=1}^{3}\sum_{j=1}^{s}c_j\mu^{\f{2j}{3}}\sum_{k\in\Z}\int_{\R}
			\left\{\mu^{-\f13} |\hat{W}_{j-1,s-(j-1)}^k|^2+\mu^2|\hat{N}_{j,s-j}^k|^2
			\right.
			\label{est-energy-21}\\
			&\left.
			+\mu^3 p |\hat{D}_{j,s-j}^k|^2
			+\mu^{\f53}  |p^{\f12}\hat{\UU}_{j,s-j}^k|^2
			+Re\left(\hat{\FFF}_{3,j+1,s-j}^k\bar{\hat{W}}_{j,s-j}^k\right)\right\}\,d\xi
			\nonumber\\
			\lesssim&\sum_{j=1}^3c_j\mu^{\f{2(j-1)}{3}}\D_{j-1,3-(j-1)}^{in}
			+\mu^{\f53}\sum_{j=1}^3c_j\mu^{\f{2j}{3}}\D_{j,3-j}^{com}
			\nonumber\\
			&+\sum_{s=1}^{3}\sum_{j=1}^{s}c_j\mu^{\f{2j}{3}}\sum_{k\in\Z}\int_{\R}
			Re\left(\hat{\FFF}_{3,j+1,s-j}^k\bar{\hat{W}}_{j,s-j}^k\right)\,d\xi,
			\nonumber
		\end{align}
		where $\widetilde{\E}_7$ and $\widetilde{\D}_7$ are defined by
		\begin{align*}
			\widetilde{\E}_7(t)\tri&\sum_{s=1}^{3}\sum_{j=1}^{s}c_j\mu^{\f{2j}{3}}\sum_{k\in\Z}\int_{\R}|\hat{W}_{j,s-j}^k|^2\,d\xi
			\\
			\widetilde{\D}_7(t)\tri&\sum_{s=1}^{3}\sum_{j=1}^{s}c_j\mu^{\f{2j}{3}}\sum_{k\in\Z}\int_{\R}\left(\f{\pa_{t}m_1}{m_1}+\f{\pa_{t}m_2}{m_2}+\mu p\right)|\hat{W}_{j,s-j}^k|^2\,d\xi.
		\end{align*}
		It is not difficult to check that
		\begin{align}\label{con-equi-5}
			\widetilde{\E}_7(t)\sim \sum_{j=1}^3c_j\mu^{\f{2j}{3}}\E_{j,s-j}^{in}(t),\quad 	\widetilde{\D}_7(t)\sim \sum_{j=1}^3c_j\mu^{\f{2j}{3}}\D_{j,s-j}^{in}(t).
		\end{align}
		The nonlinear term is controlled by an argument parallel to the proof of Proposition \ref{prop-3}.
		Because the energy  $\widetilde{\E}_7$ already contains at least one derivative of the solution,	no separate $L^2$ estimate for the zero mode is required; consequently, the proof is simpler and is therefore omitted. One obtains			
		\begin{align*}
			\sum_{s=1}^3\sum_{j=1}^sc_j\mu^{\f{2j}{3}}\int_0^t \left|\sum_{k\in\mathbb{Z}}\int_{\R}
			Re\left(\hat{\FFF}_{3,j+1,s-j}^k\bar{\hat{W}}_{j,s-j}^k\right)\,d\xi\right|\,d\tau
			\lesssim 
			\mu^{-1} |\E(t)|^{\frac12}\int_0^t  \D(\tau) \, d\tau,
		\end{align*}
		which, together with \eqref{est-energy-21} and \eqref{con-equi-5}, yields \eqref{est-5} immediately.
	\end{proof}

	\subsection{Cross-terms of compressible part with higher-order derivatives}
	In this section we establish the {\it a-priori} energy estimate for the cross terms that couple the good and bad higher-order derivatives arising from the compressible component of the solution.  A precise statement is given in the following proposition.

	\begin{prop}\label{prop-7}
		Under the assumptions of Theorem \ref{theo2} and \eqref{ass:density}, we have
		\begin{align}
			&\mu \sum_{j=1}^{3}c_j\mu^{\f{2j}{3}}\E_{j,3-j}^{com}(t)+\mu \sum_{j=1}^{3}c_j\mu^{\f{2j}{3}}\int_0^t\D_{j,3-j}^{com}(\tau)\,d\tau 
			\label{est-6}\\
			\lesssim&~
			\mu \sum_{j=1}^{3}c_j\mu^{\f{2j}{3}}\E_{j,3-j}^{com}(0)
			+\mu \sum_{j=1}^3c_j\mu^{\f{2(j-1)}{3}}\int_0^t\D_{j-1,3-(j-1)}^{com}(\tau)\,d\tau
            + \sum_{j=1}^3c_j\mu^{\f{2j}{3}}\int_0^t\D_{j,3-j}^{com,l}(\tau)\,d\tau
			\nonumber\\
			&+\f{1}{A}\mu \sum_{j=1}^3c_j\mu^{\f{2j}{3}}\int_0^t\D_{j,3-j}^{in}(\tau)\,d\tau
			+  \mu^{-1} |\E(t)|^{\frac12}\int_0^t  \D(\tau) \, d\tau,
			\nonumber
		\end{align}
		where $A$ is defined in \eqref{def-m2} and will be determined later.
	\end{prop}
	
	Recalling $\eqref{eq_n-d-w-2}_1$ and $\eqref{eq_n-d-w-2}_2$ once more, we obtain 
	\begin{align}
		\pa_{t}\hat{N}_{j+1,s-j}^k=&
        \f{j+1}{2}\f{\pa_{t}p}{p}\hat{N}_{j+1,s-j}^k
		-p^{\f12}\hat{D}_{j,s-j}^k+\hat{\FFF}_{1,j+1,s-j}^k,
		\label{eq_N-j+1}\\
		\pa_{t}\hat{D}_{j,s-j}^k=&-\nu p \hat{D}_{j,s-j}^k+\f{j+2}{2}\f{\pa_{t}p}{p}\hat{D}_{j,s-j}^k
		+\left(p^{\f12}+2\f{k^2}{p^{\f32}}\right)\hat{N}_{j+1,s-j}^k
		\label{eq_D-j+1}\\
		&-2\mu \f{k^2}{p}\hat{D}_{j,s-j}^k
		-2\f{k^2}{p}\hat{W}_{j,s-j}^k
		+\hat{\FFF}_{2,j+1,s-j}^k,
		\nonumber
	\end{align}
	where nonlinear terms are defined by
	\begin{align}
		\hat{\FFF}_{1,j+1,s-j}^k\tri& <k,\xi>^{s-j}p^{\f{j+1}{2}}\hat{\FFF}_1^k,
		\\
		\hat{\FFF}_{2,j+1,s-j}^k\tri& <k,\xi>^{s-j}p^{\f{j}{2}}\hat{\FFF}_2^k.
	\end{align}
	According to \eqref{eq_N-j+1} and \eqref{eq_D-j+1},  we have
	\begin{align}
		&\f12\f{d}{dt}\sum_{s=1}^{3}\sum_{j=1}^{s}c_j\mu^{\f{2j}{3}}\sum_{k\in\Z}\int_{\R}\left[|\hat{N}_{j+1,s-j}^k|^2+|\hat{D}_{j,s-j}^k|^2\right]\,d\xi
        +\sum_{s=1}^{3}\sum_{j=1}^{s}c_j\mu^{\f{2j}{3}}\sum_{k\in\Z}\int_{\R} \nu p |\hat{D}_{j,s-j}^k|^2\,d\xi
		\label{est-energy-22}\\
		=&\sum_{s=1}^{3}\sum_{j=1}^{s}c_j\mu^{\f{2j}{3}}\sum_{k\in\Z\backslash\{0\}}\int_{\R}\left\{\f{j+1}{2}\f{\pa_{t}p}{p} |\hat{N}_{j+1,s-j}^k|^2
		+\f{j+2}{2}\f{\pa_{t}p}{p}|\hat{D}_{j,s-j}^k|^2
		\right.
		\nonumber\\
		&\left.
		+2\f{k^2}{p^{\f32}}Re\left(\hat{N}_{j+1,s-j}^k\bar{\hat{D}}_{j,s-j}^k\right)
		-2\mu \f{k^2}{p}|\hat{D}_{j,s-j}^k|^2
		-2\f{k^2}{p}Re\left(\hat{W}_{j,s-j}^k\bar{\hat{D}}_{j,s-j}^k\right)
		\right\}\,d\xi
		\nonumber\\
		&+\sum_{s=1}^{3}\sum_{j=1}^{s}c_j\mu^{\f{2j}{3}}\sum_{k\in\Z}\int_{\R}\left\{Re\left(\hat{\FFF}_{1,j+1,s-j}^k\bar{\hat{N}}_{j+1,s-j}^k\right)+Re\left(\hat{\FFF}_{2,j+1,s-j}^k\bar{\hat{D}}_{j,s-j}^k\right)\right\}\,d\xi.\nonumber
	\end{align}
	In order to obtain the dissipation for $\hat{N}_{j+1,s-j}^k$, we compute that
	\begin{align}
		&-\f{d}{dt}\left[p^{-\f12}Re\left(\hat{N}_{j+1,s-j}^k\bar{\hat{D}}_{j,s-j}^k\right)\right]
		+\left(1+2\f{k^2}{p^2}\right)|\hat{N}_{j+1,s-j}^k|^2
		\label{est-energy-23}\\
		=&~|\hat{D}_{j,s-j}^k|^2+\f{1}{p^{\f12}}\left(-(j+1)\f{\pa_{t}p}{p}+2\mu \f{k^2}{p}+\nu p\right)Re\left(\hat{N}_{j+1,s-j}^k\bar{\hat{D}}_{j,s-j}^k\right)
		\nonumber\\
		&+2\f{k^2}{p^{\f32}}Re\left(\hat{N}_{j+1,s-j}^k\bar{\hat{W}}_{j,s-j}^k\right)
		-\f{1}{p^{\f12}}\left[Re\left(\hat{\FFF}_{1,j+1,s-j}^k\bar{\hat{D}}_{j,s-j}^k\right)+Re\left(\hat{\FFF}_{2,j+1,s-j}^k\bar{\hat{N}}_{j+1,s-j}^k\right)\right].
		\nonumber
	\end{align}
	Multiplying \eqref{est-energy-23} by $\delta_6\mu^{\f13}$, then summing over $k\in\Z$, $j$ from $1$ to $s$ and over $s$ from $1$ to $3$, we achieve that
	\begin{align}
		&-\delta_6\mu^{\f13}\f{d}{dt}\sum_{s=1}^{3}\sum_{j=1}^{s}c_j\mu^{\f{2j}{3}}\sum_{k\in\Z\backslash\{0\}}\int_{\R}\left[p^{-\f12}Re\left(\hat{N}_{j+1,s-j}^k\bar{\hat{D}}_{j,s-j}^k\right)\right]\,d\xi
		\label{est-energy-23-1}\\
		&+\delta_6\mu^{\f13}\sum_{s=1}^{3}\sum_{j=1}^{s}c_j\mu^{\f{2j}{3}}\sum_{k\in\Z\backslash\{0\}}\int_{\R}\left(1+2\f{k^2}{p^2}\right)|\hat{N}_{j+1,s-j}^k|^2\,d\xi
		\nonumber\\
		=&~\delta_6\mu^{\f13}\sum_{s=1}^{3}\sum_{j=1}^{s}c_j\mu^{\f{2j}{3}}\sum_{k\in\Z\backslash\{0\}}\int_{\R}\left\{|\hat{D}_{j,s-j}^k|^2+\f{1}{p^{\f12}}\left(-(j+1)\f{\pa_{t}p}{p}
		\right.\right.
		\nonumber\\
		&\left.+2\mu \f{k^2}{p}+\nu p\right)Re\left(\hat{N}_{j+1,s-j}^k\bar{\hat{D}}_{j,s-j}^k\right)
		+2\f{k^2}{p^{\f32}}Re\left(\hat{N}_{j+1,s-j}^k\bar{\hat{W}}_{j,s-j}^k\right)
		\nonumber\\
		&\left.-\f{1}{p^{\f12}}\left[Re\left(\hat{\FFF}_{1,j+1,s-j}^k\bar{\hat{D}}_{j,s-j}^k\right)+Re\left(\hat{\FFF}_{2,j+1,s-j}^k\bar{\hat{N}}_{j+1,s-j}^k\right)\right]\right\}\,d\xi.\nonumber
	\end{align}
	When $k=0$, we multiply \eqref{est-energy-23} by $\delta_6\mu$,  sum over $j$ from $1$ to $s-1$ and over $s$ from $2$ to $3$, and thereby obtain
	\begin{align}
		&-\delta_6\mu\f{d}{dt}\sum_{s=2}^{3}\sum_{j=1}^{s-1}c_j\mu^{\f{2j}{3}}\int_{\R}\left[p^{-\f12}Re\left(\hat{N}_{j+1,s-j}^0\bar{\hat{D}}_{j,s-j}^0\right)\right]\,d\xi
		+\delta_6\mu\sum_{s=2}^{3}\sum_{j=1}^{s-1}c_j\mu^{\f{2j}{3}}\int_{\R}|\hat{N}_{j+1,s-j}^0|^2\,d\xi
		\label{est-energy-23-2}\\
		=&~\delta_6\mu\sum_{s=2}^{3}\sum_{j=1}^{s-1}c_j\mu^{\f{2j}{3}}\sum_{k\in\Z\backslash\{0\}}\int_{\R}\left\{|\hat{D}_{j,s-j}^0|^2+\nu |\xi|Re\left(\hat{N}_{j+1,s-j}^0\bar{\hat{D}}_{j,s-j}^0\right)
		\right.
		\nonumber\\
		&\left.
		-|\xi|^{-1}\left[Re\left(\hat{\FFF}_{1,j+1,s-j}^0\bar{\hat{D}}_{j,s-j}^0\right)+Re\left(\hat{\FFF}_{2,j+1,s-j}^0\bar{\hat{N}}_{j,s-j}^0\right)\right]\right\}\,d\xi.\nonumber
	\end{align}
	We then define 
	\begin{align}
		\widetilde{\E}_8(t)\tri&\sum_{s=1}^{3}\sum_{j=1}^{s}c_j\mu^{\f{2j}{3}}\sum_{k\in\Z}\int_{\R}\left(|\hat{N}_{j+1,s-j}^k|^2+|\hat{D}_{j,s-j}^k|^2\right)\,d\xi
		\\
		&-2\delta_6\mu^{\f13}\sum_{s=1}^{3}\sum_{j=1}^{s}c_j\mu^{\f{2j}{3}}\sum_{k\in\Z\backslash\{0\}}\int_{\R}\left[p^{-\f12}Re\left(\hat{N}_{j+1,s-j}^k\bar{\hat{D}}_{j,s-j}^k\right)\right]\,d\xi
		\nonumber\\
		&-2\delta_6\mu\sum_{s=2}^{3}\sum_{j=1}^{s-1}c_j\mu^{\f{2j}{3}}\int_{\R}p^{-\f12}Re\left(\hat{N}_{j+1,s-j}^0\bar{\hat{D}}_{j,s-j}^0\right)\,d\xi,
		\nonumber\\
		\widetilde{\D}_8(t)\tri&\sum_{s=1}^{3}\sum_{j=1}^{s}c_j\mu^{\f{2j}{3}}\sum_{k\in\Z}\int_{\R}\left[\delta_6\mu^{\f13}\left(1+2\f{k^2}{p^2}\right)\right]|\hat{N}_{j+1,s-j}^k|^2\,d\xi
		\\
		&
		+\delta_6\mu\sum_{s=2}^{3}\sum_{j=1}^{s-1}c_j\mu^{\f{2j}{3}}\int_{\R}|\hat{N}_{j+1,s-j}^0|^2\,d\xi
		+\sum_{s=1}^{3}\sum_{j=1}^{s}c_j\mu^{\f{2j}{3}}\sum_{k\in\Z}\int_{\R}\nu p|\hat{D}_{j,s-j}^k|^2\,d\xi.
		\nonumber
	\end{align}
	A direct computation yields that
	\begin{align}\label{con-equi-6}
		\widetilde{\E}_8(t)\sim \sum_{j=1}^3c_j\mu^{\f{2j}{3}}\E_{j,s-j}^{com}(t),\quad \widetilde{\D}_8(t)\sim \sum_{j=1}^3c_j\mu^{\f{2j}{3}}\D_{j,s-j}^{com}(t).
	\end{align}
	Summing up  \eqref{est-energy-22}, \eqref{est-energy-23-1} and \eqref{est-energy-23-2}, we conclude  
	\begin{align}
		&\f12\f{d}{dt}\widetilde{\E}_8(t)+\widetilde{\D}_8(t)=\LL_4(t)+\sum_{i=1}^3M_i(t),
		\label{est-energy-24}
	\end{align}
	where the linear term $\LL_4(t)$ and the nonlinear terms $M_i (i=1,2,3)$ are defined by
	\begin{align}
		\LL_4(t)\tri&
		\sum_{s=1}^{3}\sum_{j=1}^{s}c_j\mu^{\f{2j}{3}}\sum_{k\in\Z\backslash\{0\}}\int_{\R}\left\{\f{j+1}{2}\f{\pa_{t}p}{p}|\hat{N}_{j+1,s-j}^k|^2+\left(\f{j+2}{2}\f{\pa_{t}p}{p}
		+\delta_6\mu^{\f13}-2\mu\f{k^2}{p}\right)|\hat{D}_{j,s-j}^k|^2
		\right.
		\label{def_L4}\\
		&+\left[2\f{k^2}{p^{\f32}}+\delta_6\mu^{\f13}\left(\nu p^{\f12}
		+2\mu\f{k^2}{p^{\f32}}\right)
		-(j+1)\delta_6\mu^{\f13}\f{\pa_{t}p}{p^{\f32}}
		\right]
		\nonumber\\
		&\left.\times Re\left(\hat{N}_{j+1,s-j}^k\bar{\hat{D}}_{j,s-j}^k\right)
		+2\delta_6\mu^{\f13}\f{k^2}{p^{\f32}}Re\left(\hat{N}_{j+1,s-j}^k\bar{\hat{W}}_{j,s-j}^k\right)
		-2\f{k^2}{p}Re\left(\hat{W}_{j,s-j}^k\bar{\hat{D}}_{j,s-j}^k\right)
		\right\}\,d\xi
		\nonumber\\
		&+\delta_6\mu\sum_{s=2}^{3}\sum_{j=1}^{s-1}c_j\mu^{\f{2j}{3}}\sum_{k\in\Z\backslash\{0\}}\int_{\R}\left[|\hat{D}_{j,s-j}^0|^2+\nu |\xi|Re\left(\hat{N}_{j+1,s-j}^0\bar{\hat{D}}_{j,s-j}^0\right)\right]\,d\xi,
		\nonumber\\
		M_1(t)\tri&\sum_{s=1}^{3}\sum_{j=1}^{s}c_j\mu^{\f{2j}{3}}\sum_{k\in\Z}\int_{\R}\left[Re\left(\hat{\FFF}_{1,j+1,s-j}^k\bar{\hat{N}}_{j+1,s-j}^k\right)
		+Re\left(\hat{\FFF}_{2,j+1,s-j}^k\bar{\hat{D}}_{j,s-j}^k\right)\right]\,d\xi,
		\nonumber\\
		M_2(t)\tri&-\delta_6\mu\sum_{s=2}^{3}\sum_{j=1}^{s-1}c_j\mu^{\f{2j}{3}}\int_{\R} |\xi|^{-1}\left[Re\left(\hat{\FFF}_{1,j+1,s-j}^0\bar{\hat{D}}_{j,s-j}^0\right)+Re\left(\hat{\FFF}_{2,j+1,s-j}^0\bar{\hat{N}}_{j+1,s-j}^0\right)\right]\,d\xi,
		\nonumber\\
		M_3(t)\tri&-\delta_6\mu^{\f13}\sum_{s=1}^{3}\sum_{j=1}^{s}c_j\mu^{\f{2j}{3}}\sum_{k\in\Z\backslash\{0\}}\int_{\R}\f{1}{p^{\f12}}\left[Re\left(\hat{\FFF}_{1,j+1,s-j}^k\bar{\hat{D}}_{j,s-j}^k\right)+Re\left(\hat{\FFF}_{2,j+1,s-j}^k\bar{\hat{N}}_{j+1,s-j}^k\right)\right]\,d\xi
		.\nonumber
	\end{align}
	\begin{lemm}
		It holds that
		\begin{align}\label{est-L4}
			\LL_4(t)\leq &~\f14\widetilde{\D}_8(t)+C\sum_{j=1}^3c_j\mu^{\f{2(j-1)}{3}}\D_{j-1,3-(j-1)}^{com}(t)
            \\
            &~
            +C\mu^{-1}\sum_{j=1}^3c_j\mu^{\f{2j}{3}}\D_{j,3-j}^{com,l}(t)
            +\f{C}{A}\sum_{j=1}^3\mu^{\f{2j}{3}}\D_{j,3-j}^{in}(t).
            \nonumber
		\end{align}
	\end{lemm}
	\begin{proof}
		We begin by treating all linear terms on the right-hand side of \eqref{def_L4} for $k\not=0$.
		The lift-up term involving the $j$-order bad derivative therein is then absorbed by the enhanced dissipation, which contains both $j$-order and $(j-1)$-order bad derivative. 
		Indeed, from \eqref{est-p-varphi} we deduce
		\begin{align}
			&\left|\f{j+1}{2}\f{\pa_{t}p}{p}|\hat{N}_{j+1,s-j}^k|^2
			+\f{j+2}{2}\f{\pa_{t}p}{p}|\hat{D}_{j,s-j}^k|^2\right|
			\label{est-L4-1}\\
			\lesssim&~ \f{|k|}{p^{\f12}} \left(|\hat{N}_{j+1,s-j}^k|^2+ |\hat{D}_{j,s-j}^k|^2\right)
			\nonumber\\
			\lesssim&~ \widetilde{\ep} \mu^{\f13}\left(|\hat{N}_{j+1,s-j}^k|^2+ |\hat{D}_{j,s-j}^k|^2\right) + C_{\widetilde{\ep}} \mu^{-\f13} \left(|\hat{N}_{j,s-(j-1)}^k|^2+ |\hat{D}_{j-1,s-(j-1)}^k|^2\right).
			\nonumber
		\end{align}
		Sine $p\geq1$ for $k\not=0$, we have
		\begin{align}
			\left|\big(\delta_6\mu^{\f13}-2\mu\f{k^2}{p}\big)\right|
			|{\hat{D}}_{j,s-j}^k|^2 
            \lesssim |p^{\f12}{\hat{\UU}}_{j,s-j}^k|^2,
			\label{est-L4-2}
		\end{align}
		\begin{align}
			&\left|\delta_6\mu^{\f13}\nu p^{\f12}Re\left(\hat{N}_{j+1,s-j}^k\bar{\hat{D}}_{j,s-j}^k\right)\right|
			\lesssim
			\delta_6\mu^{\f43}
			\left(|\hat{N}_{j+1,s-j}^k|^2+ p|\hat{D}_{j,s-j}^k|^2\right),
			\label{est-L4-3}
		\end{align}
		and
		\begin{align}
			&\left|2\delta_6\mu^{\f13}\f{k^2}{p^{\f32}}Re\left(\hat{N}_{j+1,s-j}^k\bar{\hat{W}}_{j,s-j}^k\right)
			-2\f{k^2}{p}Re\left(\hat{W}_{j,s-j}^k\bar{\hat{D}}_{j,s-j}^k\right)\right|
			\label{est-L4-4}\\
			\lesssim& ~
			\f{k^2}{p}|\hat{W}_{j,s-j}^k|^2
            +|\hat{N}_{j,s-(j-1)}^k|^2+ |\hat{D}_{j-1,s-(j-1)}^k|^2.
			\nonumber
		\end{align}
		Applying \eqref{est-p-varphi} once more, we deduce that
		\begin{align}	\label{est-L4-5}
			&\left|\left[2\f{k^2}{p^{\f32}}+\delta_6\mu^{\f13}\f{1}{p^{\f12}}\left(2\mu\f{k^2}{p}-(j+1)\f{\pa_{t}p}{p}\right)
			\right]
			Re\left(\hat{N}_{j+1,s-j}^k\bar{\hat{D}}_{j,s-j}^k\right)\right|
			\lesssim 
			|\hat{N}_{j,s-(j-1)}^k|^2+ |\hat{D}_{j-1,s-(j-1)}^k|^2.
		\end{align}
		For $k=0$, it is not difficult to check that
		\begin{align}
			\left|\delta_6\mu\nu |\xi|Re\left(\hat{N}_{j+1,s-j}^0\bar{\hat{D}}_{j,s-j}^0\right)\right|\lesssim \delta_6\mu^2\left(|\hat{N}_{j+1,s-j}^0|^2+|\xi|^2|\hat{D}_{j,s-j}^0|^2\right).
			\label{est-L4-6}
		\end{align}
		Inserting all the preceding estimates into \eqref{est-energy-24} and choosing $A$ (defined in \eqref{def-m2}) sufficiently large and			
		
		Inserting all the estimates above into \eqref{est-energy-24}, choosing $A$ defined in \eqref{def-m2} suitably large, $\widetilde{\ep}$ and $\delta_6$ suitably smalll, we immediately obtain \eqref{est-L4}.
	\end{proof}
	
	\begin{proof}[\textbf{Proof of Propostion \ref{prop-7}}]
		It follows from the same argument used in the proof of Proposition \ref{prop-4} that
		\begin{align}
			\sum_{i=1}^3\int_0^tM_{i}(\tau)\,d\tau
			\lesssim 	 \mu^{-1} |\E(t)|^{\frac12}\int_0^t  \D(\tau) \, d\tau.
			\label{est-M}	
		\end{align}
		Substituting \eqref{est-L4} and \eqref{est-M} into \eqref{est-energy-24}, then applying \eqref{con-equi-6}, we obtain \eqref{est-6} directly.
	\end{proof}

	\section{Proof of main result}
	
	In this section we prove Theorem \ref{theo2}.  The global existence theory for nonlinear PDEs is usually carried out in three steps: (1) construction of suitable approximate solutions;
	(2) uniform {\it a priori} energy estimates for these approximations; (3) passage to the limit.  Steps (1) and (3) are standard.  For brevity we therefore focus on the {\it a priori} energy estimates for smooth solutions of system \eqref{eq_n-d-w-1} within the $H^4$ framework.  Once Proposition \ref{prop-1} is established, the theorem follows immediately by a continuity argument; the remaining details are omitted.  We now proceed to the proof of Proposition \ref{prop-1}.

	\begin{proof}[\textbf{Proof of Proposition \ref{prop-1}}]
		
		Using the bootstrap argument, we first assume that the density bound \eqref{ass:density} holds.  Under this assumption we have already shown that \eqref{est-0+1}, \eqref{est-2+2'} and \eqref{est-3} remain valid for all time.  Multiplying \eqref{est-2+2'} by $\tilde{\delta}_1$ and \eqref{est-3} by $\tilde{\delta}_2$, adding the resulting equality to \eqref{est-0+1} together, then taking $A$ sufficiently large and $\tilde{\delta}_1$, $\tilde{\delta}_2$ sufficiently small, we obtain
		\begin{align}
			&\E_{0,3}^{com,l}(t)
			+\tilde{\delta}_1\E_{0,3}^{in}(t)
			+\tilde{\delta}_2\mu \E_{0,3}^{com}(t)
			+\int_0^t \left(\D_{0,3}^{com,l}
			+\tilde{\delta}_1\D_{0,3}^{in}
			+\tilde{\delta}_2\mu \D_{0,3}^{com}\right)(\tau)\,d\tau
			\label{est-E-1}\\
			\lesssim&~
			\E_{0,3}^{com,l}(0)
			+\tilde{\delta}_1\E_{0,3}^{in}(0)
			+\tilde{\delta}_2\mu \E_{0,3}^{com}(0)
			+	 \mu^{-1} |\E(t)|^{\frac12}\int_0^t  \D(\tau) \, d\tau. \nonumber
		\end{align}
		Notice that $\E_{1,2}^{com,l}(t)$ and $\D_{1,2}^{com,l}(t)$ can not be controlled by the terms on the left-hand side of \eqref{est-E-1}. Thus we must bound the cross terms involving both the good and bad derivatives of the density and velocity. 
		Multiplying \eqref{est-5} by $\tilde{\delta}_1$ and \eqref{est-6} by $\tilde{\delta}_1$ and $\tilde{\delta}_2$,  then adding the resulting equality to \eqref{est-4}, and then choosing  $A$ sufficiently large and $\tilde{\delta}_1,\tilde{\delta}_2$ sufficiently small, we obtain
		\begin{align}
			&\sum_{j=1}^3c_j\mu^{\f{2j}{3}}\left(\E_{j,3-j}^{com,l}(t)
			+\tilde{\delta}_1\E_{j,3-j}^{in}(t)
			+\tilde{\delta}_2\mu \E_{j,3-j}^{com}(t)\right)
			\label{est-E-2}\\
			&
			\quad+\sum_{j=1}^3 c_j\mu^{\f{2j}{3}}\int_0^t \left(\D_{j,3-j}^{com,l}
			+\tilde{\delta}_1\D_{j,3-j}^{in}
			+\tilde{\delta}_2\mu \D_{j,3-j}^{com}\right)(\tau)\,d\tau
			\nonumber\\
			\lesssim
			&
			\sum_{j=1}^3c_j\mu^{\f{2j}{3}}\left(\E_{j,3-j}^{com,l}(0)
			+\tilde{\delta}_1\E_{j,3-j}^{in}(0)
			+\tilde{\delta}_2\mu \E_{j,3-j}^{com}(0)\right)
			\nonumber\\
			&\quad
			+
			\sum_{j=1}^3c_j\mu^{\f{2(j-1)}{3}}\int_0^t \left(\D_{j-1,3-(j-1)}^{com,l}
			+\tilde{\delta}_1\D_{j-1,3-(j-1)}^{in}
			+\tilde{\delta}_2\mu \D_{j-1,3-(j-1)}^{com}\right)(\tau)\,d\tau\nonumber\\
			& \quad
			+	 \mu^{-1} |\E(t)|^{\frac12}\int_0^t  \D(\tau) \, d\tau.
			\nonumber
		\end{align}
		Multiplying \eqref{est-E-1} by $c_0$, and adding the result inequality to \eqref{est-E-2}, then choosing suitable $c_j$ such that $c_{j}\leq C_1 c_{j-1}$ holds for any $j=1,\cdots,3$, finally using the \textit{a priori} assumption \eqref{priori assumption}, we then deduce that
		\begin{align}
			\sup_{\tau\in[0,t]}\E(\tau)+\int_0^t \D(\tau)\,d\tau \lesssim \E(0)+\mu^{-1}|\E(t)|^{\f12}\int_0^t \D(\tau)\,d\tau\lesssim \E(0)+\ep_0^3\mu^2,
		\end{align}
		where $\E(t)$ and $\D(t)$ are defined in \eqref{def_E} and \eqref{def_D}, respectively. Then by Sobolev's inequality, we get
		\begin{align*} 
			\f34\leq N+1\leq \f32,\quad \textrm{for all}\quad t\in [0,\infty). 
		\end{align*}
		
		This finishes the proof of Proposition \ref{prop-1}.
	\end{proof}

	\medskip

	\section*{Acknowledgments}
	M. Li is supported by Postdoctoral Fellowship Program of CPSF under Grant No. GZB20240024 and China Postdoctoral Science Foundation under Grant No. 2024M760057 and No. 2025T180840. C. Wang is partially supported by NSF of China under Grant No. 12471189. Z. Zhang is partially supported by NSF of China under Grant No. 12288101.

	\bigskip 
	
	\section*{Declaration of competing interest}
	
	The authors declare that they have no conflict of interest.
	
	\bigskip

	\section*{Data availability}
	
	Data sharing not applicable to this article as no datasets were generated or analyzed during the current study.
	

\end{document}